\documentclass[twoside,11pt, letter]{article}

\usepackage{fixltx2e} 

\usepackage{cmap}

\usepackage[T1]{fontenc}
\usepackage{lmodern}
\usepackage[utf8]{inputenc}

\usepackage{verbatim}

\usepackage{booktabs}       

\usepackage{nicefrac}       

\usepackage{microtype}      

\usepackage{graphicx} 

\usepackage[colorlinks]{hyperref}
\usepackage{url}

\usepackage{geometry} 
\usepackage[toc,page]{appendix}
\usepackage{mathrsfs}
\usepackage{amsmath, amssymb, amsthm}
\usepackage[linesnumbered, ruled]{algorithm2e}
\usepackage{tikz}
\usepackage{float}

\usepackage{bm}

\usepackage[font=small,labelfont=bf,labelsep=period]{caption}

\newcommand{\alg}{\mathcal{A}}

\DeclareMathOperator*{\argmin}{arg\,min}

\newcommand{\Aut}{\ensuremath{\operatorname{Aut}}}

\newcommand{\betamax}{\overline{\beta}}
\newcommand{\betamin}{\underline{\beta}}

\newcommand{\Cbar}{\ensuremath{\overline{C}}}
\newcommand{\completespace}{\mathfrak{C}}

\newcommand{\E}{\ensuremath{\mathbb{E}}}
\newcommand{\edgeset}{\mathcal{E}}
\newcommand{\emptytext}{\text{empty}}
\newcommand{\energy}{\mathcal{H}}
\newcommand{\expect}{\ensuremath{\mathbb{E}}}

\newcommand{\field}{\mathcal{F}}

\newcommand{\graphs}{\mathfrak{G}}
\newcommand{\graphset}{\mathcal{G}}
\newcommand{\graphspace}{\mathfrak{G}}

\newcommand{\hiv}{\text{HIV}}

\newcommand{\ind}{\ensuremath{\mathbf{1}}}
\newcommand{\infected}{\mathcal{I}}
\newcommand{\infspace}{\mathbb{I}}
\newcommand{\invertex}{\text{in}}
\newcommand{\ising}{\text{Ising}}

\newcommand{\kldiv}{\text{KL}}

\newcommand{\measure}{\mu}
\newcommand{\measuremap}{\mathcal{M}}
\newcommand{\measurespace}{\mathfrak{M}}

\newcommand{\neighbors}{N}

\newcommand{\paramset}{\mathfrak{B}}
\newcommand{\pathset}{\mathfrak{P}}
\newcommand{\pathvar}{\mathcal{P}}

\newcommand{\prob}{\ensuremath{\mathbb{P}}}

\newcommand{\reals}{\mathbb{R}}

\newcommand{\scriptA}{\ensuremath{\mathcal{A}}}
\newcommand{\scriptC}{\ensuremath{\mathcal{C}}}
\newcommand{\scriptO}{\ensuremath{\mathcal{O}}}

\newcommand{\sims}{\text{sims}}

\newcommand{\ssm}{\text{SSM}}

\newcommand{\support}{\text{supp}}

\newcommand{\tvdist}{\text{TV}}
\newcommand{\typeone}{(\text{TI})}

\newcommand{\uncensored}{\mathcal{U}}

\newcommand{\vect}[1]{\bm{#1}}
\newcommand{\vertexset}{\mathcal{V}}

\newcommand{\Wbar}{\ensuremath{\overline{W}}}
\newcommand{\wneighbors}{W}

\renewcommand{\xspace}{\mathcal{X}}

\newtheorem{theorem}{Theorem}
\newtheorem{lemma}{Lemma}
\newtheorem{proposition}{Proposition}
\newtheorem{corollary}{Corollary}
\newtheorem{definition}{Definition}

\newtheorem{remark}{Remark}

\newtheorem{theorem*}[theorem]{Theorem}   
\newtheorem{lemma*}{Lemma} 
\newtheorem{corollary*}{Corollary} 
\newtheorem{remark*}{Remark}
\newtheorem{example*}{Example}
\newtheorem{definition*}{Definition}
\newtheorem{proposition*}{Proposition}

\newfloat{newtable}{thp}{lop}
\floatname{newtable}{Table}

\newcommand{\GG}[1]{}
\renewcommand{\arraystretch}{1.2} 
\newcommand{\ra}[1]{\renewcommand{\arraystretch}{#1}}

\title{Permutation Tests for Infection Graphs}

\author{Justin Khim\thanks{Machine Learning Department, Carnegie Mellon University, Pittsburgh, PA 15213.}
\and 
Po-Ling Loh\thanks{Department of Statistics,
       University of Wisconsin,
       Madison, WI 53706 and 
       Department of Statistics,
       Columbia University, 
       New York, NY 10027.}}

\date{December 16, 2019}

\begin{document}
\maketitle

\begin{abstract}
We formulate and analyze a novel hypothesis testing problem for inferring the edge structure of an infection graph. In our model, a disease spreads over a network via contagion or random infection, where the times between successive contagion events are independent exponential random variables with unknown rate parameters. A subset of nodes is also censored uniformly at random. Given the observed infection statuses of nodes in the network, the goal is to determine the underlying graph. We present a procedure based on permutation testing, and we derive sufficient conditions for the validity of our test in terms of automorphism groups of the graphs corresponding to the null and alternative hypotheses. Our test is easy to compute and does not involve estimating unknown parameters governing the process. We also derive risk bounds for our permutation test in a variety of settings, and relate our test statistic to approximate likelihood ratio testing and maximin tests. For graphs not satisfying the necessary symmetries, we provide an additional method for testing the significance of the graph structure, albeit at a higher computational cost. We conclude with an application to real data from an HIV infection network.
\end{abstract}
\section{Introduction}

Information, diseases, and the adoption of certain behaviors may spread according to a network of relationships connecting susceptible individuals~\cite{AndMay92, christakis2007, Jac08, Mor93, New02}. Natural questions to address include (a) predicting the pathway or scope of a disease; (b) inferring the source; and (c) identifying optimal interventions to slow the spread of the epidemic. The answers to these questions may vary depending on the stochastic mechanism governing the spread of disease between individuals, and/or the rate of recovery.

Algorithms for addressing such questions often assume knowledge of the underlying graph structure~\cite{borgs, brautbar, bubeck2017}. For instance, in the influence maximization problem, the goal is to identify an optimal set of nodes to initially infect in order to propagate a certain behavior as widely as possible \cite{DomRic01, kempe2003}. Due to submodularity of the influence function, one may obtain a constant factor approximation to the influence-maximizing subset using a simple greedy algorithm. However, both the connectivity of the network and edge transmission parameters must be known in order to successfully execute the algorithm \cite{Chen2010a}. Similarly, algorithms for network immunization, which aim to eliminate an epidemic by performing targeted interventions, assume knowledge of the graph \cite{AlbEtal00, CohEtal03, PasVes02}. In real-world applications, however, prior knowledge of the edge structure of the underlying network or parameters of the infection spreading mechanism may be unavailable.

Accordingly, the focus of our paper is the inference problem of identifying the underlying network based on observed infection data. One approach involves employing tools from graphical model estimation, since the graphical model corresponding to joint vectors of infection times coincides with the unknown network~\cite{gomez2016, netrapalli2012}. However, these methods critically leverage the availability of time-stamped data and observations of multiple (i.i.d.) infection processes spreading over the same graph. 
Another line of work concerns reconstructing an infection graph based on the order in which nodes are infected over multiple infection processes, and provides bounds on the number of distinct observations required to recover the edge structure of the graph. These methods are attractive in that they do not assume a particular stochastic spreading model, and are even guaranteed to reconstruct the true network when the observations are chosen adversarially~\cite{angluin2015, huang2017}. On the other hand, data from multiple infections are still assumed to be available.

In contrast to these papers, we are interested in studying scenarios where infection data are available for a \emph{single} snapshot of a single epidemic outbreak. For instance, if the goal is to perform optimal interventions on a network of individuals during an outbreak such as Ebola \cite{DudEtal17}, data may only be available about the infection status of individuals at a given point in time. Although historic data may have been collected concerning the spread of other epidemics on the same network, there is no guarantee that the other diseases will spread according to the same mechanisms, and the network may have changed over time. Instead of employing the aforementioned graph estimation procedures based on observation vectors, we cast the problem in the form of a hypothesis test: Given two candidate graphs, our goal is to identify the graph on which the infection has propagated. This type of problem was previously studied by \cite{milling2015}, who proposed inference procedures for testing an empty graph versus graphs satisfying ``speed and spread'' conditions, and for testing two graphs against each other. However, we substantially generalize their approach, removing a restrictive ``independent neighborhoods'' assumption and allowing for exogenous sources of infection that are not captured by edgewise contagion \cite{myers2012}. A graph testing problem of a somewhat similar flavor may also be found in \cite{bubeck2016}, but their goal is to identify the model of network formation for a random graph, rather than determine the network underlying an epidemic outbreak.

The crux of our approach lies in permutation testing, a notion which dates back to \cite{Fis35}. Permutation tests are classically applied when random variables are exchangeable under the null hypothesis. The test statistic is then computed with respect to random permutations of the observation vector, and the observed statistic is compared to quantiles of the simulated distribution \cite{Goo13}. We adapt this technique in a novel manner to the graph testing problem. The procedure is the same as in classical permutation testing, where we recompute a certain statistic on a set of randomly chosen permutations of the infection ``vector'' recording the observed statuses of the nodes in the graph. Based on quantiles of the empirical histogram, we calculate a rejection rule for the test.

The key idea is that if the null hypothesis corresponds to an empty or complete graph and edge transmission rates are homogeneous, the components of the infection vector are certainly exchangeable. This is an important setting in its own right, since it allows us to determine whether a network structure describes an infection better than random contagion. As derived in our paper, however, the permutation test succeeds more generally under appropriate assumptions incorporating symmetries of both the null graph and alternative graph. We develop a sufficient condition in terms of automorphism groups of the two graphs, which are the subsets of node permutations that map edges to edges in the corresponding graphs.

A scientific motivation for graph hypothesis testing is a case when a practitioner wants to decide whether a disease is spreading according to a fixed hypothesized graph structure (or a particular graph topology) versus completely random transmissions, which would correspond to an empty graph. Another plausible scenario is when a scientist wants to test a long-standing hypothesis that a genetic network follows a particular connectivity pattern, versus a proposed alternative involving the addition or removal of certain edges. In a third setting, one might encounter two distinct graphs representing connections between individuals on different social network platforms, and hypothesize that information diffusion is governed by one graph instead of the other.

Our work significantly broadens the scope of existing approaches, as follows:
(i) we model information spreads more realistically by introducing random spreading external to the network; (ii) we drop the restrictive ``independent neighborhoods'' assumption, which masks any knowledge of the identity of specific individuals in the hypothesized graphs; and (iii) we provide novel algorithms for the graph testing problem in these settings.
Our permutation test is attractive in that it is easy to compute and bypasses the need to estimate any model parameters involved in the spreading process. This also leads to a relatively simple analysis of risk bounds for concrete graph topologies, as developed in our examples.  Notably, the numerical results illustrate settings in which our algorithms succeed, whereas previously existing tests do not lead to meaningful conclusions.

For general graphs, i.e., cases where these symmetries do not exist, we provide a computationally-intensive method for hypothesis testing based on a sufficiently fine discretization of the null hypothesis space using a simple data-processing inequality argument. Importantly, we can test slight differences in edge structure between graphs, which is not possible via permutation tests. This also allows us to a build confidence set for the underlying graph structure, but it is computationally feasible only for small graphs.

The remainder of the paper is organized as follows: In Section~\ref{sec: setup}, we describe the infection spreading model and invariant statistics considered in our theory. In Section~\ref{sec: main results}, we outline our graph testing procedure and state the corresponding theoretical guarantees. Section~\ref{sec: examples} discusses some illustrative special cases and provides explicit risk bounds. In Section~\ref{secCloserLook}, we analyze our test statistics more closely in relation to likelihood ratio tests and Hunt-Stein theory for maximin tests. In Section~\ref{sec:general_graphs}, we discuss our method for testing graphs based on discretization. Section~\ref{secHIV} describes an application of our methods to an HIV infection graph. Section~\ref{sec: discuss} concludes our discussion.
Additional theory, simulations, and proofs are provided in the Appendices.


\section{Problem setup}
\label{sec: setup}

We begin by establishing the notation we will use for infections and graph-based hypothesis testing. We then describe two random infection models for which our main results are applicable, followed by terminology relevant to our invariant test statistics.

\subsection{Infection notation}
Let \(\graphset_{0} = (\vertexset, \edgeset_{0})\) and \(\graphset_{1} = (\vertexset, \edgeset_{1})\) denote two undirected graphs defined over a common set of vertices \(\vertexset = \{1, \ldots , n\}.\)
A random vector of infection statuses \(\infected := \{\infected_{v}\}_{v \in \vertexset}\) consists of entries
\begin{equation*}
\infected_{v}
=
\begin{cases}
1, & \text{if \(v\) is infected,} \\ 
0, & \text{if \(v\) is uninfected,} \\ 
\star, & \text{if \(v\) is censored.}
\end{cases}
\label{eqn: statevec def}
\end{equation*}

Let \(\infected^{1}\) denote the set of infected vertices, and let \(\infected^{0}\) and \(\infected^{\star}\) be defined analogously. We then define the space of possible infection status vectors involving exactly \(k\) infected vertices, $c$ censored vertices, and $n-k-c$ uninfected vertices, as follows:
\begin{equation*}
\infspace_{k, c}
= 
\left\{\infected \in \{0, 1, \star\}^{n}: 
|\infected^{1}| = k, \;  
|\infected^{\star}| = c \right\}.
\label{eqn: infection space}
\end{equation*}

\subsection{Infection models}
\label{SecInfectModels}

Next, we introduce two stochastic models for generating random infection vectors on a graph $\graphset$.
The first model is motivated by the infection literature and borrows elements from first-passage percolation. The second is a conditional Ising model. Although the Ising model is not designed for modeling disease propagation, the two models behave similarly for small values of the spreading parameter $\beta$.
After describing the method for generating infection vectors in both models, we explain the random censoring mechanism. We then briefly discuss the parameter spaces involved in our hypothesis test.

\subsubsection{Stochastic spreading model}

Our main infection model includes the following components:
spreading by contagion along edges of the network, and spreading via factors external to the network. The rates of spreading are governed by nonnegative parameters $\lambda, \beta \in \reals_{+}$.
We assume that an infection spreads on $\graphset$, beginning at time 0, as follows:
\begin{itemize}
\item[(i)] For each vertex $v$, generate an independent random variable $T_v \sim \text{Exp}(\lambda)$.
\item[(ii)] For each edge $(u,v) \in \edgeset$, generate an independent random variable
$T_{uv} \sim \text{Exp}(\beta)$.
\item[(iii)] For each vertex \(v\), define the infection time $t_v := \min_{u \in N(v)} \{t_u + T_{uv}\} \wedge T_v$, where $N(v)$ is the set of neighbors of $v$.
\end{itemize} 
In other words, each vertex contracts the disease via random infection according to an exponential random variable with rate $\lambda$, and contracts the disease from an infected neighbor at rate $\beta$; in particular, when $\graphset$ is an empty graph, each successively infected node is chosen uniformly at random. The presence of the parameter $\lambda$ has the interpretation that some factors involved in spreading the disease may not be fully accounted for in the hypothesized graphs. We define \(t_{k}\) to be the time at which the \(k^\text{th}\) reporting node becomes infected, and we suppose we observe the state of the graph at some time \(t\) such that \(t_{k} \leq t \leq t_{k + 1}\). Note that this differs slightly from the setting of the papers \cite{kesten1993, milling2015, shah2011}, where the time \(t\) at which the node infection statuses are observed is assumed to be a known constant.

In our analysis, we frequently use the notion of \emph{paths}.
A path \(P = (P_{1}, \ldots, P_{k})\) is an ordered set of vertices, where vertex \(P_{i}\) is the \(i^{\text{th}}\) vertex to be infected.
A given infection \(J\) in \(\infspace_{k, 0}\) corresponds to \(k!\) possible paths, which we collect into a set \(\pathset(J)\).
We also refer to the path random vector as \(\pathvar\).
Thus, we have the relation
\[
\prob(\infected = J) 
=
\sum_{P \in \pathset(J)} \prob(\pathvar = P),
\]
which we will use repeatedly in our derivations.


Finally, note that without loss of generality, we can set $\lambda = 1$ by rescaling $\beta$: Suppose \(k\) vertices have been infected in running the process until time \(t\), and we want to determine the \((k + 1)^{\text{th}}\) infected vertex.
Since all of the \(T_{v}\)'s and \(T_{uv}\)'s are independent exponential random variables, the probability that an uninfected vertex $v$ is infected next is
\[
\prob\left(v \text{ is infected next}\right)
=
\frac{\lambda + \beta N_{t, v}}{\lambda(n - k) + \beta N_{t}}
=
\frac{1 + (\beta/\lambda) N_{t, v}}{(n - k) + (\beta/\lambda) N_{t}},
\]
where \(N_{t}\) is the number of edges with one infected vertex at time $t$, and \(N_{t, v}\) is the number of infected neighbors of \(v\). This is equivalent to the parametrization $(\lambda', \beta') = (1, \beta/\lambda)$.

\subsubsection{Conditional Ising model}

We now describe the second random infection model we will consider, which can also be parametrized by $\beta \in \reals^+$. Let the \emph{energy} of an infection \(J\) in \(\infspace_{k, 0}\) be defined by
\[
\energy(J)
:=
-\sum_{(u, v) \in \edgeset} J_{u} J_{v}.
\]
We define the uncensored conditional 01-Ising model as
\begin{align}
\label{EqnProbIsing}
\prob_{\ising}^{u}\left(\infected = J\right) 
:=
\frac{\exp(-\beta \energy(J))}{\sum_{J' \in \infspace_{k, 0}} \exp(-\beta \energy(J'))}
=
\frac{1}{Z_{\ising}(\beta)} \exp(-\beta H(J)). 
\end{align}
Note that we use the term ``conditional'' because \(k\) of the vertices are known to be in state \(1\). Furthermore, we use ``\(01\)'' to denote the fact that the states are \(0\) and \(1\).
We use the superscript \(u\) to emphasize the fact that the infection is uncensored.

As we shall see, infections generated by the conditional \(01\)-Ising model share properties with the stochastic spreading model. However, note that unlike the stochastic spreading model, the conditional Ising model does not correspond to a simple procedure for generating successive infection events as a disease propagates over $\graphset$. We discuss this model because it helps highlight the properties of the random infection distribution that allow our permutation test to succeed.


\subsubsection{Censoring}

We assume that censored nodes are selected uniformly at random among the nodes in the graph. Our hypothesis tests are designed to be valid, conditional on the fact that $c$ nodes are censored. In this subsection, we define some additional notation to express the probabilities of observing particular infection vectors, conditioned on the fact that $c$ nodes are censored.

Given an infection vector \(J\) that may contain censored components, we define the set \(\uncensored(J)\) to consist of all infection vectors \(J'\) such that \(J'\) has no censored vertices and \(J_{v} = J'_{v}\) for each uncensored vertex \(v\). Thus, $J'$ can have between $k$ and $k+c$ infected vertices. We then define the measure
\begin{equation}
\measure(J; \beta) 
=
\sum_{J' \in \uncensored(J)}
\frac{1}{\binom{n}{c}} \prob^{u}\left(\infected = J'\right),
\label{eqnMeasure}
\end{equation}
where \(\prob^{u}(\infected = J')\) denotes the probability of an uncensored infection from either of the two models described above. In other words, $\measure$ simply computes the probability of observing a vector $J$ after randomly censoring $c$ nodes.
If we define the normalizing constant
\begin{equation}
Z(\beta)
:=
\sum_{J \in \infspace_{k, c}}
\measure(J; \beta) 
=
\sum_{J \in \infspace_{k, c}}
\sum_{J' \in \uncensored(J)}
\frac{1}{\binom{n}{c}} \prob^{u}(\infected = J'),
\label{eqnPartitionFunction}
\end{equation}
which provides the probability of observing any vector with $k$ infected and $c$ censored nodes,
we see that
\begin{equation}
\prob\left(\infected = J\right)
=
\frac{1}{Z(\beta)} \measure(J; \beta)
\label{eqnProbWithCensoring}
\end{equation}
is the probability of observing infection vector $J$, conditioned on $c$ nodes having been censored.
Note that in the case when \(\graphset\) is the empty graph, the values of $\measure(J; \beta)$ are the same for all values of $J$, so \(\prob\left(\infected\right) = 1 / |\infspace_{k,c}| = 1 / \binom{n}{k \;\; c}\), i.e., the inverse of the multinomial coefficient for choosing \(k\) and \(c\) items from a set of \(n\).
Finally, now that we have a full model, a pictorial description of the stochastic spreading process is provided in Figure~\ref{fig: infections}.

\begin{figure}[!htb]
\minipage{0.32\textwidth}
  \includegraphics[width=\linewidth]{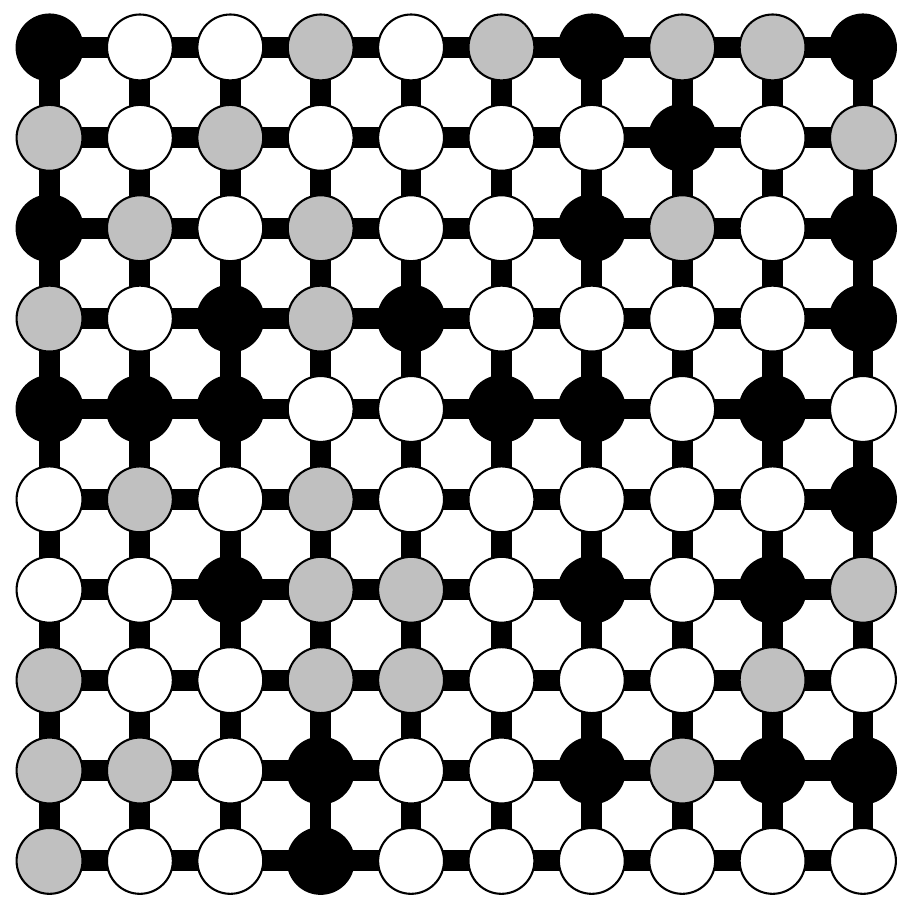}
\endminipage\hfill%
\minipage{0.32\textwidth}
  \includegraphics[width=\linewidth]{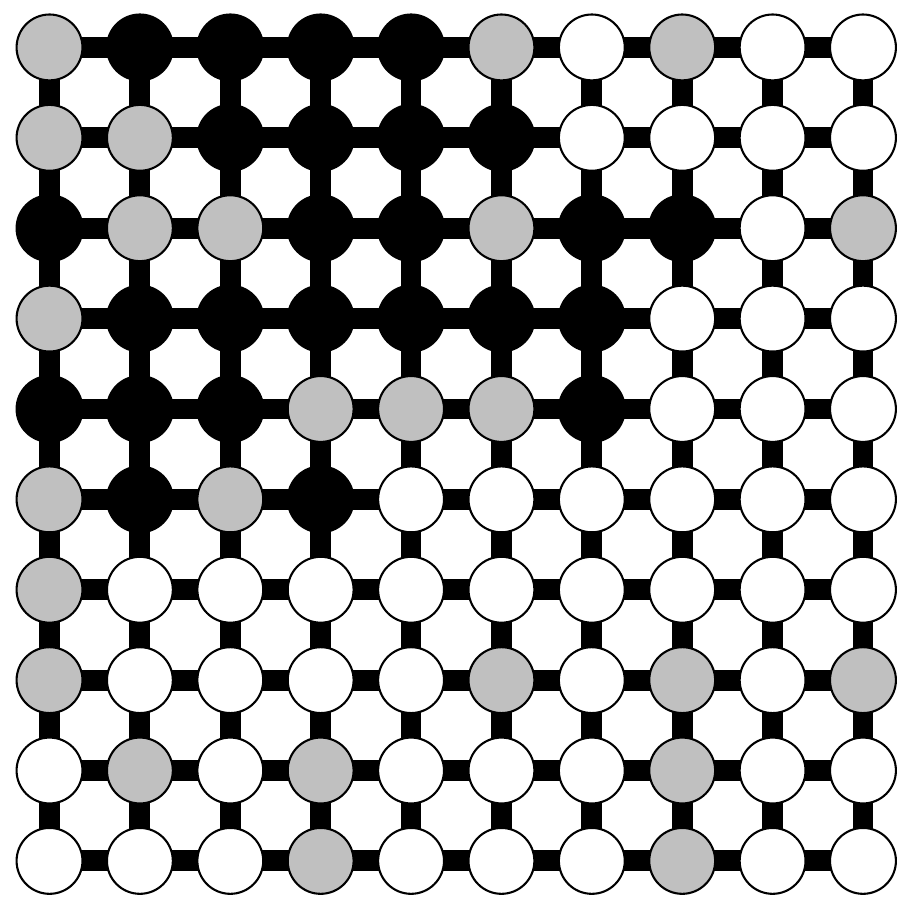}
\endminipage\hfill
\minipage{0.32\textwidth}%
  \includegraphics[width=\linewidth]{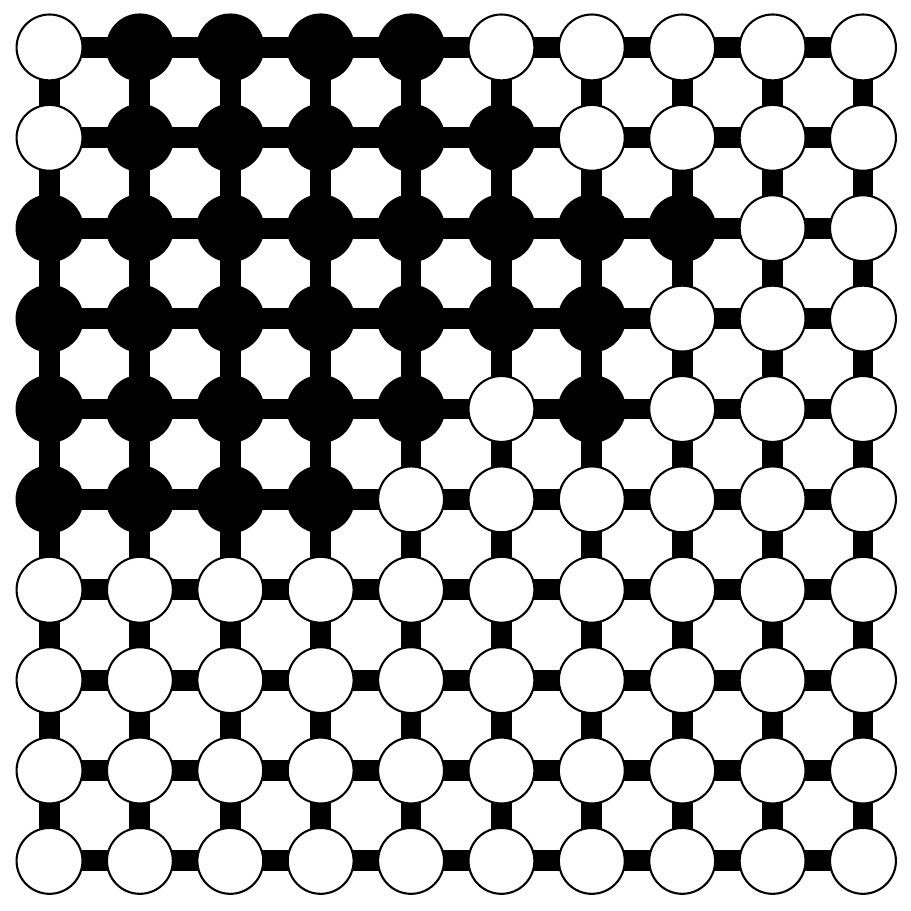}
\endminipage
\caption{Realizations of infection processes on the \(10 \times 10\) grid.
Uninfected vertices are white, censored vertices are gray, and infected vertices are black.
The first panel corresponds to \(\beta = 0\), so the infection is equivalent to uniformly random spreading.
The second and third panels depict the same infection with a large \(\beta\), with and without censoring.}
\label{fig: infections}
\end{figure}


\subsubsection{Parameter space and hypothesis tests}

We now describe the parameter space for our hypothesis tests in more detail. Suppose the null hypothesis corresponds to a spread over the graph $\graphset_0$ and the alternative hypothesis corresponds to a spread over the graph $\graphset_1$.
Let \(\paramset_{0}\) and \(\paramset_{1}\) be subsets of $[0, \infty)$, which are the sets of possible $\beta$ values in our null and alternative hypothesis classes.
Stated differently, each graph-parameter pair \((\graphset, \beta)\) parametrizes a probability measure \(\prob_{\beta}\), which can be thought of as an element of the probability simplex with entries indexed by the infections of \(\infspace_{k, c}\), as defined by equation~\eqref{eqnProbWithCensoring}.
Note that we sometimes write our parameters as \(\theta = (\graphset, \beta)\), and the resulting parameter space is then \(\Theta\). 

In the case of the stochastic spreading model (SSM), we thus use the shorthand
\begin{align*}
&\begin{aligned}
H_{0}
&: \infected \sim \ssm(\graphset_{0}, \beta_{0}) \text{ for } \beta_{0} \in \paramset_{0}, \; \graphset_{0} \in \graphspace_{0}, \\
H_{1}
&: \infected \sim \ssm(\graphset_{1}, \beta_{1}) \text{ for } \beta_{1} \in \paramset_{1}, \; \graphset_{1} \in \graphspace_{1}.
\end{aligned}
\end{align*}
Our core theory on permutation testing will focus on ``simple'' hypothesis tests, where \(\graphspace_{0}\) and \(\graphspace_{1}\) are singleton sets, and $\paramset_0 = \paramset_1 = \{\beta\}$, for some $\beta \neq 0$.

In general, a reasonable choice of parameter sets for distinguishing two graph structures $\graphspace_0 = \{\graphset_0\}$ and $\graphspace_1 = \{\graphset_1\}$ might be
\(\paramset_{0} = [0, \infty]\) and \(\paramset_{1} = (0, \infty]\).
However, in order to provably obtain non-trivial power, one would need to designate an indifference region, e.g., \(\paramset_{1} = [c, \infty]\) for some \(c > 0\).

A critical function \(\psi: \infspace_{k, c} \rightarrow \{0, 1\}\) may be used to indicate the result of a test.
We will be interested in bounding the risk, which is the sum of Type I and Type II errors:
\begin{equation*}
R_{k, c}(\psi; \beta_{0}, \beta_{1})
=
\prob_{0, \beta_{0}}\left(\psi(\infected) = 1\right) + \prob_{1, \beta_{1}}\left(\psi(\infected) = 0\right),
\label{eqn: risk}
\end{equation*}
defined for $\beta_0 \in \paramset_0$ and $\beta_1 \in \paramset_1$, where $\prob_{0, \beta_0}$ and $\prob_{1, \beta_1}$ are the measures corresponding to $(\graphset_0, \beta_0)$ and $(\graphset_1, \beta_1)$, respectively.


\subsection{Permutation-invariant statistics}
\label{secInvariant}

A natural statistic to consider for the purpose of graph testing is the likelihood ratio. 
For the stochastic spreading model, the likelihood ratio is often difficult to compute and depends on \(\beta\) in a nontrivial manner, making the theoretical derivations somewhat challenging. Our main focus will be on a class of statistics that are invariant under a group of permutations, which allow us to perform permutation testing based on symmetries in the graph sets. We first introduce some terminology regarding permutations and group actions, and then introduce a class of invariant statistics that will be central to our analysis.


\subsubsection{Permutations and group actions}

Recall that a \emph{graph automorphism}
\(\graphset = (\vertexset, \edgeset)\) is an element $\phi$ of the permutation group $S_n$ such that \((u, v) \in \edgeset\) if and only if \((\phi(u), \phi(v)) \in \edgeset\). For simple hypotheses, we denote the automorphism groups of $\graphset_0$ and $\graphset_1$ by \(\Pi_{0} = \Aut(\graphset_0)\) and \(\Pi_{1} = \Aut(\graphset_1)\), respectively.

We also need to define the \emph{action} of a permutation on vertices, graphs, and infections.
The action of a permutation \(\pi\) on a vertex \(u\) is simply the image \(\pi(u)\).
This is easily extended to tuples and subsets of vertices by applying \(\pi\) to the underlying vertices.
A specific example is the action on edges of the graph:
\[
\pi \edgeset
=
\left\{(\pi(u), \pi(v)): (u, v) \in \edgeset\right\}.
\]
The action of \(\pi\) on a graph \(\graphset = (\vertexset, \edgeset)\) is then defined to be
\[
\pi \graphset
:=
(\pi \vertexset, \pi \edgeset)
=
(\vertexset, \pi \edgeset).
\]
Another natural extension is to define the action of a set of permutations on a set of graphs:
\[
\Pi \graphspace
=
\left\{\pi \graphset: \pi \in \Pi \text{ and } \graphset \in \graphspace\right\}.
\]
If \(\graphspace_{i} = S_{n}\{\graphset_{i}\}\), we say that hypothesis \(i\) corresponds to a hypothesis of a particular graph \emph{topology}, since all node labelings are included in the set. 
We also define the action \(\Pi \Theta_{i} = \Pi \graphspace_{i} \times A_{i}\).
Finally, we define the action of a permutation \(\pi\) on an infection \(J\):
\[
\pi J
:=
\left(J_{\pi^{-1}(1)}, \ldots, J_{\pi^{-1}(n)}\right).
\]
In other words, the infection status of the image vertex \(\pi(u)\) is the infection status of \(u\) under \(J\).

\subsubsection{Invariant statistics}

The theory presented in our paper applies to the following class of statistics:
\begin{definition}
Suppose $\Pi$ is a subgroup of $S_n$. A statistic \(S: \infspace_{k,c} \rightarrow \mathbb{R}\) is \(\Pi\)-\emph{invariant} if
\(S(J) = S(\pi J)\) for any \(J \in \infspace_{k,c}\) and \(\pi \in \Pi\).
\end{definition}

We now describe some of the invariant statistics we use in our exposition. We first define the \emph{edges-within} statistic, corresponding to the number of edges within the subgraph of $\graphset$ induced by infected vertices, as follows:
\begin{equation*}
W_{\graphset}(J) 
:= 
\sum_{(u, v) \in \edgeset} \ind\{J_u = J_v = 1\}.
\label{eqn: edges within}
\end{equation*}
Intuitively, the value of $W_{\graphset}(J)$ should be comparatively larger if $\graphset$ corresponds to the graph underlying the infection. In our permutation test, we will compute the edges-within statistic with respect to the graph $\graphset_1$ appearing in the alternative hypothesis (in the case of a simple test), so we reject $H_0$ when $W_{1}(J) := W_{\graphset_1}(J)$ exceeds a certain threshold.

For Ising models, note that 
\(\energy(J) = -W(J)\), ignoring censored vertices.
We derive the invariance of the statistics $W$ and $\energy$ under the permutation group $\Pi = \Aut(\graphset)$ in the Appendix.


\subsection{Test statistics and thresholds}
\label{subsec:setup:testStatisticsThresholds}

Finally, we provide some general notation for defining rejection regions. Consider a test statistic \(S: \infspace_{k, c} \to \reals\).
Without loss of generality, we reject the null hypothesis for large values of \(S\).
For a fixed value of \(\beta\), we define the \((1 - \alpha)\)-quantile of \(S(\infected)\) to be
\[
t_{\alpha, \beta}
:=
\sup\left\{
t \in \support(S): 
\prob_{\beta}\left(S(\infected) \geq t\right)
> 
\alpha
\right\},
\]
where \(\support(S)\) is the support of \(S\), i.e., the set of values that the discrete random variable \(S\) takes with positive probability.
Note that we have the two bounds	
\begin{align}
& \begin{aligned}
\prob_{\beta}\left(S(\infected) > t_{\alpha, \beta}\right)
&\leq 
\alpha, \\
\prob_{\beta}\left(S(\infected) \geq t_{\alpha, \beta}\right)
&>
\alpha,
\label{eqnGeqThreshold}
\end{aligned}
\end{align}
i.e., the strictness of the inequality in the event determines the direction of the inequality in the probability bound.

Our discretization-based tests employ simulated null distributions of the test statistic \(S\) in order to approximate \(t_{\alpha, \beta}\).
To quantify the exact uncertainty introduced by the simulations, we define the empirical \((1 - \alpha)\)-quantile \(\hat{t}_{\alpha}:\infspace_{k, c}^{N_{\sims}} \to \reals\) to be 
\[
\hat{t}_{\alpha, \beta}
:=
\sup\left\{t \in \support(S): 
\frac{1}{N_{\sims}} \sum_{i = 1}^{N_{\sims}}
\ind\{S(I_{i}) \geq t\}
> \alpha
\right\}.
\]

\section{Main theoretical results}
\label{sec: main results}

Our theoretical results for permutation testing are motivated by the following observation: When \(\graphset_{0}\) is the empty graph, the coordinates of the infection vector $\infected$ are exchangeable. Hence, we may conduct a valid permutation test based on any test statistic computed with respect to $\infected$, where the rejection rule is given by the quantiles of the distribution of $\pi \infected$, with $\pi \sim \text{Uniform}(S_n)$. However, the permutation test remains valid in somewhat more general settings.

Throughout this section, we will assume the setting of simple hypothesis testing, where $\graphspace_0 = \{\graphset_0\}$ and $\graphspace_1 = \{\graphset_1\}$ are singleton sets and $\paramset_0 = \paramset_1 = \{\beta\}$. We will develop a sufficient condition, stated in terms of the interplay between the automorphism groups $\Pi_0$ and $\Pi_1$, which guarantees the validity of a permutation test applied to any $\Pi_1$-invariant statistic. A generalization of these results to composite hypothesis testing is provided in Appendix~\ref{appFurtherTheory}.

We let \(S\) denote a \(\Pi_{1}\)-invariant statistic, as defined in Section~\ref{secInvariant}. We also define
\begin{equation*}
\Pi_{10} := \Pi_{1} \Pi_{0} =  \{\pi_{1} \pi_{0}: \pi_{i} \in \Pi_{i}\}.
\label{eqn: def pi}
\end{equation*}
The following key theorem shows that the distribution of the test statistic is the same when applied to a random permutation of the infection vector, provided $\Pi_{10} = S_n$:

\begin{theorem}
\label{ThmPermStat}
Suppose that for any \(\pi_{0}\) in \(\Pi_{0}\), the distribution of \(\infected\) satisfies
\begin{equation}
\prob_{0}(\infected = J)
=
\prob_{0}(\pi_{0} \infected = J).
\label{eqnPermCond}
\end{equation}
Let \(\pi\) be drawn uniformly from \(S_{n}\).
If \(\Pi_{10} = S_{n}\), 
the statistics \(S(\infected)\) and \(S(\pi \infected)\) have the same distribution under the null hypothesis. 
\end{theorem}


Showing that equation~\eqref{eqnPermCond} holds for our models is straightforward, and we do this in the Appendix.
In particular, the condition $\Pi_{10} = S_n$ holds when $\graphset_0$ is the empty graph, since $\Pi_0 = S_n$ in that case. The next result shows that the condition described in Theorem~\ref{ThmPermStat} is sufficient to guarantee the success of a straightforward permutation test, described in Algorithm~\ref{AlgPermExact}:

\begin{theorem}
\label{ThmPermTest}
Suppose equation~\eqref{eqnPermCond} holds and $\Pi_{10} = S_n$. The permutation test described in Algorithm~\ref{AlgPermExact} controls Type I error at level $\alpha$.
\end{theorem}

\begin{algorithm}[!h]
\SetKwInOut{Input}{Input}
\Input{Type I error tolerance $\alpha > 0$, observed infection vector $\infected$}
For each $\pi \in S_n$, compute the statistic $S(\pi \infected)$

Determine the threshold $t_\alpha$ such that
\begin{equation*}
t_{\alpha}
=
\sup\left\{t \in \support(S):
\frac{1}{n!} \sum_{\pi \in S_n} \ind \{S(\pi \infected) \geq t\} > \alpha
\right\}
\end{equation*}

Reject $H_0$ if and only if $S(\infected) > t_\alpha$
\caption{Permutation test (exact)}
\label{AlgPermExact}
\end{algorithm}


\begin{remark}
The proof of Theorem~\ref{ThmPermTest} critically leverages the property
\begin{equation}
\label{EqnPermEquiv}
S(\infected) \stackrel{d}{=} S(\pi \infected), \qquad \text{where } \pi \sim \text{Uniform}(S_n).
\end{equation}
Note that this property would clearly hold in the case when the components of $\infected$ are exchangeable, since we have $S(\infected) \stackrel{d}{=} S(\pi \infected)$ for any fixed $\pi \in S_n$ in that case. Furthermore, under the ``independent neighborhoods condition'' invoked by \cite{milling2015}, condition~\eqref{EqnPermEquiv} holds, as well.
However, the alternative graph \(\graphset_{1}\) is randomly generated in such settings, so the statistic \(S\) is also random. We discuss this more precisely in the Appendix. 
\end{remark}

For large values of $n$, it is undesirable to compute $S(\pi \infected)$ for all permutations $\pi \in S_n$. Instead, we may approximate the rejection threshold $t_\alpha$ for the permutation test using Monte Carlo simulation, leading to Algorithm~\ref{AlgPermApprox}. As an immediate corollary to Theorem~\ref{ThmPermTest}, Algorithm~\ref{AlgPermApprox} is asymptotically accurate as $N_{\sims} \rightarrow \infty$.

\begin{algorithm}[!h]
\SetKwInOut{Input}{Input}
\Input{Type I error tolerance $\alpha > 0$, integer $N_{\sims} \ge 1$, observed infection vector $\infected$}
Draw $\pi_1, \dots, \pi_{N_{\sims}} \stackrel{i.i.d.}{\sim} \text{Uniform}(S_n)$ and compute the statistics $S(\pi_i \infected)$

Determine a threshold $\hat{t}_\alpha$ such that
\begin{equation*}
\hat{t}_{\alpha}
=
\sup\left\{
t \in \support(S):
\frac{1}{N_{\sims}} \sum_{i = 1}^{N_{\sims}} \ind\{S(\pi_i \infected) \geq t\} > \alpha
\right\}
\end{equation*}

Reject $H_0$ if and only if $S(\infected) > \hat{t}_\alpha$
\caption{Permutation test (approximate)}
\label{AlgPermApprox}
\end{algorithm}

\begin{remark}
Note that the permutation tests described in Algorithms~\ref{AlgPermExact} and~\ref{AlgPermApprox} are very simple to execute and do not involve approximating the parameters $\lambda$ or $\beta$ in any way. Rather, the algorithms exploit differences in the symmetry structures of $\graphset_{0}$ and $\graphset_{1}$. As a caveat, the usefulness of the guarantee in Theorem~\ref{ThmPermTest} also depends on properties of the graphs $\graphset_0$ and $\graphset_1$ and their relationship to the test statistic $S$. In particular, if $S = W_{1}$ is the edges-within statistic and $\graphset_1$ is the empty  graph, we always have $S = 0$. Thus, the threshold for the permutation test would be $t_\alpha = 0$, and the test would never reject $H_0$. Of course, this is a valid level-$\alpha$ test, but it has power \(0\). As seen in the simulations of the Appendix, this may also be a pitfall of the algorithms suggested by \cite{milling2015}. However, we prove rigorously in Section~\ref{sec: examples} below that our permutation test results in meaningful hypothesis testing procedures with reasonable risk bounds for a variety of interesting scenarios.
\end{remark}


\section{Examples and risk bounds}
\label{sec: examples}

We now provide several examples to illustrate the use of Theorem~\ref{ThmPermTest}. We focus on the cases when one graph is the star graph and the other is a vertex-transitive graph. In these cases, we are able to compute interpretable risk bounds for our permutation test; our calculations are valid under the stochastic spreading model, which we assume to be the setting for all the risk bounds computed in this section. We also assume a simple hypothesis testing scenario, where we only specify the parameter $\beta$ for $H_1$ and let \(\beta\) be arbitrary for \(H_{0}\). 
Also, we compute the risk for Algorithm~\ref{AlgPermExact}, although a finite-simulation result could also be obtained for Algorithm~\ref{AlgPermApprox} via a union bound argument.

The following corollary to Theorem~\ref{ThmPermStat} will be useful in our development. It implies that the risk incurred by testing \(\graphset_{0}\) against \(\graphset_{1}\) is exactly equal to the risk incurred by testing the empty graph against \(\graphset_{1}\):

\begin{corollary}
Suppose $S$ is a $\Pi_1$-invariant statistic and $\Pi_{10} = S_n$. 
The risk of any test based on $S$ is equal to the risk of the same test computed with respect to the null hypothesis $H_0'$ involving the empty graph.
\label{CorRisk}
\end{corollary}

In our examples, we analyze the simple case of a star graph versus a graph such that \(\Pi_{10} = S_{n}\).
These graphs are exactly the vertex transitive graphs. Recall the following definition~\cite{godsil2013}:
\begin{definition}
\label{DefVertTran}
A graph $\graphset$ is \emph{vertex-transitive} if every pair of vertices is equivalent under some element of $\Aut(\graphset)$; i.e., for any $u,v \in V(\graphset)$, we have $\pi(u) = v$ for some $\pi \in \Aut(\graphset)$.
\end{definition}
\noindent Basic examples of vertex-transitive graphs include the cycle graph and the toroidal grid.

\begin{figure}[!htb]
\minipage{0.32\textwidth}
  \includegraphics[width=\linewidth]{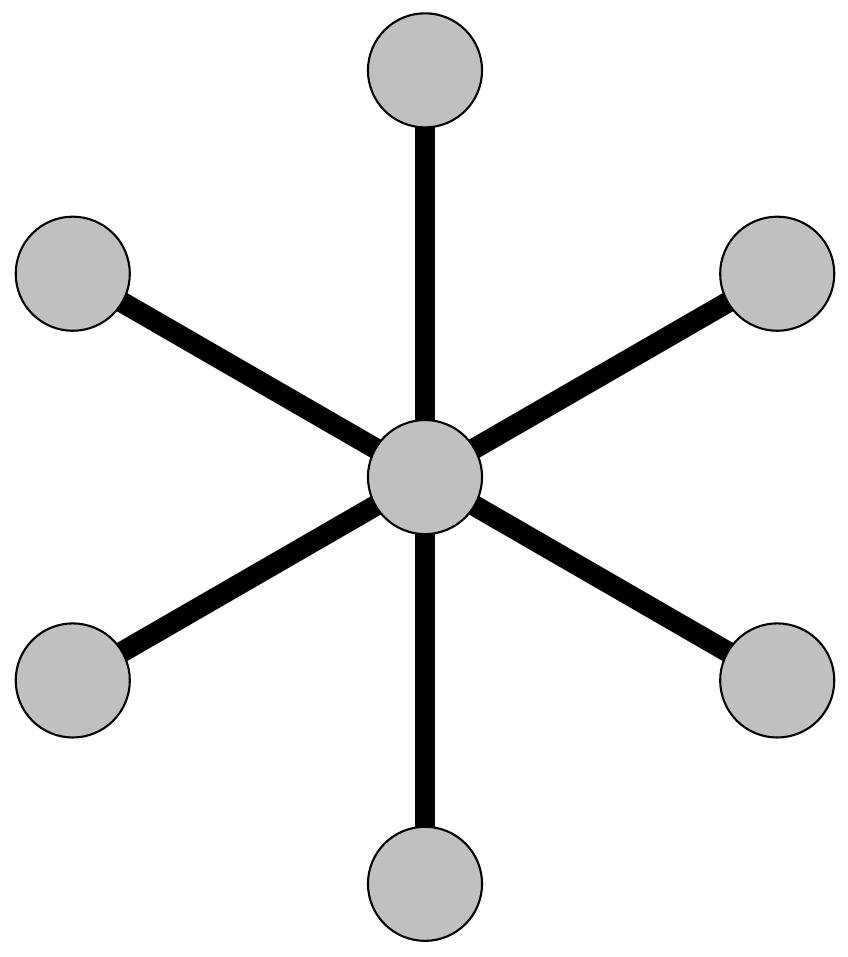}
\endminipage\hfill%
\minipage{0.32\textwidth}
  \includegraphics[width=\linewidth]{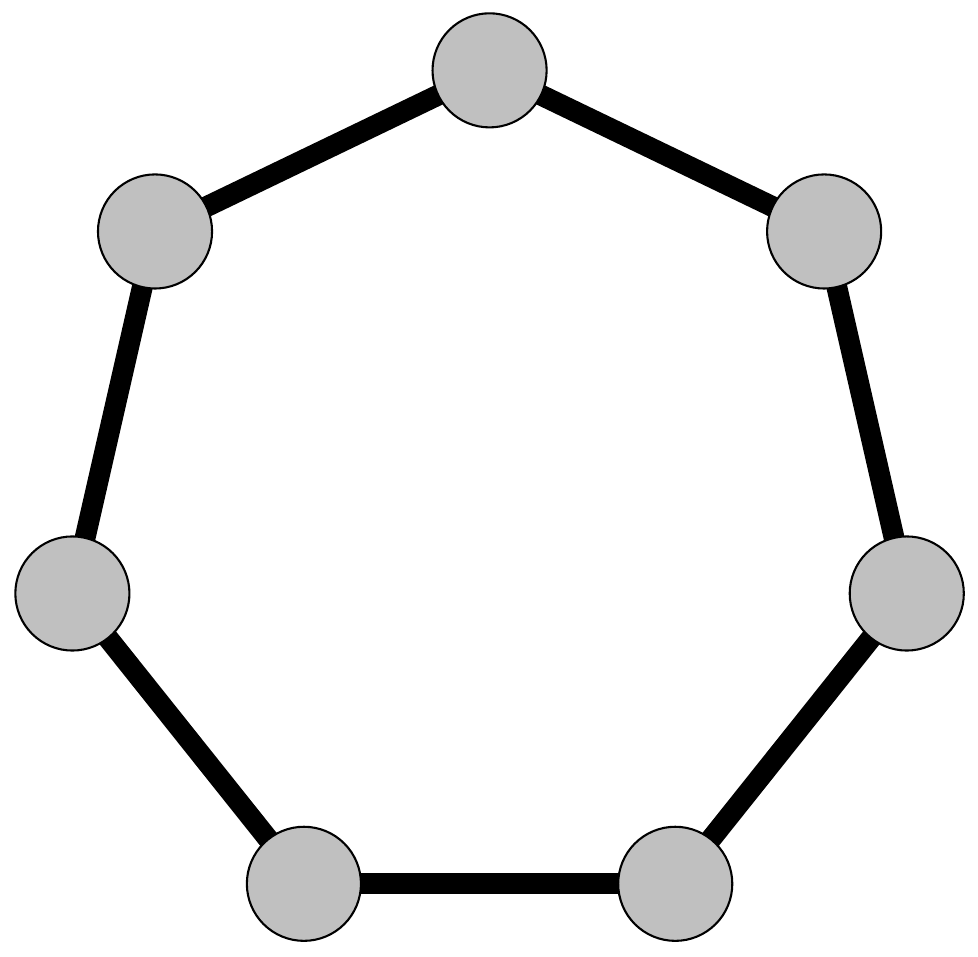}
\endminipage\hfill
\minipage{0.32\textwidth}%
  \includegraphics[width=\linewidth]{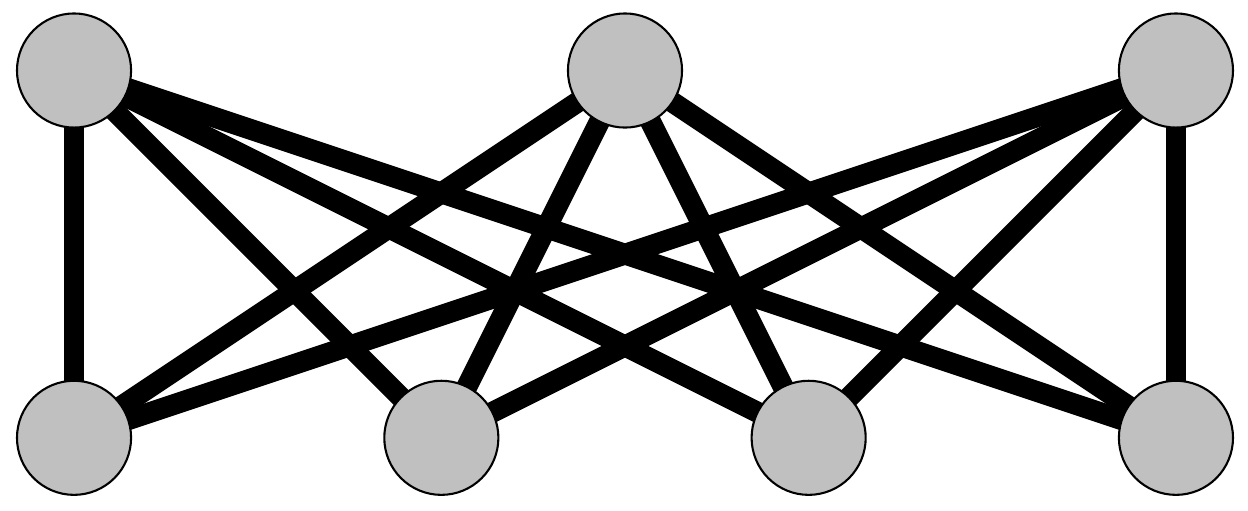}
\endminipage
\caption{A \(6\)-star, a \(7\)-cycle, and the complete bipartite graph \(K_{3, 4}\). 
If $\graphset_0$ is the \(7\)-star and $\graphset_1$ is the \(7\)-cycle, then \(\Pi = S_{7}\).
If $\graphset_0$ is the 7-star and \(\graphset_{1} = K_{3, 4}\), then \(\Pi \neq S_{7}\).}
\label{fig: graphs}
\end{figure}

\subsection{The star graph as $\graphset_0$}

First, let \(\graphset_{0}\) be the star graph on \(n\) vertices. Without loss of generality, let vertex \(1\) be the center vertex. Note that \(\Pi_{0} \cong S_{n - 1}\). In fact, \(\Pi_{10} = S_{n}\) whenever \(\Pi_{1}\) contains permutations mapping vertex \(1\) to any other vertex, which is equivalent to the definition of vertex transitivity. We summarize this observation in the following corollary:

\begin{corollary}
Let \(\graphset_{0}\) be the star graph and suppose \(\graphset_{1}\) is vertex-transitive.
Let \(\pi \sim \text{Uniform}(S_{n})\),
and suppose \(S\) is a \(\Pi_{1}\)-invariant statistic.
Then \(S( \infected) \stackrel{d}{=} S(\pi \infected)\) under the null hypothesis, and the permutation test described in Algorithm~\ref{AlgPermExact} controls the Type I error at level $\alpha$. A similar result holds when $\graphset_1$ is the star graph and $\graphset_0$ is vertex-transitive. 
\label{cor: vertex transitive}
\end{corollary}

We now turn to risk bounds. Before we present the bounds, we provide some motivation.
When \(\beta\) is small, the two hypothesis spaces are close together, so the risk will be large; thus, we focus instead on the regime of moderate to large \(\beta\).
When the null hypothesis is true, the infection still looks uniformly random when viewing \(\graphset_{1}\), as stated in Corollary~\ref{CorRisk}.
Using a standard concentration bound, we then may verify that \(t_{\alpha}\) is 
\(O(k^{2}/n + \sqrt{k})\), where the two terms correspond to the expected value and the variance of \(W_{1}(\infected)\) under the null. When the alternative is true, the infected vertices should induce a connected subgraph of \(\graphset_{1}\), with high probability. 
Thus, the expected value of \(W_{1}(\infected)\)  is \(\Omega(k)\).
Finally, we use another standard concentration bound to obtain an upper bound on the probability that \(W_{1}(\infected) \leq t_{\alpha}\) under the alternative.

Recall that all the vertices of a vertex-transitive graph have the same degree, which we denote by $D$. It is natural that the risk depends on $D$, since larger values of $D$ allow for more variation in $W_{1}(\infected)$.
Let \(N_{t, \min}\) denote the minimum possible cut between the \(t\) infected vertices and \(n - t\) uninfected vertices at time $t$.
We define the function
\begin{equation*}
H(\beta)
=
\prod_{m = 1}^{k - 1}
\frac{\beta}{n - m + \beta N_{t, \min}}.
\label{eqn: H}
\end{equation*}
We also define a \emph{cascade} on \(k\) vertices to be a surjective map 
\(f: \vertexset \to \{0, \ldots, k\}\),
such that
\begin{itemize}
\item[(i)] \(v\) is uninfected when \(f(v) = 0\),
\item[(ii)] \(v\) is the \(i^\text{th}\) vertex infected when \(f(v) = i\), and
\item[(iii)] if $f(v) = i$, then $v$ must be adjacent to one of the first \(i - 1\) infected nodes.
\end{itemize}
Let \(\scriptC_{k}(u, v)\) denote the set of cascades on \(k\) vertices such that both \(u\) and \(v\) are infected, and let $C_{k} := \min_{(u, v) \in \edgeset_{1}} |\scriptC_{k}(u, v)|$. We have the following bound:
\begin{proposition}
Suppose \(\graphset_{1}\) is a connected vertex-transitive graph with degree \(D\).
Let \(\psi_{W, \alpha}\) be the level-$\alpha$ permutation test based on the edges-within statistic $W_{1}$. Then
\begin{align*}
& \begin{aligned}
R_{k, 0}(\psi_{W, \alpha}, \beta)
\leq \alpha
+  
\exp\left\{-\frac{2}{k D^{2}}\left(\frac{D}{2} C_{k} H(\beta)
-  \frac{D k(k - 1)}{2 (n - 1)} - \sqrt{\frac{kD^2}{2} \log \frac{1}{\alpha}}\right)^{2}\right\}.
   \end{aligned}
\end{align*}
\label{PropStarRisk}
\end{proposition}

\begin{remark}
Note that $H(\beta)$ is increasing in $\beta$, so the risk bound in Proposition~\ref{PropStarRisk} decreases as $\beta$ increases. This agrees with intuition, since higher values of $\beta$ correspond to a higher chance that the infection propagates via edges rather than by random infections. Thus, the graphs $\graphset_0$ and $\graphset_1$ should be easier to distinguish.
\end{remark}

For a cycle graph with $k < n / 2$, we have the following result:

\begin{corollary}
Let \(\graphset_{1}\) be the \(n\)-cycle. Then $C_k = (k-1) 2^{k-1}$ and $D = 2$, so
\begin{align*}
& \begin{aligned}
R_{k, 0}(\psi_{W,\alpha}, \beta) 
&\le 
\alpha + \exp\Bigg\{-\frac{1}{2k} \Bigg((k-1) 2^{k-1} \prod_{m=1}^{k-1} \frac{\beta}{n-m+2\beta} \\
& \qquad - \frac{k(k-1)}{n-1} - \sqrt{2k\log\left(\frac{1}{\alpha}\right)}\Bigg)^{2}\Bigg\}.
   \end{aligned}
\end{align*}
In particular, if $\alpha =\exp(-C k / 2)$ and 
\(k / n  + C \leq 1 - \epsilon\) for some \(C > 0\) and \(\epsilon > 0\),
then there exist \(C', C'' > 0\) such that 
\begin{equation*}
\lim_{\beta \rightarrow \infty} R_{k, 0}(\psi_{W, \alpha}, \beta) 
\le 
C'\exp(-C'' k).
\end{equation*}
\label{CorStarRisk0}
\end{corollary}

The last statement reveals that as \(\beta \to \infty\), the risk will vanish for sufficiently large values of $k$. However, if the fraction $k/n$ of infected nodes becomes too large, the two hypotheses are again difficult to distinguish.

\subsection{The star graph as $\graphset_1$}
\label{SecStarAlt}

We now consider the case when $\graphset_1$ is the star graph on $n$ vertices. Again,
let vertex \(1\) denote the center of the star.
Perhaps unsurprisingly, it turns out that the maximum likelihood estimator and a test based on the edges-within statistic reduce to the same decision rule, depending on whether vertex \(1\) is included in the infected set:

\begin{proposition}
Let \(c = 0\).
Suppose $\graphset_0$ is the empty graph. Maximum likelihood estimation is equivalent to the center indicator test \(C = \ind\{\infected_{1} = 1\}\), which is in turn equivalent to permutation testing at level $\alpha$ based on the edges-within statistic $W_{1}$, when \(\alpha \ge k/n\).
\label{prop: star center}
\end{proposition}

Risk bounds for hypothesis testing based on $C$ are relatively easy to compute when $\graphset_0$ is the empty graph. Corollary~\ref{CorRisk} implies that such bounds hold for permutation testing when $\graphset_0$ is any vertex-transitive graph, from which we may derive the following result:

\begin{proposition}
Suppose $\graphset_0$ is a vertex-transitive graph. The risk of the center indicator test on the star graph on \(n\) vertices satisfies the following bounds:
\begin{equation*}
R_{k,0}(P_C, \beta) \ge \frac{k}{n} +  \exp\left(- \frac{k + \beta k(k - 1)/2}{n - k} \right),
\end{equation*}
and
\begin{equation*}
R_{k, 0}(P_{C}, \beta) \le \begin{cases}
\frac{k}{n} + \exp\left(-\frac{k + \beta k (k - 1)/2}{(n-k+1) + (k-1)\beta}\right), & \text{if } \beta \ge 1, \\
\frac{k}{n} + \exp\left(-\frac{k + \beta k(k-1)/2}{n}\right), & \text{if } \beta < 1.
\end{cases}
\end{equation*}
\label{prop: star risk}
\end{proposition}

Again, we can interpret the behavior of the risk bounds in terms of the fraction of infected vertices $k / n$. When \(\beta\) is fixed, the bound is $k / n + \exp\left(\Theta\left(-\beta k^2 / n\right)\right)$; thus, if we consider the size of the graph to be growing, we require  $k / n \to 0$ and $k^2 / n \to \infty$ in order to have vanishing risk.
The first condition suggests that \(k\) cannot be large enough to randomly infect the center of the star under the null.
The intuition for the latter condition is that under the alternative, each infected leaf attempts to infect the center of the star at successive time steps.
This leads to at most \(k(k - 1)/2\) infection attempts, and when the strength of these attempts is large enough, the center is infected with high probability.


\section{A closer look at test statistics}

\label{secCloserLook}
It is natural to wonder which of the invariant statistics leads to the best statistical test, or even whether it is reasonable to focus our attention on invariant statistics. We address the first question by showing a rough equivalence of likelihood ratio testing to tests based on the edges-within statistic. For the second question, we derive some results motivated by the Hunt-Stein theory of hypothesis testing.


\subsection{Likelihood ratio and edges-within}
\label{subsecLikelihoodRatio}
As noted earlier, the edges-within statistic appears in the probability density function of the \(01\)-Ising model: \(\energy(J) = -W(J)\). Thus, a test based on a likelihood calculation may be equivalently expressed in terms of the edges-within statistic.

Turning to the stochastic spreading model, we now show that $W_1$ arises as the first-order coefficient in the series expansion of the likelihood with respect to $\beta$.
This is somewhat reminiscent of the use of signed triangles in the random graph testing literature~\cite{banerjee2018, banerjee2017, bubeck2016}, which appears in the latter two cases as a first-order approximation to the asymptotic distribution of the log-likelihood.
Furthermore, we will see that when considering the likelihood ratio of a test with a specific type of composite null hypothesis against the simple alternative $\graphspace_1 = \{\graphset_1\}$, the edges-within statistic \(W_{1}\) also appears as the first-order coefficient of the likelihood ratio. These approximations are quite attractive, since as noted earlier, computing the likelihood ratio would require summing over \(k!\) different infection paths and is itself intractable.


Before stating the main result, we introduce some additional notation. Let $L(\graphset, \beta; J)$ denote the likelihood of infection $J$ under graph $\graphset$ and spreading parameter $\beta$. Let
\begin{equation*}
R(\graphset_{0}, \graphset_{1}, \beta; J) = \frac{L(\graphset_1, \beta; J)}{L(\graphset_0, \beta; J)}
\end{equation*}
denote the likelihood ratio, and let \(\graphset_{\emptytext}\) denote the empty graph. Recall that $N_t$ denotes the number of edges connecting infected vertices to uninfected vertices at time $t$.

Our main theorem shows the approximate equivalence between likelihood ratio tests and thresholding the edges-within statistic in the case of simple hypothesis testing when the null graph is empty.

\begin{theorem}
Consider the hypothesis test of \(\graphset_{\emptytext}\) versus \(\graphset_{1}\).
For an uncensored $P \in \pathset(J)$ in which \(k + c'\) vertices are infected, define the function
\[
Q(P)
= 
\sum_{t = 1}^{k + c'}
\frac{\neighbors_{t}}{n + 1 - t}.
\]
If \(c = 0\), we have 
\begin{align*}
& \begin{aligned}
R(\graphset_{\emptytext}, \graphset_{1}, \beta; J)
&= 
\left(1 + \beta W_{1}(J) + O(\beta^{2} W_{1}(J)^{2})\right)  
\left(
1  - \frac{1}{k!}\sum_{P \in \pathset(J)}  \beta Q(P) 
\right).
\end{aligned}
\end{align*}
If \(c > 0\), we have
\begin{align*}
& \begin{aligned}
R(\graphset_{\emptytext}, \graphset_{1}, \beta; J)
&=
\sum_{J' \in \uncensored(J)}
D(\beta, k, c, |J'|)
\left(1 + \beta W_{1}(J') + 
 O(\beta^{2} W_{1}(J')^{2})\right) \\ 
 & \qquad \times
\left(1 - \frac{1}{|J'|!} \sum_{P \in \pathset(J')}O(\beta Q(P))\right),
\end{aligned}
\end{align*}
where
\[
D(\beta, k, c, r)
:=
\binom{n}{k \; \; c} 
\left(
Z_{1}(\beta) \binom{n}{c} \binom{n}{r}
\right)^{-1}.
\]
\label{theoremLikelihoodRatio}
\end{theorem}

Thus, Theorem~\ref{theoremLikelihoodRatio} shows that when \(c = 0\), the edges-within statistic \(W_{1}(\infected)\) is the leading term of the likelihood ratio expanded as a function of \(\beta\).
When \(c > 0\), the expression is somewhat more complicated, since the number of edges within the (uncensored) infection subgraph may differ depending on the statuses of censored vertices. The first-order term then equals the value of the edges-within statistic averaged over all possible infection vectors giving rise to the censored vector $J$.

\begin{remark}
Note that even if the likelihood ratio converges to its first-order approximation \(1 + \beta W_{1}\) for some sequence of parameters, this does not ensure that \(W_{1}\) will \emph{always} lead to a useful test.
For instance, in the case that \(\graphset_{1}\) is a star graph, the condition \(\beta W_{1} \to 0\) and the conditions for asymptotically vanishing risk are incompatible.
As discussed in Section~\ref{SecStarAlt}, the latter conditions require \(k/n \to 0\) and \(\beta k^{2} / n \to \infty\).
Since \(W_{1}\) may take on the value \(k - 1\), requiring that \(\beta W_{1}, k / n \to 0\) would imply that \(\beta k^{2} / n \to 0\), as well.
\end{remark}

Finally, to obtain a test between two non-empty graphs, we use the simple relation
\[
R(\graphset_{0}, \graphset_{1}, \beta; \infected)
=
\frac{R(\graphset_{\emptytext}, \graphset_{1}, \beta; \infected)}{R(\graphset_{\emptytext}, \graphset_{0}, \beta; \infected)}.
\]
This expression depends on both \(W_{0}\) and \(W_{1}\). However, cases exist where the expression only depends on $W_1$. For example, in the case of a composite null hypothesis involving testing a graph topology, we may derive the following theorem:

\begin{theorem}
If \(\graphspace_{0} = S_{n} \graphspace_{0}\),
then \(\sup_{\theta \in \Theta_{0}} L(\theta; J)\) is constant over \(J\) in $\infspace_{k,c}$.
Consequently, if we denote this constant by \(D'\),
we have
\[
R(\graphspace_{0}, \graphset_{1}, \beta; J)
=
\frac{1}{D'}
R(\graphset_{\emptytext}, \graphset_{1}, \beta; J).
\]
\label{theoremConstantLikelihood}
\end{theorem}

In conjunction with Theorem~\ref{theoremLikelihoodRatio}, Theorem~\ref{theoremConstantLikelihood} shows that we may again extract $W_1$ as the leading term in the expansion for the likelihood ratio.
Another immediate corollary is that it is not possible to directly test between two unlabeled topologies, since the likelihoods for both hypotheses would be constant.


\subsection{Hunt-Stein theory for graph testing}

When conducting hypothesis tests, it is natural to consider maximin tests, i.e., tests that maximize the minimum power over the  space of alternative hypotheses.
A standard way to do this for invariant tests is via the Hunt-Stein theorem~\cite{lehmann2006}. 
The goal of this section is to demonstrate that a version of the Hunt-Stein theorem also holds for the graph testing problem, further motivating our use of hypothesis testing via $\Pi_1$-invariant test statistics.
The results in this subsection hold for general null and alternative parameter spaces \(\Theta_{0} = \graphspace_{0} \times \reals_{+}\) and \(\Theta_{1} = \graphspace_{1} \times \reals_{+}\). Let \(\Theta := \Theta_{0} \cup \Theta_{1}\).

Recall that a critical function \(\varphi\) outputs a value in \(\{0, 1\}\) for each observed infection vector $\infected$, corresponding to the selected hypothesis. A test based on $\varphi$ is $\Pi$-invariant if $\varphi(\pi \infected) = \varphi(\infected)$ for all $\pi \in \Pi$. Furthermore, the test is \emph{maximin} at level $\alpha$ if
\begin{equation}
\label{EqnAlpha}
\max_{\theta \in \Theta_0} \E_\theta [\varphi(\infected)] \le \alpha,
\end{equation}
and the value of
\begin{equation*}
\min_{\theta \in \Theta_1} \E_{\theta} [\varphi(\infected)]
\end{equation*}
is maximized among all tests $\varphi'$ satisfying inequality~\eqref{EqnAlpha}.

Analogous to canonical Hunt-Stein results, we will assume that both \(\graphspace_{0}\) and \(\graphspace_{1}\) are invariant under the same group of transformations, which in our setting is \(\Pi_{1}\). A natural case where this condition is satisfied is when \(\graphspace_{1}\) consists of a single graph \(\graphset_{1}\) and \(\graphspace_{0}\) consists of all permutations of a graph \(\graphset_{0}\). For instance, suppose \(\graphset_{1}\) is a cycle graph;
if we wish to test a null hypothesis involving a star graph $\graphset_0$, we can define $\graphspace_0$ to include all permutations of $\graphset_0$ under $\Pi_1$, as well, which yields \(n\) stars with different center vertices.



\begin{theorem}[Hunt-Stein for graph testing]
Let \(\Pi\) be a group of transformations on \(\infected\), and let \(\Theta_{0}\), and \(\Theta_{1}\) be such that
\(\Pi \Theta_{0} = \Theta_{0}\) and \(\Pi \Theta_{1} = \Theta_{1}\). 
If there exists a level-\(\alpha\) test \(\varphi^{*}\) maximizing \(\inf_{\theta \in \Theta} \expect_{\theta} [\varphi(\infected)]\), then there also exists a $\Pi$-invariant test with this property, defined by
\[
\psi^*(J)
=
\frac{1}{|\Pi|} \sum_{\pi \in \Pi} \varphi^*(\pi J).
\]
\label{theoremHuntStein}
\end{theorem}

Thus, for composite tests in which \(\Theta_{0}\) and \(\Theta_{1}\) are both \(\Pi_{1}\)-invariant, 
whenever we can find a maximin critical function \(\varphi\), we can also find a maximin \(\Pi_{1}\)-invariant critical function.
Note, however, that one could likely do better than a \(\Pi_{1}\)-invariant test when \(\graphspace_{0}\) is not \(\Pi_{1}\)-invariant, such as in the case of a star with a fixed center.


\section{General testing procedures}
\label{sec:general_graphs}

Our permutation testing procedure exploits graph symmetries to great effect; however, the strong symmetry assumptions may sometimes be prohibitive.
One particular case would be testing a graph \(\graphset\) against a ``close'' variant \(\graphset'\), such as a subgraph of \(\graphset\) with a few edges removed.

Consider testing the hypotheses
\begin{align*}
&\begin{aligned}
H_{0}
&: \infected \sim \ssm(\graphset_{0}, \beta_{0}) \text{ for } \beta_{0} \in \paramset_{0}, \\
H_{1}
&: \infected \sim \ssm(\graphset_{1}, \beta_{1}) \text{ for } \beta_{1} \in \paramset_{1}.
\end{aligned}
\end{align*}
We now propose and analyze a general testing procedure. The idea is straightforward: For a finite set $\completespace_D$, define a discretization function 
\(F_{D}: [0, \infty) \to \completespace_{D}\) such that
\begin{equation}
\tvdist\left(\prob _{0, \beta}, \prob _{0, F_{D}(\beta)}\right)
\leq 
\delta,
\label{eqnDiscretizationFunction}
\end{equation}
and \(\paramset_{0, D}\) = \(F_{D}(\paramset_{0})\).
We then simulate a sufficient number of infection processes with the parameters \((\graphset_{0}, \beta_{0})\) for \(\beta_{0}\) in \(\paramset_{0, D}\) to approximate the distribution for all parameter values in the null hypothesis, which allows us to perform tests or, alternatively, form confidence sets for the underlying graph structure.
Although the idea is fairly simple, the appropriate discretization is somewhat lengthy to state; therefore, we defer further details on the discretization to Appendix~\ref{app:general_graphs}.
In the remainder of the section, we describe basic hypothesis tests and confidence interval algorithms.

\subsection{Hypothesis testing algorithm}
\label{subsecHypothesisTestingAlgorithm}

With a discretization satisfying equation~\eqref{eqnDiscretizationFunction} in hand, it is straightforward to devise a simulation-based algorithm that controls the Type I error.
This is summarized in Algorithm~\ref{algorithmPCDiscretizationTest}.
Further details on algorithm parameters may be found in Appendix~\ref{app:general_graphs}, but briefly, we note that decreasing \(\epsilon\), \(\delta\), or \(\xi\) increases the size of the discretization and therefore the computational burden, while increasing the statistical power.

\begin{algorithm}[!h]
\SetKwInOut{Input}{Input}
\Input{Type I error tolerance $\alpha > 0$, approximation parameters \(\epsilon\), \(\delta\), and \(\xi\), observed infection vector $\infected$, null graph \(\graphset_{0}\), statistic \(S\), discretization \(\paramset_{0, D}\) of size \(N\), discretization function \(F_{D}\)}

Define 
\(
N_{\sims} 
\geq \left(\frac{1}{2\epsilon^{2}} + \frac{8}{3\epsilon} \right) \log
\frac{1}{\xi}
\)

For each \(\beta\) in \(\paramset_{0, D}\), simulate an infection \(N_{\sims}\) times on \(\graphset_{0}\) to obtain the approximate \(1 - (\alpha - \gamma - \epsilon)\) quantile \(\hat{t}_{\alpha - \delta - \xi - \epsilon, \beta}\)

Compute \(\hat{t}_{\alpha - \delta - \xi - \epsilon} = \max_{\beta \in \paramset_{0, D}} \hat{t}_{\alpha - \gamma - \epsilon, \beta}\)

Reject the null hypothesis if \(S(\infected) > \hat{t}_{\alpha - \delta - \xi - \epsilon}\)
\caption{One-Statistic Test}
\label{algorithmPCDiscretizationTest}
\end{algorithm}

\begin{theorem}
Let \(\graphset_{0}\) be a graph and \(S\) be a statistic.
Let \(\alpha > 0\) be given, and set \(\epsilon < \alpha - \delta - \xi\).
Then Algorithm~\ref{algorithmPCDiscretizationTest} with a discretization function satisfying equation~\eqref{eqnDiscretizationFunction} controls the Type I error at level \(\alpha\).
\label{theoremPCDiscretizationAlgorithm}
\end{theorem}

We prove this result in Appendix~\ref{app:general_graphs:proofs}.

\begin{remark}
While Theorem~\ref{theoremPCDiscretizationAlgorithm} proves the correctness of simulation tests, we still need to address further statistical and computational details.
Statistically, we are interested in tests with good power.
Since the test statistic \(S\) may be arbitrary, we cannot make statements about the power of the procedure in general; for example, a constant statistic \(S\) would have a Type I error of \(0\) and a Type II error of \(1\).
Nonetheless, in addition to testing a broader class of graphs, we may also use many more statistics than the edges-within statistic.
However, this does not include the likelihood ratio, even though the likelihood ratio suffers from the further problem of being generally uncomputable, because the likelihood statistic is usually a function of the parameters \(\beta_{0}\) and \(\beta_{1}\), as well.
\end{remark}

We also make one remark regarding the relation between Algorithm~\ref{AlgPermApprox} and Algorithm~\ref{algorithmPCDiscretizationTest}.
The question of how large to make \(N_{\sims}\) is more important in the case of general graphs because the size of the discretization \(N\) is usually quite large, making increasing \(N_{\sims}\) costly.
However, if we wished to be more precise for Algorithm~\ref{AlgPermApprox}, note that Theorem~\ref{theoremPCDiscretizationAlgorithm} provides the necessary error framework.
To account for error properly in Algorithm~\ref{AlgPermApprox}, we would choose the parameters of Algorithm~\ref{algorithmPCDiscretizationTest} with \(\delta = 0\), \(\paramset_{0, D} = \{0\}\), and, consequently, \(N = 1\).

\subsection{Confidence set algorithm}
\label{subsectionConfidenceSet}

A convenient consequence of the general discretization procedure is that we can invert the tests to obtain confidence sets for both the graph and the parameter \(\beta\).
To this end, we provide a slightly modified algorithm to compute confidence sets, matching the test that we have given.
In analog with Algorithm~\ref{algorithmPCDiscretizationTest}, we provide the following confidence set algorithm in Algorithm~\ref{algorithmOneStatisticConfidenceSet}.

\begin{algorithm}[!h]
\SetKwInOut{Input}{Input}
\Input{Type I error tolerance $\alpha > 0$, approximation parameters \(\epsilon\), \(\delta\), and \(\xi\), observed infection vector $\infected$, set of graphs \(\graphspace\), statistic \(S\), discretization \(\paramset_{\graphset, D}\) of size \(N_{\graphset}\) for each \(\graphset\) in \(\graphspace\), discretization functions \(F_{\graphset, D}\)}

Define 
\(
N_{\sims} 
\geq \left(\frac{1}{2\epsilon^{2}} + \frac{8}{3\epsilon} \right) \log
\frac{1}{\xi}
\)

For each \(\beta\) in \(\paramset_{\graphset, D}\), simulate an infection \(N_{\sims}\) times on \(\graphset\) to obtain the approximate \(1 - (\alpha - \delta - \xi - \epsilon)\) quantile \(\hat{t}_{\alpha - \delta - \xi - \epsilon, \beta}\)

Define the discrete confidence sets
\[
C_{\graphset, D}
:=
\left\{\beta \in \paramset_{D}:
S(\infected) \leq \hat{t}_{\alpha - \delta - \xi - \epsilon, \beta}
\right\}
\]

Return the confidence set \(C = \bigcup_{\graphset \in \graphspace} 
\left(\graphset, F_{\graphset, D}^{-1}(C_{\graphset, D})\right)\)
\caption{One-Statistic Confidence Set}
\label{algorithmOneStatisticConfidenceSet}
\end{algorithm}
\begin{corollary}
If for each \(\graphset\) in \(\graphspace\) the discretization \(\paramset_{\graphset, D}\) and discretization function \(F_{\graphset, D}\) satisfy equation~\eqref{eqnDiscretizationFunction} and the approximation parameters are chosen as in Theorem~\ref{theoremPCDiscretizationAlgorithm}, the set \(C\) returned by Algorithm~\ref{algorithmOneStatisticConfidenceSet} is a \((1 - \alpha)\)-confidence set.
\label{corOneStatisticConfidenceSet}
\end{corollary}


\section{Application to an HIV graph}
\label{secHIV}

\begin{figure}[!h]
\centering 
\includegraphics[width=3in]{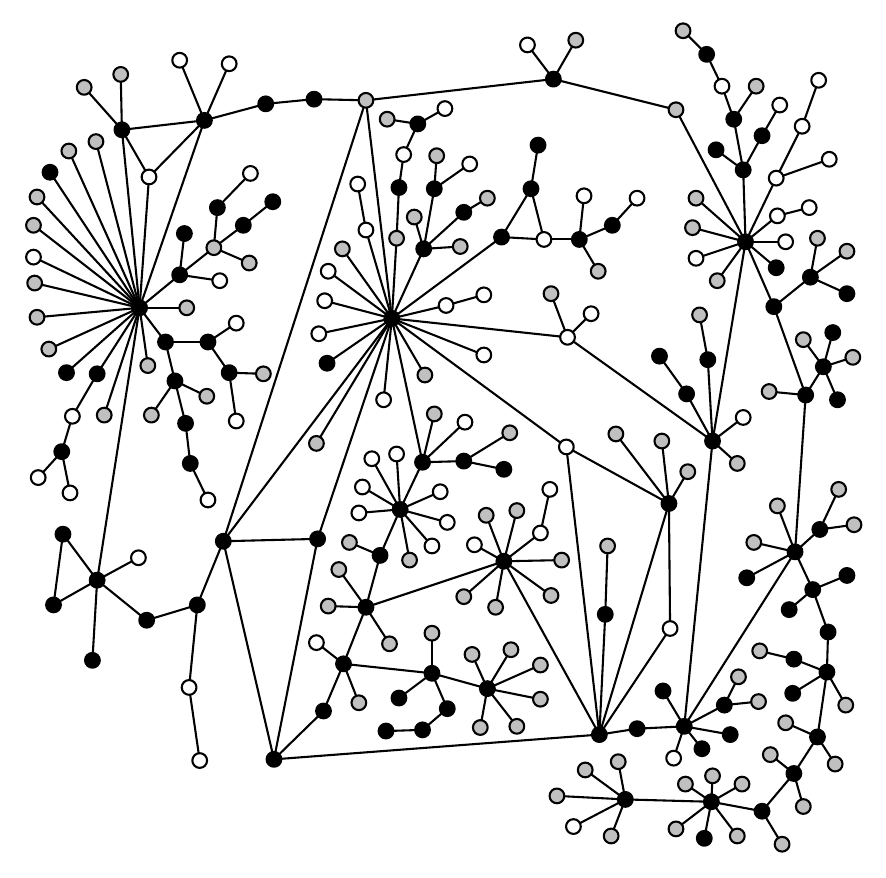}
\caption{The HIV infection graph. Black vertices are infected, white vertices are uninfected, and gray vertices are censored. Edges represent sexual or injecting drug use relationships.
Reproduced from {\it Sexually Transmitted Infections}, Potterat et al., 78, i159--i163, 2002, with permission from BMJ Publishing Group, Ltd.}
\label{fig: hiv graph}
\end{figure}

Finally, we applied our permutation testing procedures to assess the explanatory power of an HIV network. We also examined two close variants of the network with our general graph methods. The network is constructed from the population of Colorado Springs from 1982--1989, and the observed infection statuses of individuals, obtained from the CDC and the HIV Counseling and Testing Center, is reported in \cite{potterat2002}. For a visualization of the network, see Figure~\ref{fig: hiv graph} above. The network was constructed from contact tracing for sexual and injecting drug partners; beginning in 1985, officials interviewed people testing positive for HIV for information to identify relevant partners.

\subsection{Permutation results}

We first applied  Algorithm~\ref{AlgPermApprox} with the edges-within statistic \(W\), resulting in \(W(\infected) = 98\). Based on $N_{\sims} = 1000$ randomly drawn permutations, we obtained an approximate $p$-value of 0; i.e., none of the simulations yielded a value of \(W\) of at least \(98\). This should provide substantial evidence of the explanatory power of the graph. However, a closer examination of the network reveals that the censored nodes might not be chosen at random. To further validate this observation, we simulated our stochastic spreading model on the HIV graph, and computed \(W\) for $\lambda = 1$ and various values of \(\beta\) in \([10, 1000]\). A boxplot is provided in Figure~\ref{fig: hiv boxplots}. In particular, none of the simulations produced a value of \(W\) greater than or equal to \(98\), and increasing \(\beta\) even further seems to have little effect.
\begin{figure}[!h]
\centering 
\minipage{0.48\textwidth}
\includegraphics[width=\linewidth]{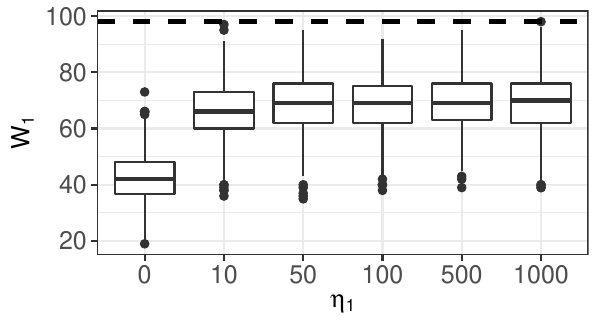}
\endminipage 
\minipage{0.48\textwidth}
\includegraphics[width=\linewidth]{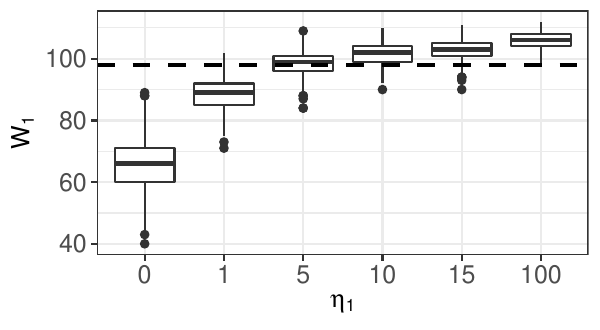}
\endminipage\hfill
\caption{The distributions of \(W\) for different values of \(\beta\) on the HIV graph, based on 1000 simulations. The left plot corresponds to random censoring, whereas the right plot conditions on the censored nodes. The dashed line shows \(W_{1} = 98\), the value from the data. No \(\beta\) appears to be consistent with the observed value \(W = 98\) in the randomized censoring model; such a value is more typical in the conditional censoring model.}
\label{fig: hiv boxplots}
\end{figure}

Subsequently, we used Algorithm~\ref{AlgCondApprox} of the Appendix with statistic \(W\), with the empty graph as the null. The approximate $p$-value computed from 1000 simulations was again 0.
However, simulating the stochastic spreading model on the HIV graph, conditioned on the censored nodes, leads to more reasonable values of $W$. Boxplots for simulations corresponding to $\lambda = 1$ and $\beta \in \{0, 1, 5, 10, 15, 100\}$ are shown in the right-hand plot of Figure~\ref{fig: hiv boxplots}. Note that for moderate values of \(\beta\), the observed value 98 lies squarely within the range of the empirical distribution of $W$.
This suggests that the model which conditions on the censored nodes provides a better fit to the data.

Finally, note that neither of the algorithms from \cite{milling2015} produce meaningful conclusions.
In the case of the TB algorithm, all positive integers \(d\) lead to a threshold larger than the diameter of the graph, so the algorithm will always reject the null hypothesis. Similarly, the threshold for the TT algorithm is \(779\), which exceeds the weight of any spanning tree since the graph has only 250 vertices. The TB and TT algorithms both fail due to the closely-connected structure of HIV graph, which is a characteristic of many real-world networks, and the inability to select a good threshold.

\subsection{General graph results}

Here, we consider confidence sets for the pairs \((\graphset, \beta)\).
Since our procedure relies on simulating a null distribution for a given pair of parameters, we cannot hope to construct the full confidence set, since enumerating all graphs for even small \(n\) is difficult.
However, we can use our procedure to test whether a particular graph lies in a given \((1 - \alpha)\)-confidence set; we do this for a pair of graphs to demonstrate the applicability of the method.

For the confidence set algorithm, we use Algorithm~\ref{algorithmTwoStatisticConfidenceSet}. 
We  set \(\alpha = 0.2\), with \(\alpha_{0} = \alpha_{1} = \alpha / 2\).
We  set the parameters \(\epsilon = 0.029\), \(\xi = 0.001\), and \(\delta = 0.02\).
For a discussion on choosing parameter values, see Appendix~\ref{app:general_graphs}.
Our two statistics are \(W\) and \(-W\), i.e., we wish to have two-sided confidence sets for the statistic \(W\).
Finally, we restrict to the set \(\paramset = [0, 10]\) in computing the confidence set, since this leads to over \(15,000\) values of \(\beta\) for each of the two graphs we consider.

First, we consider a graph that we call ``graph 25'' or \(\graphset_{25}\), since it is the HIV graph with 25 randomly-chosen edges removed.
A picture is available in the left part of Figure~\ref{fig:G25Double}.

\begin{figure}[!h]
\centering 
\minipage{0.48\textwidth}
\includegraphics[width=\linewidth]{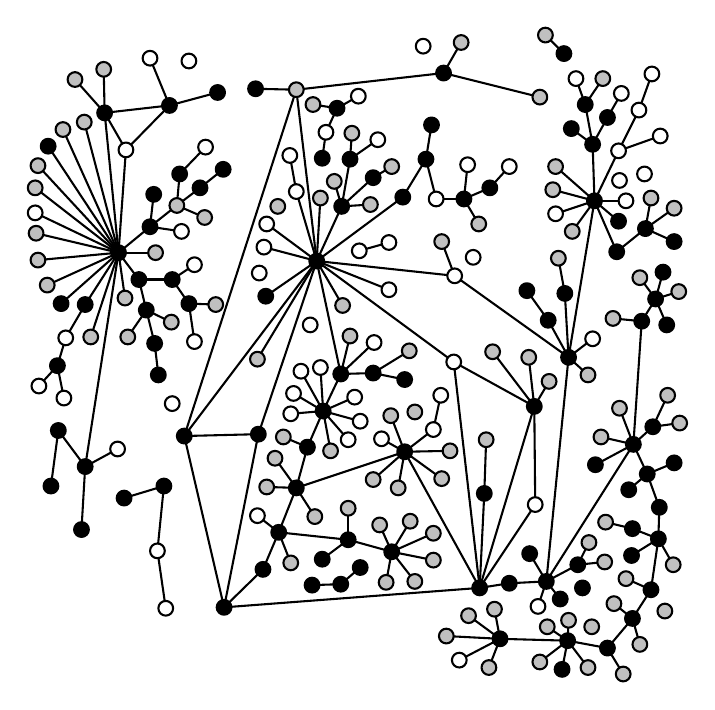}
\endminipage 
\minipage{0.48\textwidth}
\includegraphics[width=\linewidth]{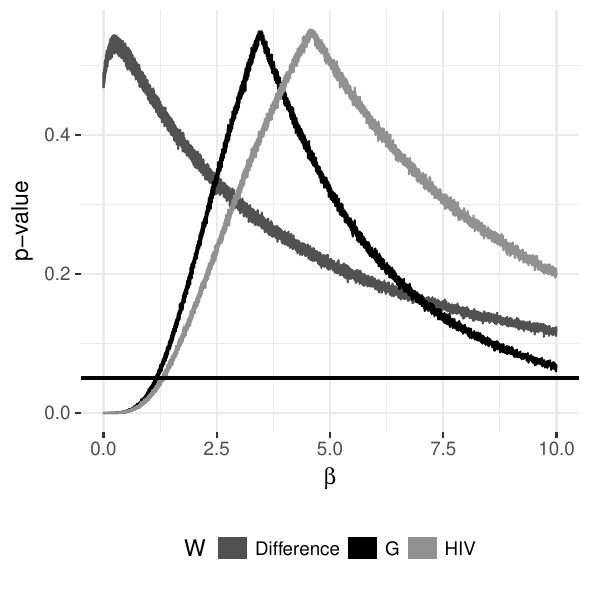}
\endminipage\hfill
\caption{The graph \(\graphset_{25}\) used in computing the joint confidence set of \((\graphset_{25}, \beta)\) and the \(p\)-values for various statistics. 
The graph was formed from the HIV graph by removing 25 edges uniformly at random.
Black vertices are infected, white vertices are uninfected, and gray vertices are censored.
On the right, we have the \(p\)-values as a function of \(\beta\) for the HIV graph using the two-statistic test with different test statistics. 
The statistics are \(W_{\graphset_{25}}\), \(W_{\graphset_{\hiv}}\), and \(W_{\graphset_{25}} - W_{\graphset_{\hiv}}\), and the \(p\)-values for these are given by the black line, the light gray line, and the dark graph line respectively.
Values of \(\beta\) for which a curve is above the horizontal line at \(0.05\) are in the confidence set generated by the corresponding statistic.}
\label{fig:G25Double}
\end{figure}

For illustration, we consider two-sided tests for \(W_{\graphset_{25}}\), \(W_{\graphset_{\hiv}}\), and \(W_{\graphset_{25}} - W_{\graphset_{\hiv}}\), i.e., the edges-within statistic computed with respect to \(\graphset_{25}\), the edges-within statistic computed with respect to the HIV graph, and the difference of these statistics.
The precise algorithm is Algorithm~\ref{algorithmTwoStatisticConfidenceSet}, given in the Appendix.
The results are shown on the right in Figure~\ref{fig:G25Double}.
This figure is to be interpreted as follows: For each value of \(\beta\) in \(\paramset_{D}\), we perform many simulations to estimate the upper and lower percentiles \(\hat{p}_{\graphset, S}(\beta, \infected)\) and \(\hat{p}'_{\graphset, S}(\beta, \infected)\) of the statistic \(S\), where 
\begin{align*}
& \begin{aligned}
\hat{p}_{\graphset, S}(\beta, \infected)
&=
\frac{1}{N_{\sims}} \sum_{i = 1}^{N_{\sims}}
\ind\left\{
S(I_{i}) \leq S(\infected)
\right\} \\
\hat{p}'_{\graphset, S}(\beta, \infected)
&=
\frac{1}{N_{\sims}} \sum_{i = 1}^{N_{\sims}}
\ind\left\{
S(I_{i}) \geq S(\infected)
\right\}
   \end{aligned}
\end{align*}
and \(I_{i}\) are the simulated infections with parameter \(\beta\).
Then, we plotted the minimum percentile
\(q_{S}(\beta) = \min\left\{\hat{p}_{\graphset, S}(\beta, \infected), \hat{p}'_{\graphset, S}(\beta, \infected)\right\}\) against \(\beta\) for three different statistics.
Therefore, the approximate confidence set for \(\beta\) given by \(S\) consists of the \(\beta\) for which the \(q_{S}(\beta)\) is above the horizontal line at \(0.05\).
The precise confidence interval would be obtained by mapping back via \(F_{D}^{-1}\), but rounding after two or three decimal places is sufficient here.

Note that the statistics \(W_{\graphset_{25}}\) and \(W_{\graphset_{\hiv}}\) lead to confidence sets for \(\beta\) of approximately \([1.14, 10]\) and \([1.27, 10]\) respectively.
On the other hand, the difference \(W_{\graphset_{25}} - W_{\graphset_{\hiv}}\) leads to the confidence interval \([0, 10]\). 
This is unsurprising, since the \(\graphset_{25}\) and \(\graphset_{\hiv}\) are not too different. We now repeat this procedure on a different graph, \(\graphset_{100}\), which has \(100\) edges removed from the HIV graph uniformly at random.
Note that \(\graphset_{100}\) is a subgraph of \(\graphset_{25}\). 
An image of the graph and the results of the analysis are shown in Figure~\ref{fig:G100Double}. The major difference for \(\graphset_{100}\) is that no value of \(\beta\) is included in the confidence set for the difference statistic \(W_{\graphset_{100}} - W_{\graphset_{\hiv}}\).
This is likely due to the fact that many edges exist in \(\graphset_{\hiv}\) between infected vertices that do not exist when simulating the distribution using \(\graphset_{100}\).
In short, for the values of \(\beta\) from \([0, 10]\), we have found a non-trivial graph that is not in the confidence set of graphs given by the observed infection.

\begin{figure}[!h]
\centering 
\minipage{0.48\textwidth}
\includegraphics[width=\linewidth]{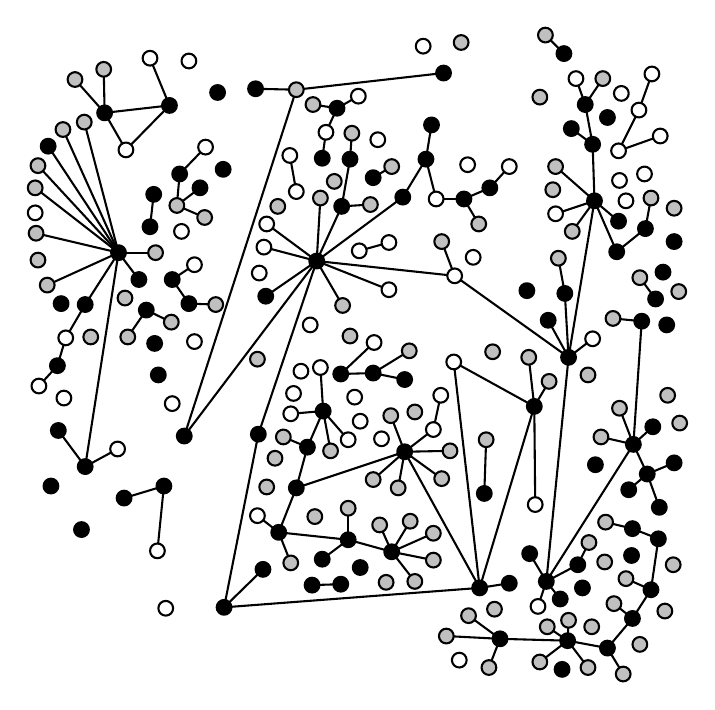}
\endminipage 
\minipage{0.48\textwidth}
\includegraphics[width=\linewidth]{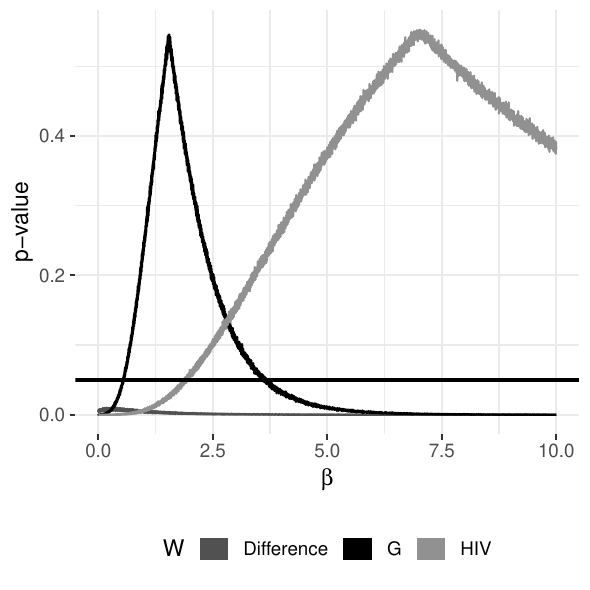}
\endminipage\hfill
\caption{The graph \(\graphset_{100}\) used in computing the joint confidence set of \((\graphset_{100}, \beta)\) and the \(p\)-values for various statistics. 
The graph was formed from the HIV graph by removing 100 edges uniformly at random.
Black vertices are infected, white vertices are uninfected, and gray vertices are censored.
On the right, we have the \(p\)-values as a function of \(\beta\) for the HIV graph using the two-statistic test with different test statistics. 
The statistics are \(W_{\graphset_{100}}\), \(W_{\graphset_{\hiv}}\), and \(W_{\graphset_{25}} - W_{\graphset_{\hiv}}\), and the \(p\)-values for these are given by the black line, the light gray line, and the dark graph line respectively.
Values of \(\beta\) for which a curve is above the horizontal line at \(0.05\) are in the confidence set generated by the corresponding statistic.}
\label{fig:G100Double}
\end{figure}

\section{Discussion}
\label{sec: discuss}

The results presented in our paper suggest several avenues for further research. 
Although we only considered two infection models and closely examined the stochastic spreading model, the permutation testing framework could be extended to other infection models, as well. As an extension of Ising models, one could consider other classes of Markov random fields. Toward network modeling, such tests would also be helpful for testing network structure, such as Erd\H{o}s-Renyi networks versus networks with block structure.
We have left open the question of whether a version of Theorem~\ref{ThmPermStat} holds in more generality. It would also be useful to devise a more interpretable or verifiable sufficient condition for the validity of the permutation test.
Additionally, while our general graph methods can in principle be applied to any graph, it would be computationally prohibitive to do so for larger graphs.
Thus, an open question is how to devise more efficient algorithms for testing general graph structures.


\section*{Acknowledgments}

The authors thank Varun Jog, Dylan Small, and Miklos Racz for helpful comments and discussions. The authors also thank the AE and anonymous reviewers for their many insightful suggestions.


\bibliography{networks}

\begin{thebibliography}{10}

\bibitem{AlbEtal00}
R.~Albert, H.~Jeong, and A.-L. Barab{\'a}si.
\newblock Error and attack tolerance of complex networks.
\newblock {\em Nature}, 406(6794):378--382, 2000.

\bibitem{AndMay92}
R.~M. Anderson and R.~M. May.
\newblock {\em Infectious Diseases of Humans: Dynamics and Control}.
\newblock Dynamics and Control. Oxford University Press, 1992.

\bibitem{angluin2015}
D.~Angluin, J.~Aspnes, and L.~Reyzin.
\newblock Network construction with subgraph connectivity constraints.
\newblock {\em Journal of Combinatorial Optimization}, 29(2):418--432, 2015.

\bibitem{banerjee2018}
D.~Banerjee.
\newblock Contiguity and non-reconstruction results for planted partition
  models: the dense case.
\newblock {\em Electronic Journal of Probability}, 23, 2018.

\bibitem{banerjee2017}
D.~Banerjee and Z.~Ma.
\newblock Optimal hypothesis testing for stochastic block models with growing
  degrees.
\newblock {\em arXiv preprint arXiv:1705.05305}, 2017.

\bibitem{borgs}
C.~Borgs, M.~Brautbar, J.~Chayes, S.~Khanna, and B.~Lucier.
\newblock The power of local information in social networks.
\newblock {\em Proceedings of the 8th International Workshop on Internet and
  Network Economics}, pages 406--419, 2012.

\bibitem{boucheron2013}
S.~Boucheron, G.~Lugosi, and P.~Massart.
\newblock {\em Concentration Inequalities: a Nonasymptotic Theory of
  Independence}.
\newblock Oxford University Press, 2013.

\bibitem{brautbar}
M.~Brautbar and M.~Kearns.
\newblock Local algorithms for finding interesting individuals in large
  networks.
\newblock {\em Innovations in Theoretical Computer Science}, 2010.

\bibitem{bubeck2017}
S.~Bubeck, L.~Devroye, and G.~Lugosi.
\newblock Finding {A}dam in random growing trees.
\newblock {\em Random Structures \& Algorithms}, 50(2):158--172, 2017.

\bibitem{bubeck2016}
S.~Bubeck, J.~Ding, R.~Eldan, and M.~R{\'a}cz.
\newblock Testing for high-dimensional geometry in random graphs.
\newblock {\em Random Structures \& Algorithms}, 49(3):503--532, 2016.

\bibitem{Chen2010a}
W.~Chen, C.~Wang, and Y.~Wang.
\newblock {Scalable influence maximization for prevalent viral marketing in
  large-scale social networks}.
\newblock In {\em Proceedings of the 16th ACM SIGKDD International Conference
  on Knowledge Discovery and Data Mining}, pages 1029--1038. ACM, 2010.

\bibitem{christakis2007}
N.~A. Christakis and J.~H. Fowler.
\newblock The spread of obesity in a large social network over 32 years.
\newblock {\em The New England Journal of Medicine}, 357(4):370--379, 2007.

\bibitem{CohEtal03}
R.~Cohen, S.~Havlin, and D.~Ben-Avraham.
\newblock Efficient immunization strategies for computer networks and
  populations.
\newblock {\em Physical Review Letters}, 91(24):247901, 2003.

\bibitem{cover2012}
T.~M. Cover and J.~A. Thomas.
\newblock {\em Elements of Information Theory}.
\newblock John Wiley \& Sons, 2012.

\bibitem{darga2004}
P.~T. Darga, M.~H. Liffiton, K.~A. Sakallah, and I.~L. Markov.
\newblock Exploiting structure in symmetry detection for cnf.
\newblock In {\em Proceedings of the 41st Annual Design Automation Conference},
  pages 530--534, New York, NY, USA, 2004. ACM.

\bibitem{DomRic01}
P.~Domingos and M.~Richardson.
\newblock Mining the network value of customers.
\newblock In {\em Proceedings of the seventh ACM SIGKDD International
  Conference on Knowledge Discovery and Data Mining}, pages 57--66. ACM, 2001.

\bibitem{DudEtal17}
G.~Dudas et~al.
\newblock Virus genomes reveal the factors that spread and sustained the {W}est
  {A}frican {E}bola epidemic.
\newblock {\em Nature}, 544:309--315, 2017.

\bibitem{dummit2004}
D.~S. Dummit and R.~M. Foote.
\newblock {\em Abstract Algebra}.
\newblock Wiley Hoboken, third edition, 2004.

\bibitem{Fis35}
R.~A. Fisher.
\newblock {\em The Design of Experiments}.
\newblock Oliver and Boyd, 1935.

\bibitem{garey1979}
M.~Garey and D.~Johnson.
\newblock {\em Computers and Intractability: A Guide to the Theory of
  NP-completeness}.
\newblock Books in mathematical series. W. H. Freeman, 1979.

\bibitem{gilbert1968}
E.~Gilbert and H.~Pollak.
\newblock Steiner minimal trees.
\newblock {\em SIAM Journal on Applied Mathematics}, 16(1):1--29, 1968.

\bibitem{godsil2013}
C.~Godsil and G.~F. Royle.
\newblock {\em Algebraic Graph Theory}, volume 207.
\newblock Springer Science \& Business Media, 2011.

\bibitem{goemans1995}
M.~X. Goemans and D.~P. Williamson.
\newblock Improved approximation algorithms for maximum cut and satisfiability
  problems using semidefinite programming.
\newblock {\em Journal of the ACM (JACM)}, 42(6):1115--1145, 1995.

\bibitem{gomez2016}
M.~Gomez-Rodriguez, L.~Song, H.~Daneshmand, and B.~Sch\"{o}lkopf.
\newblock Estimating diffusion networks: Recovery conditions, sample complexity
  and soft-thresholding algorithm.
\newblock {\em Journal of Machine Learning Research}, 17(90):1--29, 2016.

\bibitem{Goo13}
P.~Good.
\newblock {\em Permutation Tests: A Practical Guide to Resampling Methods for
  Testing Hypotheses}.
\newblock Springer Series in Statistics. Springer New York, 2013.

\bibitem{hadlock1975}
F.~Hadlock.
\newblock Finding a maximum cut of a planar graph in polynomial time.
\newblock {\em SIAM Journal on Computing}, 4(3):221--225, 1975.

\bibitem{huang2017}
Y.~Huang, M.~V. Janardhanan, and L.~Reyzin.
\newblock Network construction with ordered constraints.
\newblock {\em arXiv preprint arXiv:1702.07292}, 2017.

\bibitem{Jac08}
M.~O. Jackson.
\newblock {\em Social and Economic Networks}.
\newblock Princeton University Press, Princeton, NJ, USA, 2008.

\bibitem{junttila2007}
T.~Junttila and P.~Kaski.
\newblock Engineering an efficient canonical labeling tool for large and sparse
  graphs.
\newblock In {\em 2007 Proceedings of the Ninth Workshop on Algorithm
  Engineering and Experiments (ALENEX)}, pages 135--149. SIAM, 2007.

\bibitem{karp1972}
R.~M. Karp.
\newblock Reducibility among combinatorial problems.
\newblock In {\em Complexity of Computer Computations}, pages 85--103.
  Springer, 1972.

\bibitem{kempe2003}
D.~Kempe, J.~Kleinberg, and E.~Tardos.
\newblock Maximizing the spread of influence through a social network.
\newblock In {\em Proceedings of the Ninth ACM SIGKDD International Conference
  on Knowledge Discovery and Data Mining}, KDD '03, pages 137--146, New York,
  NY, USA, 2003. ACM.

\bibitem{kesten1993}
H.~Kesten.
\newblock On the speed of convergence in first-passage percolation.
\newblock {\em The Annals of Applied Probability}, pages 296--338, 1993.

\bibitem{lehmann2006}
E.~L. Lehmann and J.~P. Romano.
\newblock {\em Testing Statistical Hypotheses}.
\newblock Springer Science \& Business Media, 2005.

\bibitem{snapnets}
J.~Leskovec and A.~Krevl.
\newblock {SNAP Datasets}: {Stanford} large network dataset collection.
\newblock \url{http://snap.stanford.edu/data}, 2014.

\bibitem{lubiw1981}
A.~Lubiw.
\newblock Some {NP}-complete problems similar to graph isomorphism.
\newblock {\em SIAM Journal on Computing}, 10(1):11--21, 1981.

\bibitem{luks1982}
E.~M. Luks.
\newblock Isomorphism of graphs of bounded valence can be tested in polynomial
  time.
\newblock {\em Journal of Computer and System Sciences}, 25(1):42--65, 1982.

\bibitem{mckay1978}
B.~D. McKay.
\newblock Computing automorphisms and canonical labellings of graphs.
\newblock In {\em Combinatorial Mathematics}, pages 223--232. Springer, 1978.

\bibitem{mehlhorn1988}
K.~Mehlhorn.
\newblock A faster approximation algorithm for the steiner problem in graphs.
\newblock {\em Information Processing Letters}, 27(3):125--128, 1988.

\bibitem{milling2015}
C.~Milling, C.~Caramanis, S.~Mannor, and S.~Shakkottai.
\newblock Distinguishing infections on different graph topologies.
\newblock {\em IEEE Transactions on Information Theory}, 61(6):3100--3120,
  2015.

\bibitem{Mor93}
M.~Morris.
\newblock Epidemiology and social networks: modeling structured diffusion.
\newblock {\em Sociological Methods and Research}, 22(1):99--126, 1993.

\bibitem{myers2012}
S.~A. Myers, C.~Zhu, and J.~Leskovec.
\newblock Information diffusion and external influence in networks.
\newblock In {\em Proceedings of the 18th ACM SIGKDD International Conference
  on Knowledge Discovery and Data Mining}, pages 33--41. ACM, 2012.

\bibitem{netrapalli2012}
P.~Netrapalli and S.~Sanghavi.
\newblock Learning the graph of epidemic cascades.
\newblock In {\em ACM SIGMETRICS Performance Evaluation Review}, volume~40,
  pages 211--222. ACM, 2012.

\bibitem{New02}
M.~E.~J. Newman.
\newblock Spread of epidemic disease on networks.
\newblock {\em Physical Review E}, 66(1), 2002.

\bibitem{PasVes02}
R.~Pastor-Satorras and A.~Vespignani.
\newblock Immunization of complex networks.
\newblock {\em Physical Review E}, 65(3):036104, 2002.

\bibitem{potterat2002}
J.~Potterat, L.~Phillips-Plummer, S.~Muth, R.~Rothenberg, D.~Woodhouse,
  T.~Maldonado-Long, H.~Zimmerman, and J.~Muth.
\newblock Risk network structure in the early epidemic phase of hiv
  transmission in colorado springs.
\newblock {\em Sexually Transmitted Infections}, 78(suppl 1):i159--i163, 2002.

\bibitem{shah2011}
D.~Shah and T.~Zaman.
\newblock Rumors in a network: who's the culprit?
\newblock {\em IEEE Transactions on Information Theory}, 2011.

\bibitem{tsybakov2009}
A.~B. Tsybakov.
\newblock {\em Introduction to Nonparametric Estimation}.
\newblock Springer Series in Statistics. Springer, New York, 2009.

\bibitem{wu2017}
Y.~Wu.
\newblock Lecture notes.
\newblock 2017.

\end{thebibliography}
\newpage
\appendix

\section{Organization}

In this supplement, we provide additional theory, simulation results, and proofs omitted from the main paper. Section~\ref{appFurtherTheory} provides extensions to the theory, including 
(i) a partial converse to the main result of the paper; (ii) an extension to composite null hypotheses, (iii) an extension to composite alternative hypotheses, (iv) an extension to the case of non-uniform censoring; (v) a discussion of computing automorphism groups; (vi) theory for multiple infection processes; (vii) a method for dealing with slight departures from the automorphism condition; and (viii) an extension to randomly-generated graphs. This final part shows that our permutation test is valid for the random graph settings considered in \cite{milling2015}.

In Section~\ref{app:general_graphs}, we discuss the case of discretization-based hypothesis testing for general graphs in more detail. Specifically, we state the discretization of \(\beta\) values necessary to obtain valid tests and confidence intervals. Additionally, we provide some additional analysis for the HIV graph example.

More extensive simulation results are contained in Section~\ref{appSims}.
Simulations were conducted on various graphs to illustrate different phenomena.
The first set of simulations concerns testing vertex-transitive graphs against the star graph to validate our theory. The second set of simulations corresponds to irregular graphs obtained from online social networks, where we use the empty graph as the null graph. The third set of simulations covers two cases of randomly-generated graphs. In the first case, our theory predicts that  permutation test succeeds on average, whereas the second case corresponds to a departure from the conditions required for our theory to succeed, and shows that the permutation test might fail to control the Type I error as edge correlations increase.

Finally, we provide detailed proofs of all theoretical results stated in the paper.
In Section~\ref{appProofs}, we provide proofs of our main theorems and propositions.
In Section~\ref{AppCors}, we prove the corollaries.
We provide proofs of technical lemmas in Section~\ref{AppLemmas}, and we prove results for discretization procedures in Section~\ref{app:general_graphs:proofs}.
We conclude with standard supporting lemmas in Section~\ref{AppAux}.

\section{Further permutation results}
\label{appFurtherTheory}

This section contains additional theoretical results extending the main results of our paper. Proofs are contained in subsequent sections.

\subsection{A partial converse}

A natural question is whether the condition $\Pi_{10} = S_n$ is unnecessarily strong for guaranteeing the theory of our proposed permutation test. We now state and prove a partial converse to Theorem~\ref{ThmPermStat}. 

In particular, it is possible to show (cf.\ Corollary~\ref{cor: vertex transitive} below) that when $\graphset_0$ is a star graph, meaning a graph with $n-1$ edges connecting all nodes to a center node, the condition $\Pi_{10} = S_n$ is equivalent to the graph $\graphset_1$ being vertex-transitive. We have the following result:

\begin{theorem}
Let \(\graphset_{0}\) be a star graph and \(\graphset_{1}\) be a graph that is non-vertex-transitive.
Let \(\pi \sim \text{Uniform}(S_{n})\).
Then there exists a \(\Pi_{1}\)-invariant statistic \(S\) such that \(S(\infected)\) and \(S(\pi \infected)\) do not have the same distribution under $H_0$.
\label{ThmConverse}
\end{theorem}

In other words, if $\Pi_{10} \neq S_n$, it is possible to find a $\Pi_1$-invariant statistic that does \emph{not} satisfy the sufficient condition that we use to derive the validity of the permutation test. We conjecture that a similar converse holds more broadly for general graphs $\graphset_0$; i.e., if $\Pi_1 \Pi_0 \neq S_n$, a $\Pi_1$-invariant statistic $S$ always exists such that $S(\infected) \stackrel{d}{\neq} S(\pi \infected)$ under $H_0$.



\subsection{Composite null hypotheses}

Our first result concerns hypothesis tests involving a composite null hypothesis. In particular, note that the argument in Theorem~\ref{ThmPermTest} applies equally well to a composite null hypothesis, since we only require $\Pi_1 \Pi_0 = S_n$ for all graphs $\graphset_0$ appearing in the null hypothesis. Hence, the permutation test described in Algorithm~\ref{AlgPermExact} controls the Type I error at level $\alpha$ if all graphs $\graphset_0$ in the composite null hypothesis satisfy \(\Pi_{1} \Pi_{0} = S_{n}\).

The following result is an easy variant of Corollary~\ref{CorRisk}, which we state without proof:

\begin{corollary}
Suppose that $S$ is $\Pi_1$-invariant statistic and $\mathcal{C}$ is a collection of graphs such that $\Pi_1 \Aut(\graphset_0) = S_n$ for all $\graphset_0 \in \mathcal{C}$.
The risk of any test based on $S$ is equal to the risk of the same test computed with respect to a single null hypothesis $H_0'$ involving the empty graph.
\label{CorComposite}
\end{corollary}

Such a result may be desirable in cases where there is uncertainty about the exact topology for a particular form of transmission, such as a water-borne illness, and we are testing against a very different transmission mechanism, such as a blood-borne disease.
Testing star graph topologies could realistically arise in practice, for instance to network scientists who study networks from a game-theoretic or economic point of view, where star graphs naturally arise. On the other hand,  the condition $\Pi_1 \Pi_0 = S_n$ could be quite restrictive in other settings.

\subsection{Composite alternative hypotheses}
Suppose one wishes to test the null hypothesis \(H_{0}\): the infection spread on \(\graphset_{0}\), versus a composite alternative \(H_{1}\): the infection spread on some
\(\graphset\) in \(\graphs_{1}\).
For simplicity, we consider the two-graph alternative 
\(\mathfrak{G}_{1} = \{\graphset_{1}, \graphset_{2}\}\)
and refer to the graph automorphism groups as \(\Pi_{1}\) and \(\Pi_{2}\).
If \(\graphset_{1}\) and \(\graphset_{2}\) are not equivalent, 
i.e., there is no \(\pi_{1}\) in \(\Pi_{1}\) such that 
\(\graphset_{2} = \pi_{1} \graphset_{1}\),
we consider
the two-dimensional statistic \((W_{1}, W_{2})\), which is \((\Pi_{1}, \Pi_{2})\)-invariant.
If \(\Pi_{1} \Pi_{0} = \Pi_{2} \Pi_{0} = S_{n}\), both \(W_{1}\) and \(W_{2}\) have the same distribution under the null as if the true graph were empty; however, the two statistics could be correlated.

An exact permutation test is proposed in Algorithm~\ref{algComposite}, with theoretical guarantee in Theorem~\ref{theoremCompositeAlpha}. The idea is to split the Type I error tolerance over the rejection regions of both edges-within statistics.

\begin{algorithm}[!h]
\SetKwInOut{Input}{Input}
\Input{Type I error tolerance $\alpha > 0$, observed infection vector $\infected$}
For each $\pi \in S_n$, compute the statistic $(S_{1}(\pi \infected), S_{2}(\pi \infected))$

For \(i = 1, 2\),  
determine a threshold $t_{\alpha, i}$ such that for \(j < i\),
\begin{equation*}
t_{\alpha, i}
=
\sup\left\{t \in \support(S_{i}):
\frac{1}{n!} \sum_{\pi \in S_n} \ind\{S_{i}(\pi \infected) \geq t_{\alpha, i} \text{ and } S_{j}(\pi \infected) \leq t_{\alpha, j} \} > \frac{\alpha}{2}
\right\}
\end{equation*}

Reject $H_0$ if and only if $S_{i}(\infected) > t_{\alpha, i}$ for either \(i = 1\) or \(i = 2\)
\caption{Permutation test for composite alternatives (exact)}
\label{algComposite}
\end{algorithm}


\begin{theorem}
Let \(\pi \sim \text{Uniform}(S_{n})\).
Consider a hypothesis test of \(\graphset_{0}\) versus the composite alternative \(\graphspace_{1} = \{\graphset_{1}, \graphset_{2}\}\).
Let \(\Pi_{0}\), \(\Pi_{1}\), and \(\Pi_{2}\) be the permutation groups of \(\graphset_{0}\), \(\graphset_{1}\), and \(\graphset_{2}\) respectively,
Let \(S_{1}\) and \(S_{2}\) be \(\Pi_{1}\)-invariant and \(\Pi_{2}\)-invariant statistics.
If we have \(\Pi_{1} \Pi_{0} = \Pi_{2} \Pi_{0} = S_{n}\),
then \(S_{i}(\infected)\) and \(S_{i}(\pi\infected)\) have the same distribution under the null hypothesis for \(i = 1, 2\).
In particular, Algorithm~\ref{algComposite} controls the Type I error at level \(\alpha\). 
\label{theoremCompositeAlpha}
\end{theorem}


\subsection{Conditioning on censoring}
\label{subsecCondCens}

One important variant of the permutation test above involves conditioning on the censored vertices. This is crucial when censoring may not occur uniformly at random, but the stochastic spread still follows our model and occurs independently of the censoring.
For a concrete example, consider the HIV graph in Section~\ref{secHIV} of the main paper.

Accordingly, we devise an alternative permutation testing procedure that conditions on the location of the censored nodes in the network. The natural analog of our permutation test is to simply build a histogram for the values of the test statistic under the subset of random permutations that fix the locations of the censored nodes. With a small abuse of notation, let $S_{n-c}$ denote this set of $(n-c)!$ permutations. The exact and approximate permutation tests are described in Algorithms~\ref{AlgCondCensored} and~\ref{AlgCondApprox}:

\begin{algorithm}[!h]
\SetKwInOut{Input}{Input}
\Input{Type I error tolerance $\alpha > 0$, observed infection vector $\infected$}
For each $\pi \in S_{n - c}$, compute the statistic $S(\pi \infected)$

Determine a threshold $t_\alpha$ such that
\begin{equation*}
t_\alpha
=
\sup\left\{t \in \support(S):
\frac{1}{(n - c)!} 
\sum_{\pi \in S_{n - c}} \ind\{S(\pi \infected) \geq t\} 
> 
\alpha
\right\}
\end{equation*}

Reject $H_0$ if and only if $S(\infected) > t_\alpha$
\caption{Permutation test (exact)}
\label{AlgCondCensored}
\end{algorithm}

\begin{algorithm}[!h]
\SetKwInOut{Input}{Input}
\Input{Type I error tolerance $\alpha > 0$, integer $N_{\sims} \ge 1$, observed infection vector $\infected$}
Draw $\pi_1, \dots, \pi_{N_{\sims}} \stackrel{i.i.d.}{\sim} \text{Uniform}(S_{n - c})$ and compute the statistics $S(\pi_i \infected)$

Determine a threshold $\hat{t}_\alpha$ such that
\begin{equation*}
\hat{t}_{\alpha}
=
\sup\left\{
t \in \support(S):
\frac{1}{N_{\sims}} \sum_{i = 1}^{N_{\sims}} \ind\{S(\pi_i \infected) \geq t\} > \alpha
\right\}
\end{equation*}

Reject $H_0$ if and only if $S(\infected) > \hat{t}_\alpha$
\caption{Permutation test (approximate)}
\label{AlgCondApprox}
\end{algorithm}

The validity of the permutation tests is again stated in terms of automorphism groups of the null and alternative graphs, where we restrict our attention to automorphisms that fix the identity of the censored vertices. We define the set of possible infection vectors with $k$ infected nodes and a fixed censoring pattern as
\[
\infspace_{k, c}^{cc}
=
\{\infected \in \{0, 1\}^{n - c} \times \{\star\}^{c} : \text{\(\infected\) contains exactly \(k\) 1's.}\},
\]
where the ``cc'' stands for ``conditional on the censoring.''
If we let \(\Pi_{0, cc}\) and \(\Pi_{1, cc}\) denote the stabilizer subgroups of \(\infspace_{k, c}^{cc}\) in \(\Pi_{0}\) and \(\Pi_{1}\); i.e.,
\begin{equation*}
\Pi_{i, cc} := \left\{\pi \in \Pi_i: \pi \infspace_{k, c}^{cc} = \infspace_{k, c}^{cc}\right\},
\end{equation*}
then permutations in $\Pi_{i,cc}$ act separately on the sets of censored and uncensored nodes. Thus, we may write \(\Pi_{i, cc} = \Pi_{i, cc}^{u} \times \Pi_{i, cc}^{c}\) as a direct product of subgroups of \(\Pi_{i}\) acting on the uncensored and censored nodes.
The analog of Theorems~\ref{ThmPermStat} and~\ref{ThmPermTest}, guaranteeing the success of the permutation test, is stated in terms of the subgroup \(\Pi_{10}^{u} = \Pi_{1, cc}^{u} \Pi_{0, cc}^{u}\):

\begin{theorem}
Let \(\pi \sim \text{Uniform}(S_{n - c})\).
If \(\Pi_{10}^{u} = S_{n - c}\), then \(S(\infected)\) and \(S(\pi \infected)\) have the same distribution under the null hypothesis (conditioned on the identities of the censored nodes).
In particular, the permutation test in Algorithm~\ref{AlgCondCensored} controls Type I error at level \(\alpha\).
\label{ThmCondCensDist}
\end{theorem}
\noindent The proof closely parallels that of Theorems~\ref{ThmPermStat} and~\ref{ThmPermTest}, so we omit the details.


\subsection{Computational considerations}

Here, we comment on the task of verifying the condition 
\(\Pi_{10} = S_{n}\). As demonstrated in Section~\ref{sec: examples} of the main paper, it is sometimes possible to verify this condition analytically; however, we may need to check the condition computationally for a fixed pair of null and alternative graphs.

The computational complexity of computing the automorphism group of a graph is not known in general, and it is often studied as a reduction of the problem of determining whether two graphs are isomorphic~\cite{lubiw1981}. In the case of graphs of bounded degree, the problem is polynomial~\cite{luks1982}, and a number of algorithms have been proposed that test for nontrivial automorphisms, such as NAUTY~\cite{mckay1978}, SAUCY~\cite{darga2004}, and BLISS~\cite{junttila2007}. Empirical evidence shows that these algorithms also perform reasonably well on moderately-sized graphs. Once $\Pi_0$ and $\Pi_1$ have been computed, we still need to verify that \(\Pi = S_{n}\). 
One approach is to compute
\begin{equation}
|\Pi_{10}|
=
|\Pi_{1} \Pi_{0}|
=
\frac{|\Pi_{1}| |\Pi_{0}|}{|\Pi_{1} \cap \Pi_{0}|},
\label{eqn: pi size}
\end{equation}
and determine whether this expression is equal to \(|S_{n}| = n!\).
Equation~\eqref{eqn: pi size} may be easily verified by considering cosets~\cite{dummit2004}.

\subsection{Multiple infection spreads}
\label{subsec: multiple}

Thus far, we have only discussed the case of observing a single infection spreading vector \(\infected\),
but our results may easily be extended to the case of multiple spreads.
Let \(\infected(1), \ldots, \infected(m)\) be observation vectors from i.i.d.\ infection spreads on $\graphset_0$. Generalizing our earlier framework, we say that a statistic \(S\) is \(\Pi_{1}^{m}\)-\emph{invariant} if
\begin{equation}
S\left(J(1), \ldots, J(m)\right) = S\left(\pi^{(1)} J(1), \ldots, \pi^{(m)} J(m)\right),
\label{eqn: pi1m invariant}
\end{equation}
for any $J(1), \dots, J(m) \in \infspace_{k,c}$ and any permutations \(\pi^{(1)}, \ldots, \pi^{(m)} \in \Pi_{1}\). The proof of the following theorem is analogous to the proof of Theorem~\ref{ThmPermStat}:

\begin{theorem}
Let $\pi^{(1)}, \ldots, \pi^{(m)} \stackrel{i.i.d.}{\sim} \text{Uniform}(S_n)$. If $\Pi = S_n$, then \(S\left(\infected(1), \ldots, \infected(m)\right)\) and \(S\left(\pi^{(1)} \infected(1), \ldots, \pi^{(m)} \infected(m)\right)\) have the same distribution under $H_0$.
\label{ThmMultiple}
\end{theorem}

Note that when multiple spreads are observed on the same graph, one natural approach is to infer the edges of the graph using an appropriate estimation procedure~\cite{gomez2016, netrapalli2012}. On the other hand, Theorem~\ref{ThmMultiple} shows that a permutation test may also be employed in a hypothesis testing framework. The description of the permutation test is identical to Algorithm~\ref{AlgPermExact}, except the statistic $S(\infected)$ is replaced by $S\left(\pi^{(1)} \infected(1), \dots, \pi^{(m)} \infected(m)\right)$, and the average is taken over all $(n!)^m$ possible choices of $\left(\pi^{(1)}, \dots, \pi^{(m)}\right)$. 

A special case of a $\Pi_1^m$-invariant statistic is the average of all edges-within statistics:
\begin{equation}
\label{EqnEdgesMult}
\Wbar\left(\pi^{(1)} \infected(1), \dots, \pi^{(m)} \infected(m)\right) := \frac{1}{m} \sum_{i=1}^m W_{\graphset_1}\left(\pi^{(i)} \infected(i)\right).
\end{equation}
The dependence on $m$ leads to an exponential reduction in the Type II error, due to the concentration of the empirical average of $m$ samples of the edges-within statistic:

\begin{proposition}
Suppose $\graphset_0$ is the star graph and \(\graphset_{1}\) is a connected vertex-transitive graph of degree \(D\).
Let \(\psi_{\overline{W}, \alpha}\) be the level $\alpha$ permutation test based on the average edges-within statistic~\eqref{EqnEdgesMult}. The risk of this test is bounded by
\begin{align*}
& \begin{aligned}
R_{k, 0}(\psi_{\overline{W}, \alpha}, \beta) 
&\leq 
\alpha 
+ \exp\left\{-\frac{2m}{k D^{2}}\left( \frac{D}{2} C_{k} H(\beta)
- \frac{D k(k - 1)}{2 (n - 1)} 
- \sqrt{\frac{kD^2}{2m} \log \frac{1}{\alpha}}\right)^{2}\right\}.
   \end{aligned}
\end{align*}
\label{PropStarRiskMult}
\end{proposition}


We may also obtain analogs of Propositions~\ref{prop: star center} and~\ref{prop: star risk}. The MLE rejects $H_0$ when the average center indicator statistic $\Cbar$ exceeds a threshold. We have the following result:

\begin{proposition}
\label{PropMultRisk}
Suppose $\graphset_0$ is a vertex-transitive graph and \(\graphset_{1}\) is the star graph. 
Let \(\psi_{\Cbar, \alpha}\) be the level \(\alpha\) permutation test based on the average center indicator statistic.
Let
\begin{equation*}
p_{k,0}(\beta) 
:= 
\begin{cases}
\left\{1- \exp\left(-\frac{k + \beta k(k-1)/2}{(n-k+1) + (k-1)\beta}\right)\right\}, & \text{if } \beta \ge 1, \\
\left\{1 - \exp\left(-\frac{k + \beta k(k-1)/2}{n}\right)\right\}, & \text{if } \beta < 1.
\end{cases}
\end{equation*}
Then we have the risk bound
\begin{equation*}
R_{k,0}(\psi_{\Cbar, \alpha}, \beta) 
\le 
\alpha + \exp\left(-2m\left(\frac{k}{n} + \sqrt{\frac{1}{2m} \log\frac{1}{\alpha}} - p_{k,0}(\beta)\right)^2\right).
\end{equation*}
\end{proposition}


\subsection{Relaxing the choice of alternative graph}
\label{subsec: relaxing alt graph}

If $\Pi_1 \Pi_0 \neq S_n$, then $S(\infected)$ and $S(\pi \infected)$ may not have the same distribution under $H_0$, undermining the theory behind Algorithm~\ref{AlgPermExact}. Nonetheless, it may be worthwhile to consider an alternative statistic $S'$ that is $\Pi_1'$-invariant for a subset $\Pi_1' \subseteq S_n$ such that $\Pi_1' \Pi_0 = S_n$. A simple modification of the proof of Theorem~\ref{ThmPermStat} furnishes the following result:

\begin{theorem}
Let \(\Pi_{0} = \Aut(\graphset_0)\),
and let \(\Pi_{1}'\) be a subset of \(S_{n}\) such that $\Pi_1' \Pi_0 = S_n$. 
Let \(\pi \sim \text{Uniform}(S_{n})\).
If \(S'\) is a \(\Pi_{1}'\)-invariant statistic; i.e., $S'(J) = S'(\pi_1'J)$ for all $\pi_1'$ in $\Pi_1'$ and $J$ in $\infspace_{k,c}$, 
then \(S'(\infected)\) and \(S'(\pi \infected)\) have the same distribution under the null hypothesis.
\label{thm: relaxing}
\end{theorem}

In particular, Theorem~\ref{thm: relaxing} establishes that the permutation test in Algorithm~\ref{AlgPermExact} controls the Type I error at level $\alpha$ for any $\Pi_1'$-invariant statistic. 
Theorem~\ref{thm: relaxing} may provide useful guarantees when \(\graphset_{1}\) is close to having $\Pi_1'$ as its automorphism group. 
For instance, we may have $\Pi_1' = \Aut(\graphset_1')$ for some graph $\graphset_1'$ that is only a slight modification of $\graphset_1$. Consider the following example: Let \(\graphset_{0}\) be the star graph on \(n\) vertices, 
and let \(\graphset_{1}\) be the line graph $L_n$ on \(n\) vertices. Clearly, $L_n$ is not vertex-transitive, so $\Pi_1 \Pi_0 \neq S_n$. On the other hand, for large \(n\), the line graph is almost the cycle graph, which we denote by \(C_{n}\). Let \(W_{L_{n}}\) and \(W_{C_{n}}\) be the edges-within statistic on the line graph and the cycle graph.
Then
\begin{align*}
  &\begin{aligned}
W_{C_{n}}(J) & = W_{L_n}(J) + \ind\{J_n = J_1 = 1\}. 
  \end{aligned}
\end{align*}
As a result, these statistics are quite similar, so the risk of a permutation test based on the statistic \(W_{C_{n}}\) is also similar. We have the following result:

\begin{corollary}
Let $\graphset_1 = L_n$, and let $\psi_{W, \alpha}$ be the level $\alpha$ permutation test based on $W_{C_n}$. Suppose $k < n / 2$. Then
\begin{multline*}
R_{k, 0}(\psi_{W,\alpha}, \beta) 
\le 
\alpha +  \exp\Bigg\{-\frac{1}{2k} \Bigg(\frac{1}{n} \cdot (k-1) 2^{k-1} \cdot \frac{n-k+1}{n} \\
\cdot \prod_{m=1}^{k-1} \frac{\beta}{n-m+2\beta}
- \frac{k(k-1)}{n-1} - \sqrt{2k\log\left(\frac{1}{\alpha}\right)}\Bigg)^{2}\Bigg\}.
\end{multline*}
\label{CorStarRisk2}
\end{corollary}

Compared with Corollary~\ref{CorStarRisk0}, the expression in Corollary~\ref{CorStarRisk2} only contains an additional factor of 
$(n - k + 1) / n$ in the first term of the exponent.


\subsection{Random graphs}
\label{AppRandomGraphs}

For our final extension, we turn to random graphs.
We first introduce some additional notation.

Throughout this section, consider \(\graphset_{0}\) and \(\graphset_{1}\) to be random variables, and let \(\Pi_{0}''\) and \(\Pi_{1}''\) be the subgroups of permutations such that for any fixed graph \(G\) and permutation \(\pi_{i}\) in \(\Pi_{i}''\), we have
\begin{equation}
\prob(\graphset_{i} = G)
=
\prob(\pi_{i} \graphset_{i} = G).
\label{eqnRandomGraphPerm}
\end{equation}
As an example, an Erd\H{o}s-Renyi random graph satisfies equation~\eqref{eqnRandomGraphPerm} with \(\Pi_{i}'' = S_{n}\).
Let \(S(\infected, \graphset_{1})\) be a statistic, where we explicitly include the graph dependence.
We are most interested in \(W(\infected, \graphset_{1})\), which we use to denote the usual edges-within statistic on the random graph \(\graphset_{1}\).
We have the following lemma:

\begin{lemma}
Suppose that under \(H_{0}\), the random variables \(\infected\) and \(\graphset_{1}\) satisfy
\begin{equation}
\prob(S(\infected, \graphset_{1}) \in \mathcal{B})
=
\prob(S(\pi \infected, \graphset_{1}) \in \mathcal{B}),
\label{eqnRandomGraphsCondition}
\end{equation}
for all \(\pi\) in \(S_{n}\).
Then we have 
\begin{equation}
\prob(S(\infected, \graphset_{1}) \in \mathcal{B})
=
\frac{1}{n!}\sum_{\pi \in S_{n}}
\prob(S(\pi \infected, \graphset_{1}) \in \mathcal{B}).
\label{eqnRandomGraphsConclusion}
\end{equation}
\label{lemmaRandomGraphs}
\end{lemma}

The proof is obtained by averaging over \(S_{n}\).

\begin{remark}
The condition in equation~\eqref{eqnRandomGraphsCondition} is satisfied in a number of cases.
One example is the following: if \(\graphset_{0}\) and \(\graphset_{1}\) are drawn independently and the former is an Erd\H{o}s-Renyi random graph, we have
\begin{align*}
& \begin{aligned}
\prob((\infected, \graphset_{1}) = (J, G))
&=
\prob(\infected = J) \prob(\graphset = G) \\ 
&=
\prob(\pi \infected = J) \prob(\graphset = G) \\ 
&=
\prob((\pi \infected, \graphset_{1}) = (J, G)),
\end{aligned}
\end{align*}
for any \(\pi\) in \(S_{n}\).
The first and last equalities are due to independence, and the second is the result of the \(S_{n}\)-invariance of the Erd\H{o}s-Renyi random graph.

Another example satisfying equation~\eqref{eqnRandomGraphsCondition} is the case where \(S = W\) and the topology of the alternative \(\graphset_{1}\) is fixed, but the labeling of the vertices is uniformly random, such as in \cite{milling2015}.
In this case, we have
\begin{align*}
& \begin{aligned}
\prob(W(\infected, \graphset_{1}) \in \mathcal{B})
&=
\prob(W(\infected, \pi \graphset_{1}) \in \mathcal{B})  
=
\prob(W(\pi \infected, \graphset_{1}) \in \mathcal{B}),
\end{aligned}
\end{align*}
where we have used the fact that \(\Pi_{1}'' = S_{n}\) and the equality \(W(\infected, \pi \graphset_{1}) = W(\pi \infected,  \graphset_{1})\).
Additionally, we have assumed the independence of \(\graphset_{0}\) and \(\graphset_{1}\).
\end{remark}

Turning to the conclusion of Lemma~\ref{lemmaRandomGraphs}, note that equation~\eqref{eqnRandomGraphsConclusion} implies that the permutation test controls the Type I error on average with respect to draws of \(\graphset_{1}\).
In the special case where the topology of \(\graphset_{1}\) is fixed---or alternatively, \(\graphset_{1}\) is deterministic up to a permutation---equation~\eqref{eqnRandomGraphsConclusion} implies that the permutation test does control the Type I error.
To see this, suppose that \(\graphset_{1}\) is chosen from \(\graphspace = S_{n}\{G^{\star}\}\) for some fixed \(G^{\star}\), and let \(S\) be a statistic that is invariant in the sense that \(S(\infected, G) = S(\pi \infected, \pi G)\).
We have
\begin{align*}
& \begin{aligned}
\prob(S(\infected, \graphset_{1}) \in \mathcal{B})
&=
\frac{1}{n!} \sum_{\pi \in S_{n}} \prob(S(\pi \infected, \graphset_{1}) \in \mathcal{B}) \\
&=
\frac{1}{n!} \sum_{\pi \in S_{n}} \sum_{G \in \graphspace} 
\prob(\graphset_{1} = G) \prob(S(\pi \infected, G) \in \mathcal{B}) \\ 
&=
\frac{1}{n!} \sum_{\pi \in S_{n}} \sum_{G \in \graphspace} 
\prob(\graphset_{1} = G) \prob(S( \pi \infected, \pi' G^{\star}) \in \mathcal{B}) \\
&=
\frac{1}{n!} \sum_{\pi'' \in S_{n}} \sum_{G \in \graphspace} 
\prob(\graphset_{1} = G) \prob(S(\pi'' \infected, G^{\star}) \in \mathcal{B}) \\
&=
\frac{1}{n!} \sum_{\pi'' \in S_{n}} 
\prob(S(\pi'' \infected, G^{\star}) \in \mathcal{B})
\sum_{G \in \graphspace} 
\prob(\graphset_{1} = G)  \\
&=
\frac{1}{n!} \sum_{\pi'' \in S_{n}} 
\prob(S(\pi'' \infected, G^{\star}) \in \mathcal{B}).
   \end{aligned}
\end{align*}
Note that the second equality is by independence, the third equality is by the deterministic topology, and the fourth equality is by the invariance of \(S\) described above.
Thus, the permutation test indeed controls the Type I error in this case.

\section{Details for general graphs}
\label{app:general_graphs}

In this appendix, we detail the discretization that we need to test general graphs.
The overall idea is straightforward.
We have a space of probability measures \(\measuremap(\graphset, \completespace)\) where \(\prob_{i, \beta} = \measuremap(\graphset_{i}, \beta)\) and $\completespace = [0, \infty]$. 
We wish to find a \(\delta\)-net in total variation distance for the set of distributions \(\measuremap(\graphset, \paramset) \subseteq \measuremap(\graphset, \completespace)\),
corresponding to a finite subset \(\paramset_{D} \subset \paramset\).
Using this finite subset, we can approximately simulate the possible null distributions to conduct hypothesis tests and build confidence intervals.

In order to construct a discretization \(\paramset_{D}\) for a given \(\paramset\), our strategy is to first construct a finite set \(\completespace_{D}\) of \(\completespace\) such that \(\measuremap(\graphset, \completespace_{D})\) yields a \(\delta\)-net of \(\measuremap(\graphset, \completespace)\).
Then, for a discretization function \(F_{D}\) satisfying equation~\eqref{eqnDiscretizationFunction},
we define the discretization set \(\paramset_{D} := F_{D}(\paramset)\).

The first question one might have is whether such a \(\completespace_{D}\) exists.
While our construction proves that it does, we can also reason that this must be true from a topological perspective.
Using the standard topology associated with \(\completespace\) and the topology on \(\measuremap(\graphset, \completespace)\) induced by the total variation distance as a metric, the compactness of \(\completespace\) and the continuity of the function \(\measuremap(\graphset, \cdot)\) imply that \(\measurespace\) must also be compact, which is certainly sufficient for the existence of a \(\delta\)-net. 

Given the existence of a cover \(\completespace_{D}\), one might wonder why
we cannot simply define $\paramset_D = \completespace_{D}$. 
Such a definition would certainly yield a valid test; however, to obtain a more powerful test, it is beneficial to eliminate superfluous points of \(\completespace_{D}\) when forming \(\paramset_{D}\).
In particular, it is important to eliminate small values of \(\beta\) if possible, since such values of \(\beta\) produce measures \(\measuremap(\graphset, \beta) = \prob _{\beta}\) that do not depend strongly on \(\graphset\).
From a computational viewpoint, eliminating superfluous values of \(\beta\) allows us to conduct our tests with less computation.

In the remaining subsections, we introduce the discretization that our general graph algorithms require; we discuss computation; we suggest algorithms that allow us to more closely approximate the likelihood ratio; and we give further details on the HIV graph example.

\subsection{Discretization}
\label{subsecDiscretization}

We provide precise statements and proofs with one additional level of generality.
Specifically, we define a function \(w: \edgeset \to \reals_{+}\) to be an edge weight function, and for adjacent vertices \(u\) and \(v\), we define the random variable on the edge $(u,v)$ to be \(T_{uv} \sim \text{Exp}(\beta w(u, v))\) in the stochastic spreading model.
In the Ising model, we define
\[
\energy(J) 
=
-\sum_{(u, v) \in \edgeset} w(u, v) J_{u} J_{v}.
\]
The only constraint is that \(w(u, v) \geq 1\) for each edge \((u, v)\).
Thus, the models in Section~\ref{SecInfectModels} correspond to the special case where \(w\) is \(1\).
Finally, we use the notation \(\wneighbors_{t, v}\) and \(\wneighbors_{t}\) to denote the weighted infected in-degree for vertex \(v\) and the weighted cut between infected and uninfected nodes at time $t$, respectively. Accordingly, for a path $P_{1:k} := (P_1, \dots, P_k)$, we write $W_t(P_{1:k})$ to be the weighted cut when $\{P_1, \dots, P_k\}$ are infected.

We now define the following discretization with its associated discretization function:

\begin{definition}
Let \(\graphset\) be a weighted graph.
Let
\[
M_{0} 
=
\sum_{t = 1}^{k + c}
\max_{P_{1:t - 1}}
\frac{\wneighbors_{t}(P_{1:t - 1})}{n + 1 - t}.
\]
Define \(M_{1}\)  to be an integer satisfying
\[
M_{1}
\geq
\frac{1}{2 \delta}
M_{0}
.
\]
Define \(M_{2}\) to be an integer such that 
\[
M_{2}
\geq
\frac{k + c}{2\delta}. 
\]
Let \(N_{0}\) be an integer such that
\[
N_{0}
\geq 
(k + c)\left(n - \frac{k + c - 1}{2} \right).
\]
Let \(N_{1}\) be an integer satisfying
\[
N_{1} 
\geq 
\frac{k + c}{2 \delta}
\left(
\log
\frac{N_{0}}{\delta}
+
\log
\frac{M_{1}}{M_{2}}
\right).
\]
We define the one-dimensional discretization 
\(\completespace_{D}\)
to include the following values of \(\beta\):
\begin{itemize}
\item 
\(\beta_{i} = \frac{i}{M_{1}}\), for \(i = 0, 1, \ldots, M_{2}\), and

\item 
\(\beta_{M_{2} + i} 
= 
\beta_{M_{2} + i - 1} 
\exp\left(\frac{2\delta}{k + c}\right)\), for \(i = 1, \ldots, N_{1}\).
\end{itemize}
Finally, the discretization function \(F_{D}\) associated with this discretization is defined by
\begin{align*}
& \begin{aligned}
F_{D}(\beta) 
&:=
\begin{cases}
\argmin_{\beta' \in \completespace_{D}} |\beta - \beta'|, 
& \text{for } \beta \leq \frac{M_{2}}{M_{1}}, \\
\max\left\{\beta' \in \completespace_D: \beta' \leq \beta  \leq \exp\left(\frac{\delta}{k + c}\right)\beta' \right\} & \mbox{} \\
\quad \text{ or }
\min\left\{\beta' \in \completespace_D: \exp\left(-\frac{\delta}{k + c}\right)\beta' < \beta  \leq \beta'\right\},
& \text{for } \frac{M_{2}}{M_{1}} < \beta \leq \frac{N_{0}}{\delta}, \\
\max\{\beta' \in \completespace_{D}\}, & \text{for } \frac{N_{0}}{\delta} < \beta.
\end{cases}
   \end{aligned}
\end{align*}
\label{defPinskerCouplingDiscretization}
\end{definition}

Note that for  \(\beta\) satisfying \(M_{2} / M_{1} < \beta \leq N_{0}/\delta\), one of the sets must be nonempty, so \(\beta'\) is well-defined.
For the graph \(\graphset_{i}\), we denote the corresponding discretization set by \(\completespace_{i, D}\).
Additionally, we denote the size of the set by \(N = |\completespace_{D}| = M_{2} + N_{1} + 1\).
We have the following proposition:

\begin{proposition}
For any graph \(\graphset\), the discretization \(\completespace_{D}\) with discretization function \(F_{D}\)  satisfies equation~\eqref{eqnDiscretizationFunction}.
\label{propDiscretizationMain}
\end{proposition}

The proof is in Appendix~\ref{SecDiscProof}.

\subsection{Computation}
\label{subsecComputation}

We now comment on the computational aspects of Algorithm~\ref{algorithmPCDiscretizationTest}, noting that the computation for Algorithm~\ref{algorithmOneStatisticConfidenceSet} is identical.
To run this algorithm, we first need to compute \(M_{0}\).
Each summand yields a cardinality-constrained weighted max-cut problem.
In general, the associated unconstrained decision problem is NP-complete \cite{karp1972}, but there are polynomial-time approximation algorithms \cite{goemans1995} and polynomial time algorithms for planar graphs \cite{hadlock1975}.
Alternatively, we can use the sum of the largest \(t - 1\) weighted degrees to bound each \(\wneighbors_{t}(P_{1:t-1})\). 
A good, computable choice is the minimum of this crude degree bound and a scaled approximate solution of the maxcut problem, i.e., the value of the approximate maxcut solution divided by the approximation factor.

Next, we need to run \(N \times N_{\sims}\) simulations.
Using a simple weighted degree upper bound, we can give an upper bound on the number of simulations required.
Let \(D\) be the maximum weighted degree in the graph.
Clearly, it suffices to have
\[
N
\geq 
1 
+ 
\frac{k + c}{2\delta}
\left(1 + 
\log \frac{N_{0}}{\delta}
+  \log \frac{D}{k + c} \cdot \sum_{t = 1}^{k + c} \frac{t-1}{n + 1 - t}
\right).
\]
Fortunately, this only depends logarithmically on the size of the graph \(n\), when $k$ and $c$ are fixed.
Additionally, since the order of the simulations does not matter, each of them can be run in parallel.
However, even for graphs of moderate size, this may be too many simulations for testing the entirety of \(\completespace\).
Thus, it may be necessary to restrict our consideration to a smaller parameter set \(\paramset\).
If we wish to restrict ourselves to the range \([0, R]\) for \(R > M_{2} / M_{1}\), 
then one could replace the \(N_{0} / \delta\) in the logarithm above with \(R\) to obtain an upper bond on the number of simulations.

Finally, we make a brief remark on choosing \(\delta\), \(\epsilon\), and \(\xi\).
All of these make the approximate threshold worse, so from a statistical perspective, we would like to minimize each of these parameters.
From a computational perspective, however, we would like all of these parameters to be large.
This tradeoff means that we have to select how to allocate a statistical approximation budget between the three parameters or, alternatively, decide how small we can make each parameter based on computational resources.
The important point is that while all parameters affect the statistical approximation of the threshold equally, they do not affect the computation equally.
As a rough approximation of \(N \times N_{\sims}\), the dependence on \(\delta\) is \(\delta^{-1}\); the dependence on \(\epsilon\) is \(\epsilon^{-2}\); and the dependence on \(\xi\) is \(\log \xi^{-1}\).
Thus, a sensible computational approach may be to make these terms roughly equal, i.e., set \(\delta^{-1} \approx \epsilon^{-2} \approx \log \xi^{-1}\).

\subsection{Likelihood-based tests}
\label{subsecLikelihoodBasedTests}

The setup that we have discussed so far has let us consider statistics such as \(W_{1} - c W_{0}\) for some constant \(c\), where $W_1$ and $W_0$ are the respective edges-within statistics.
However, suppose we wish to consider the log-likelihood ratio statistic
\begin{align*}
& \begin{aligned}
\ell(\beta_{0}, \beta_{1}; J) 
&=
\log L(\graphset_{1}, \beta_{1}; J)
- 
\log L(\graphset_{0}, \beta_{0}; J) 
=
\ell_{1}(\beta_{1}; J) - \ell_{0}(\beta_{0}; J). 
\end{aligned}
\end{align*}
Since the log-likelihood ratio is difficult to compute for our stochastic spreading model, we might also be interested in approximations of the form
\begin{equation}
S(\beta_{0}, \beta_{1}; J)
=
S_{1}(\beta_{1}; J) - S_{0}(\beta_{0}; J),
\label{eqnParameterDependentStatistic}
\end{equation}
where \(S_{i}\) is a computable approximation of \(\ell_{i}\).
In particular, we saw in Section~\ref{subsecLikelihoodRatio} that 
\[
S_{W}(\beta_{0}, \beta_{1}; J)
=
\beta_{1} W_{1}(I) - \beta_{0}W_{0}(I)
\]
is the first-order approximation of the log-likelihood ratio when \(c = 0\).

However, the statistic $S_W$ is now
a function of \(\infected\) \emph{and} the parameters \(\beta_{0}\) and \(\beta_{1}\).
We propose two possible solutions.
First, instead of finding a threshold \(\hat{t}\) to compare \(S(I)\) against, we find a function \(\hat{t}: \paramset_{0} \times \paramset_{1} \to \reals\) such that \(\hat{t}(\beta_{0}, \beta_{1}) \leq S(I; \beta_{0}, \beta_{1})\) with probability at most \(\alpha\) under the null hypothesis across values of \(\beta_{0}\) and \(\beta_{1}\).
The second option is to skirt dependencies on the parameter value in the statistic.
For example, using \(S_{W}\) as inspiration, we might consider a class of linear statistics
\begin{equation}
S(\beta_{0}, \beta_{1}; J)
=
\beta_{1} S_{1}(I) - \beta_{0} S_{0}(I).
\label{eqnLinearStatistics}
\end{equation}
Then by considering \(S_{0}\) and \(S_{1}\) separately, we can reject the null hypothesis if for each value of \(\beta_{0}\), either \(S_{0}\) is too small or \(S_{1}\) is too large.

Since the threshold function method is more cumbersome, we defer it to Appendix~\ref{SecThreshold}.
Now, we consider the second approach: two-statistic tests.
Let \(s_{0, 1 - \alpha, \beta}\) and \(s_{1, \alpha, \beta}\) denote the \(\alpha\) and \(1 - \alpha\) quantiles of \(S_{0}\) and \(S_{1}\) respectively, and let \(\hat{s}_{0, 1 - \alpha, \beta}\) and \(\hat{s}_{1, \alpha, \beta}\) denote their empirical counterparts.
The method is summarized in Algorithm~\ref{algorithmTwoStatistic}.

\begin{algorithm}[!h]
\SetKwInOut{Input}{Input}
\Input{Type I error tolerances \(\alpha_{0}\) and \(\alpha_{1}\) where $\alpha_{0} + \alpha_{1} = \alpha$, approximation parameters \(\epsilon\), \(\delta\), and \(\xi\), observed infection vector $\infected$, null graph \(\graphset_{0}\), statistics \(S_{0}\) and \(S_{1}\), discretization \(\paramset_{0, D}\), discretization function \(F_{D}\)}

Define 
\(N_{\sims} \geq \left(\frac{1}{2\epsilon^{2}} + \frac{8}{3 \epsilon}\right) \log \frac{1}{\xi}\)

For each \(\beta\) in \(\paramset_{0, D}\), simulate an infection \(N_{\sims}\) times on \(\graphset_{0}\) to obtain the approximate the \((\alpha_{0} - \gamma - \epsilon)\) quantile \(\hat{s}_{0, 1 - \alpha_{0} - \delta - \xi - \epsilon, \beta}\) and the \(1 - (\alpha_{1} - \delta - \xi - \epsilon)\) quantile \(\hat{s}_{1, (\alpha_{1} - \delta - \xi - \epsilon), \beta}\)

Reject the null hypothesis if for each \(\beta \in \paramset_{0, D}\), either
\(S_{0} < \hat{s}_{0, 1 - (\alpha_{0} - \delta - \xi - \epsilon), \beta}\)
or 
\(S_{1} > \hat{s}_{1, (\alpha_{1} - \delta - \xi - \epsilon), \beta}\)

\caption{Two-Statistic Test}
\label{algorithmTwoStatistic}
\end{algorithm}
\begin{proposition}
If the discretization \(\paramset_{D}\) with discretization function \(F_{D}\) satisfies equation~\eqref{eqnDiscretizationFunction}, then
Algorithm~\ref{algorithmTwoStatistic} controls the Type I error at a level \(\alpha\).
\label{propTwoStatistic}
\end{proposition}

The proof is in Appendix~\ref{SecTwoStatistic}. We explore the performance of Algorithm~\ref{algorithmTwoStatistic} in a simulation study in the next subsection.


\subsection{HIV graph details}

In this section, we discuss additional details for our HIV graph analysis as it relates to our methods for general graph tests.
First, we comment on the computation of the upper bound \(M_{0} \leq 348.8\).
We upper-bounded each \(\max_{P_{1:t - 1}} W_{t}(P_{1:t - 1})\) by the minimum of the maxcut value \(258\) and the sum of the degrees of the first \(t - 1\) infected vertices.
The bounds on \(M_{0}\) for \(\graphset_{25}\) and \(\graphset_{100}\) were computed analogously.
The maxcut can be seen in Figure~\ref{fig:disc_numerical:hivCut}.

\begin{figure}[!htb]
\centering
  \includegraphics[width=3in]{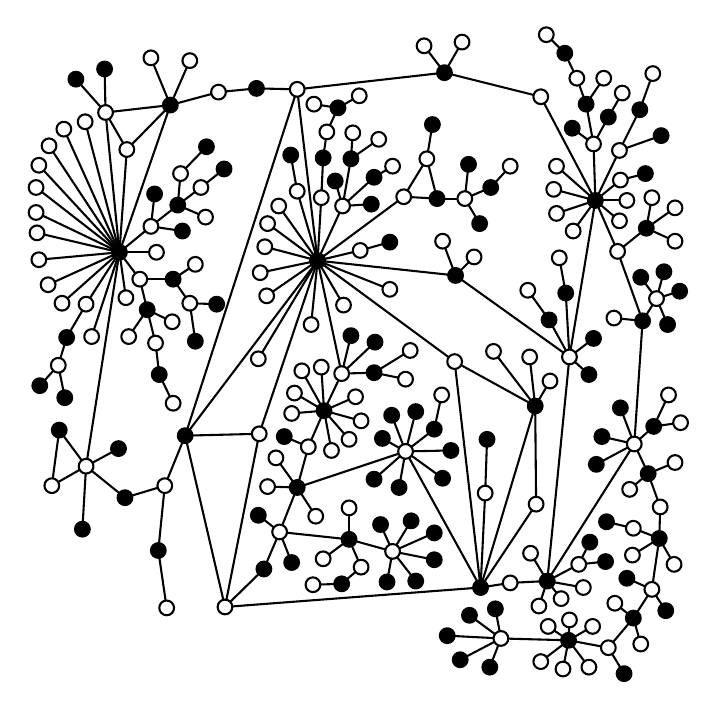}
\caption[The HIV Maxcut Partition]{The HIV  maxcut partition. 
The maximum cut consists of the sets of black and white vertices; the cut-set consists of the edges between the black and white vertices.
}
\label{fig:disc_numerical:hivCut}
\end{figure}



Next, we consider a confidence set of \(\beta\), assuming that the HIV graph is the true graph.
We apply the two-statistic test of Algorithm~\ref{algorithmTwoStatisticConfidenceSet} with the same algorithm parameters.
The resulting \(p\)-values are plotted as a function of \(\beta\) in Figure~\ref{fig:disc_numerical:betaPValuesHIV}.

\begin{figure}[!htb]
\centering
  \includegraphics[width=3.5in]{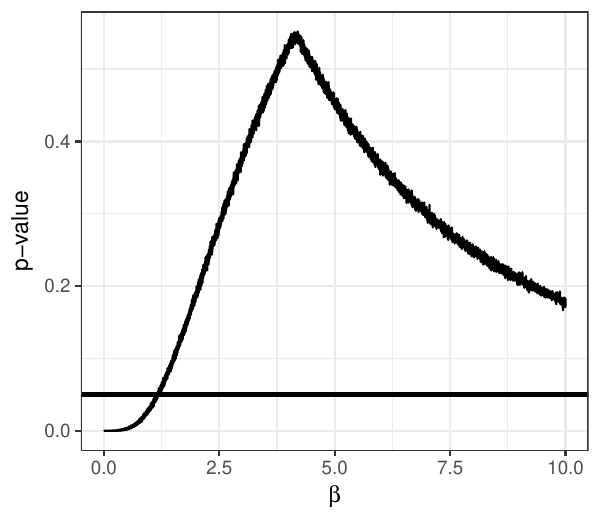}
\caption[\(p\)-values for the HIV Graph]{The \(p\)-values as a function of \(\beta\) for the HIV graph using the two-statistic test. 
The two statistics are \(W\) and \(-W\), i.e., the \(p\)-value is small if \(W\) is either too big or too small. 
Values of \(\beta\) above the horizontal line at \(0.05\) are in the \(0.8\)-confidence set.
}
\label{fig:disc_numerical:betaPValuesHIV}
\end{figure}

If we consider a confidence interval formed by the smallest and largest values of \(\beta\) for which the \(p\)-value is greater than \(0.05\), i.e., ignoring fluctuations around \(0.05\), the resulting interval is \([1.14, 10]\).
Thus, we have evidence for spreading on the HIV graph for \(\beta\) in this interval.

\section{Simulations}
\label{appSims}

This section reports the results of additional simulations we performed to validate our permutation test. Our main interest was to explore the power of our test on various graphs. We compared the algorithms proposed in \cite{milling2015}, computed with respect to the following test statistics:
\begin{enumerate}
\item The \emph{infection radius} $R(\infected)$, defined by
\begin{equation*}
\label{EqnRadius}
R_{\graphset}(\infected) 
:= 
\min_{v \in \vertexset} \left\{\max_{u \in \infected} d_{\graphset}(u, v)\right\},
\end{equation*}
where \(d_{\graphset}(u, v)\) is the length of the shortest path connecting \(u\) to \(v\) in $\graphset$.
\item The \emph{Steiner minimal tree}, defined for a subset of vertices \(M\) in a connected graph \(\graphset\) to be a subtree of minimal edge weight containing all vertices in \(M\) \cite{gilbert1968}. Here, $M$ is the set of infected vertices and each edge has unit weight.

We are interested in a hypothesis test computed with respect to the edge weight of the Steiner minimal tree. However, finding the Steiner minimal tree is NP-complete \cite{garey1979}, so the statistic used for simulations is based on an approximate Steiner minimal tree, computed via a tractable algorithm due to \cite{mehlhorn1988}. Let \(T(\infected)\) be the edge weight of the approximate Steiner minimal tree.
\end{enumerate}

The algorithms of \cite{milling2015} reject the null hypothesis if the statistic \(R(\infected)\) or \(T(\infected)\), evaluated on \(\graphset_{1}\), exceeds a certain threshold.
The corresponding tests are called the Threshold Ball (TB) and Threshold Tree (TT) algorithms. The threshold for the TB algorithm is computed theoretically to be $1.1d^{2} \left(kn \log( \log n) / (n - c) \right)^{\frac{1}{d}}$, when $\graphset_1$ is a toroidal grid of dimension \(d\). For non-toroidal graphs, no threshold is provided, so our approach was to try to select a value of $d$ appropriate to the threshold calculation. In the case of the TT algorithm, the threshold suggested by \cite{milling2015} is $k n (\log \log n)^{3} / (n - c)$.

Note that both $R$ and $T$ are defined in relation to the topology of $\graphset_1$, so they are $\Pi_1$-invariant. Thus, we may apply Algorithm~\ref{AlgPermApprox} and compare the results of the permutation test directly to the thresholding algorithms proposed by \cite{milling2015}. We write \(\alg_{2}(S, B, \alpha)\) to denote the result of applying Algorithm~\ref{AlgPermApprox} with \(\Pi_{1}\)-invariant statistic \(S\), based on \(B\) randomly drawn permutations, with Type I error level \(\alpha\). We write \(\alg_{TB}(d)\) and \(\alg_{TT}\) to denote the TB algorithm with parameter \(d\) and the TT algorithms, respectively. 

\subsection{Vertex-transitive graphs}
\label{subsecRegGraph}

For the first set of simulations, we let \(\graphset_{0}\) be the empty graph and \(\graphset_{1}\) the 2-dimensional toroidal grid on \(n\) vertices.
We set \(n \in \{2500, 10000\}\) and \(k = c = 500\).
The results are based on 1000 simulations, for $\lambda = 1$ and $\beta \in \{1, 10, 100, 1000\}$.
Table~\ref{tableGrid2D} contains the numerical results.

Algorithm~\ref{AlgPermApprox} with the edges-within statistic \(W\) performed significantly better than all other algorithms in terms of Type II error.
Note that the Type II error for Algorithm~\ref{AlgPermApprox} with statistics \(R\) and \(T\) decreases as $\beta$ increases. On the other hand, the TB and TT algorithms are uninformative for both grid sizes, since the thresholds set by the algorithms result in a rejection rule that never rejects the null hypothesis. This emphasizes an important feature of our permutation test, which always provides Type I error control, in contrast to the methods of \cite{milling2015}, which only guarantee that the Type I error will converge to 0 as $n$ grows. Further note that our simulations involve a parameter $\lambda > 0$, which is not covered by the theory of \cite{milling2015}. On the other hand, for large $\beta$, most infected nodes should contract the disease from a neighbor rather than exogenous sources.

\begin{table*}\centering
\ra{1.3}
\begin{tabular}{@{}llllllll@{}}\toprule
Size & Algorithm  & Threshold & Type I Error & \multicolumn{4}{c}{Type II Error by \(\beta\)} \\
\cmidrule{1-8}
&&&
& \(1\) & \(10\) & \(100\) & \(1000\)\\ \midrule
\(n = 2,500\) 
&\(\alg_{2}(W, \; 100, \; 0.01)\)  
& \(220\) & \(0.028\)& \(0.000\) & \(0.000\) & \(0.000\) & \(0.000\) \\
&\(\alg_{2}(R, \; 100, \; 0.01)\)  
& \(45\)& \(0.150\) & \(0.517\) & \(0.000\) & \(0.000\) & \(0.000\) \\
&\(\alg_{2}(T, \; 100, \; 0.01)\)  
&\(1186\)& \(0.009\) & \(0.858\) & \(0.050\) & \(0.000\) & \(0.000\) \\
&\(\alg_{TB}(2)\)  
& \(157\) & \(1.000\) & \(0.000\) & \(0.000\) & \(0.000\) & \(0.000\) \\
&\(\alg_{TT}\)  
& \(5441\) & \(1.000\) & \(0.000\) & \(0.000\) & \(0.000\) & \(0.000\) \\
\(n = 10,000\) 
&\(\alg_{2}(W, \; 100, \; 0.01)\)  
& \(63\) & \(0.026\) & \(0.000\)& \(0.000\)& \(0.000\)& \(0.000\) \\
&\(\alg_{2}(R, \; 100, \; 0.01)\)  
& \(87\) & \(0.001\) & \(1.000\) & \(0.915\) & \(0.000\) & \(0.000\)\\
&\(\alg_{2}(T, \; 100, \; 0.01)\)  
& \(2,509\) & \(0.008\) & \(0.970\) & \(0.206\) & \(0.000\) & \(0.000\)\\
&\(\alg_{TB}(2)\)  
& \(150\)  & \(1.000\) & \(0.000\) & \(0.000\) & \(0.000\) & \(0.000\) \\
&\(\alg_{TT}\)  
& \(5,760\) & \(1.000\) & \(0.000\) & \(0.000\) & \(0.000\) & \(0.000\)\\
\bottomrule
\end{tabular}
\caption{Threshold, Type I error, and Type II error for algorithms on the two-dimensional toroidal grid graph of size \(n\). Here, \(\graphset_{0}\) is the empty graph. We set \(k = c = 500\). The statistic \(W\) performs best by far, and the permutation tests offer more reasonable error bounds.
The Type I and Type II errors are approximated using \(1000\) simulations each.}
\label{tableGrid2D}
\end{table*}


\subsection{Online social network}
\label{appOnlineSocialNetwork}

In order to gauge the performance of our algorithm on more ``irregular'' networks, we performed simulations on real social network graphs with the empty graph as the null. We used two connected components of the Facebook social network from the Stanford Network Analysis Project (SNAP) \cite{snapnets}. The graphs, which we refer to as Facebook 1 and Facebook 2, correspond to users who downloaded the Social Circles app and can be seen in Figure~\ref{figFacebook}.

\begin{figure}[t!]
\centering 
\minipage{0.4\textwidth}
\includegraphics[width=\linewidth]{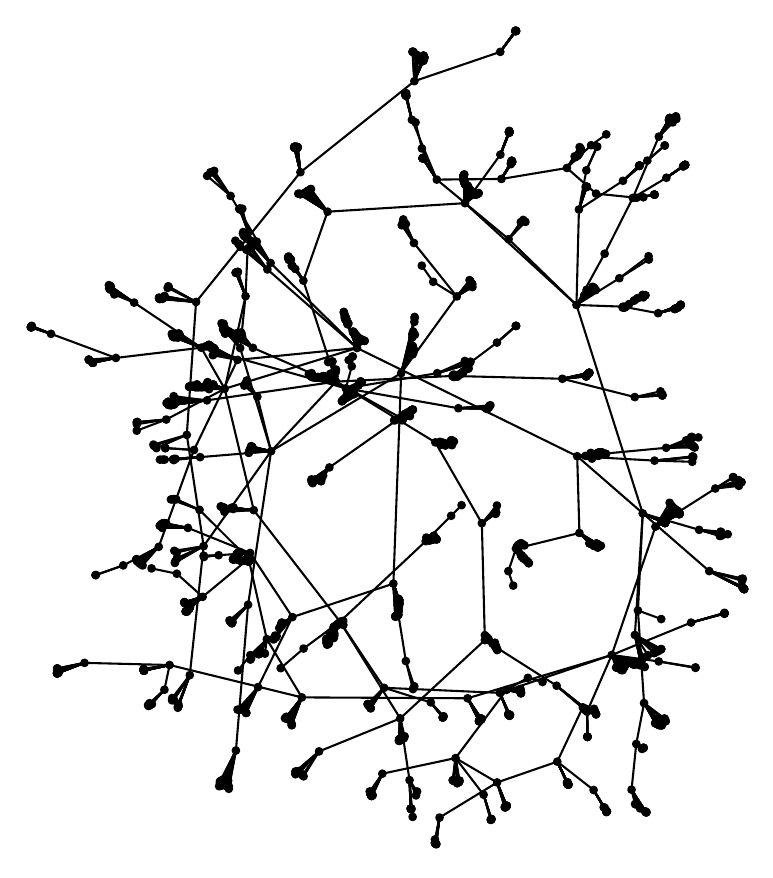}
\endminipage\hfill%
\minipage{0.4\textwidth}
\includegraphics[width=\linewidth]{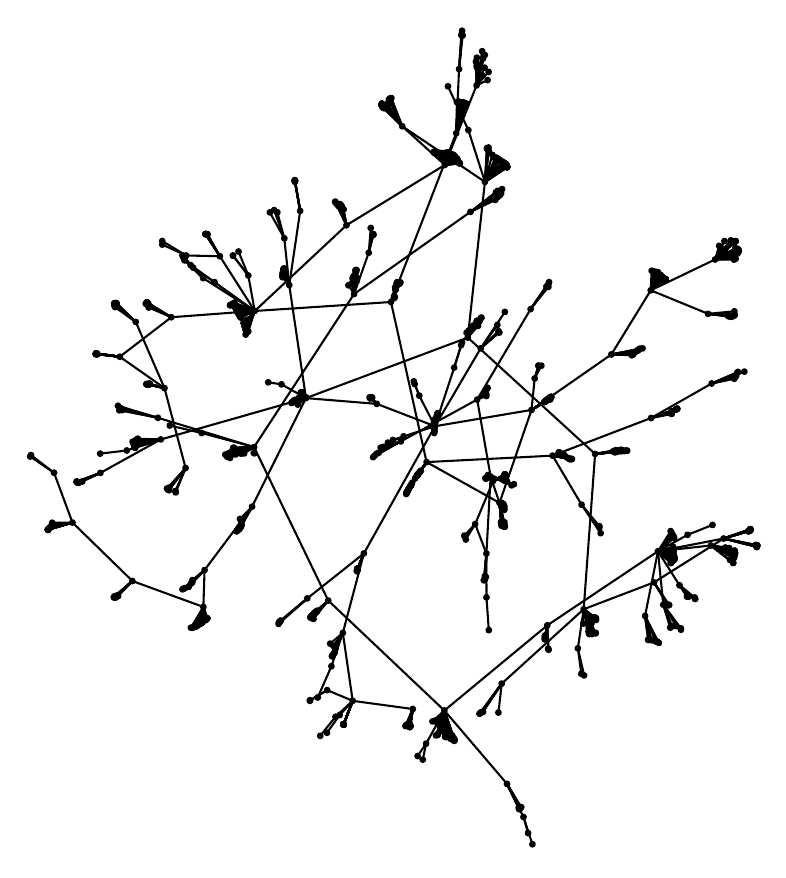}
\endminipage\hfill%
\caption{The Facebook 1 and Facebook 2 graphs, with \(2040\) and \(1567\) nodes. Note that the graphs contain a large proportion of leaves.}
\label{figFacebook}
\end{figure}

We simulated 1000 infections on the graphs with \(k = 200\), \(c = 300\), $\lambda = 1$, and $\beta \in \{1, 10, 100, 1000\}$.  For the TB algorithm, we chose the integer \(d\) providing the smallest threshold. 
Table~\ref{tableSocialNetwork} details the results.
As seen in the table, Algorithm~\ref{AlgPermApprox} with statistic $W$ again performed better than the other algorithms in terms of Type II error. The TB and TT algorithms performed poorly, since the thresholds for all values of \(d\) are too large.

\begin{table*}\centering
\ra{1.3}
\begin{tabular}{@{}llllllll@{}}\toprule
Network & Algorithm  & Threshold & Type I Error & \multicolumn{4}{c}{Type II Error by \(\beta\)} \\
\cmidrule{1-8}
&&&
& \(1\) & \(10\) & \(100\) & \(1000\)\\ \midrule
Facebook 1
&\(\alg_{2}(W, \; 200, \; 0.01)\)  
& \(46\) & \(0.005\)& \(0.242\) & \(0.000\) & \(0.000\) & \(0.000\) \\
&\(\alg_{2}(R, \; 200, \; 0.01)\)  
& \(11\) & \(0.018\)& \(0.958\) & \(0.798\) & \(0.307\) & \(0.062\) \\
&\(\alg_{2}(T, \; 200, \; 0.01)\)  
& \(273\) & \(0.004\)& \(0.700\) & \(0.000\) & \(0.000\) & \(0.000\) \\
&\(\alg_{TB}(3)\)  
& \(77\) & \(1.000\) & \(0.000\) & \(0.000\) & \(0.000\) & \(0.000\) \\
&\(\alg_{TT}\)  
& \(1964\) & \(1.000\) & \(0.000\) & \(0.000\) & \(0.000\) & \(0.000\) \\
\\
Facebook 2 
&\(\alg_{2}(W, \; 100, \; 0.01)\)  
& \(51\) & \(0.017\)& \(0.048\) & \(0.000\) & \(0.000\) & \(0.001\) \\
&\(\alg_{2}(R, \; 100, \; 0.01)\)  
& \(10\) & \(0.007\)& \(0.987\) & \(0.830\) & \(0.339\) & \(0.091\) \\
&\(\alg_{2}(T, \; 100, \; 0.01)\)  
& \(249\) & \(0.029 \)& \(0.290\) & \(0.000\) & \(0.000\) & \(0.000\) \\
&\(\alg_{TB}(3)\) 
& \(78\) & \(1.000\) & \(0.000\) & \(0.000\) & \(0.000\) & \(0.000\) \\
&\(\alg_{TT}\)  
& \(1965\) & \(1.000\) & \(0.000\) & \(0.000\) & \(0.000\) & \(0.000\) \\
\bottomrule
\end{tabular}
\caption{Threshold, Type I error, and Type II error for algorithms when \(\graphset_{1}\) is one of the subsets of the Facebook graph. 
Here, we set \(k = 200\) and \(c = 300\). Note that \(\alg_{2}(W, B, \alpha)\) seems to perform better than the other algorithms.}
\label{tableSocialNetwork}
\end{table*}


\subsection{Erd\H{o}s-Renyi versus stochastic block model graphs}
\label{appSBM}

\begin{figure}[!h]
\centering 
\includegraphics[width=4.00in]{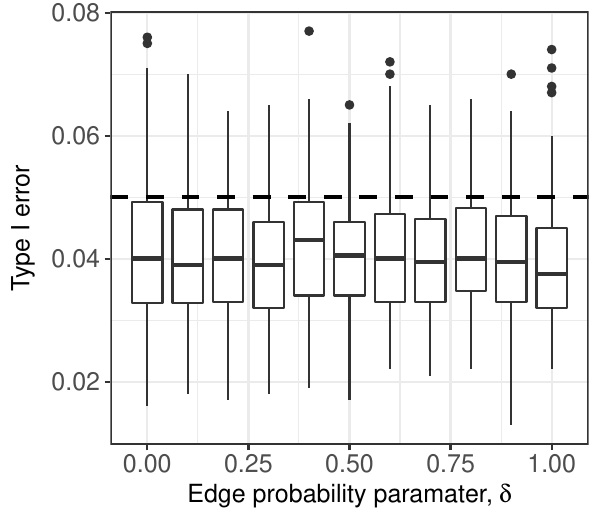}
\caption{The Type I error of the permutation test applied to the test of an Erd\H{o}s-Renyi random graph versus a two-block stochastic block model graph on \(n = 500\) vertices and $k = c = 50$. The edge density is \(p = \log(n)/n\), and the within-block connectivity is \(a = 2 \delta p\) and $\beta = 1$.
On average, the permutation test controls the Type I error at the nominal level \(\alpha = 0.05\).}
\label{figSBMTypeI}
\end{figure}

For our first set of random graph simulations, we considered an infection spreading on an Erd\H{o}s-Renyi graph, versus a two-block stochastic block model with equal-sized partitions, where \(n = 500\). We let \(k = c = 50\).
In the null graph, edges exist independently with probability \(p = \log(n) / n\).
In the alternative graph, the probability that an edge exists between nodes in the same partition is \(a\), and the probability that an edge exists between nodes in different partitions is \(b\), where \(a/2 + b/2 = p\).
We let \(a = 2 \delta p\) and varied \(\delta\) in \([0, 1]\) in increments of \(0.1\).
To estimate thresholds and error probabilities, we sampled each infection process \(1000\) times. We drew \(100\) pairs of random graphs for each \(\delta\).

Lemma~\ref{lemmaRandomGraphs} predicts the validity of the permutation test in controlling Type I error. Furthermore, we expect the Type II error to decrease as the infections become sufficiently dissimilar through increasing \(\beta\) or \(\delta\).
The first of these conclusions is verified by examining Figure~\ref{figSBMTypeI}: As \(\delta\) varies, the permutation test controls the Type I error at a level \(\alpha = 0.05\) on average, and with reasonably high probability. Furthermore, even outliers had Type I errors of no more than \(0.08\).

\begin{figure}[!h]
\centering 
\includegraphics[width=4.00in]{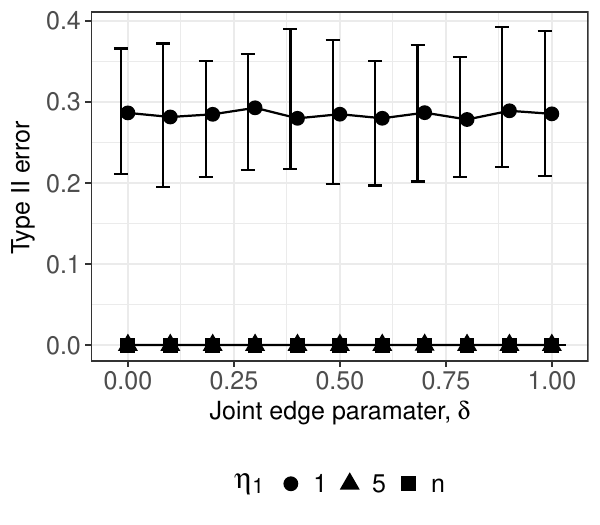}
\caption{The Type II error of the permutation test applied to the test of an Erd\H{o}s-Renyi random graph versus a two-block stochastic block model graph on \(n = 500\) vertices and $k = c = 50$. The edge density is \(p = \log(n)/n\), and the within-block connectivity is \(a = 2 \delta p\).
The markers denote the average error, and the bars denote the range over the \(100\) simulated graphs.
Perhaps surprisingly, the Type II error does not seem to depend on \(\delta\).
However, the Type II error decreases with \(\beta_{1}\).}
\label{figSBMEmptyTypeII}
\end{figure}

\begin{figure}[!h]
\centering 
\includegraphics[width=4.00in]{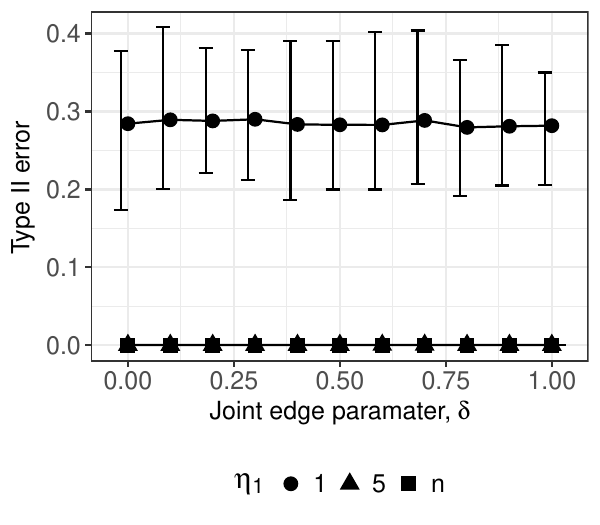}
\caption{The Type II error simulating the distribution of \(\infected\) on an Erd\H{o}s-Renyi \(\graphset_{0}\) with parameters \(\beta_{0} = n = 500\) against a two-block stochastic block model \(\graphset_{1}\).
Here, \(a = 2 \delta p\) and \(b = 2p - a\) are the within- and between-block edge probabilities.
Additionally, we set \(k = c = 50\).
The marker indicates the average value across \(100\) graph realizations, and the bars show the range of Type II errors.
Perhaps surprisingly, this graph is very similar to Figure~\ref{figSBMEmptyTypeII}, both in terms of the average Type II error and range of values.}
\label{figSBMNullTypeII}
\end{figure}

To examine the Type II error, we provide two plots.
First, Figure~\ref{figSBMEmptyTypeII} shows the average Type II error for various levels of \(\beta_{1}\) and \(\delta\), when \(\beta_{0} = 0\).
The marker indicates the average value across the \(100\) graph realizations, and the bars show the range of Type II errors.
Additionally, the bars near each point indicate the range of the Type II errors across simulated graphs.
Note that the Type II error is low for even modest values of \(\beta_{1}\). The second plot is Figure~\ref{figSBMNullTypeII}.
Here, the mean Type II error is plotted for various levels of \(\beta_{1}\) and \(\delta\), where \(\beta_{0} = n\).
Additionally, the bars near each point indicate the range of Type II errors over simulated graphs.
Note that Figures~\ref{figSBMEmptyTypeII} and \ref{figSBMNullTypeII} are very similar.
Thus, even strong spreading on the Erd\H{o}s-Renyi random graph does not seem to affect Type II error both in terms of the average and the range of values.


\subsection{Correlated Erd\H{o}s-Renyi random graphs}
\label{appCER}

We now consider a case where the permutation test performs poorly, by examining correlated Erd\H{o}s-Renyi random graphs.
We took $\graphset_0$ to be an Erd\H{o}s-Renyi graph on \(n = 500\) nodes, with \(p = \log(n) / n\). We set the infection parameters to be \(k = c = 50\).
The alternative graph \(\graphset_{1}\) was also an Erd\H{o}s-Renyi random graph, with edges drawn independently with probability \(p\).
However, the existence of an edge \((u, v)\) in \(\graphset_{0}\) and \(\graphset_{1}\) was correlated, so \(\gamma p\) was the probability of \((u, v)\) existing in both graphs, where \(\gamma\) in \([0, 1]\) was taken in increments of \(0.1\). Note that \(\gamma = p\) corresponds to independent draws. To estimate thresholds and error probabilities for each graph, we sampled infection processes \(1000\) times for each set of parameters.
Finally, we drew \(100\) pairs of graphs \((\graphset_{0}, \graphset_{1})\).

Due to correlations in edge appearance probabilities in $\graphset_0$ and $\graphset_1$, our theory for permutation testing no longer applies. Intuitively, we expect that as \(\gamma\) increases, the test loses its Type I error control: For larger values of \(\gamma\), the edges of \(\graphset_{0}\) and \(\graphset_{1}\) are more positively correlated, which increases the probability that the vertices of an edge \((u, v)\) in \(\edgeset_{1}\) are infected under the null distribution. Thus, the permutation test sets the threshold for the test using \(W_{1}\) to be too low, causing a large Type I error. Indeed, this behavior is seen in Figure~\ref{figCERTypeI}, where
we see that the nominal Type I error level \(\alpha = 0.05\) is far from achieved for every realization with \(\gamma \geq 0.2\).

\begin{figure}[!h]
\centering 
\includegraphics[width=4.00in]{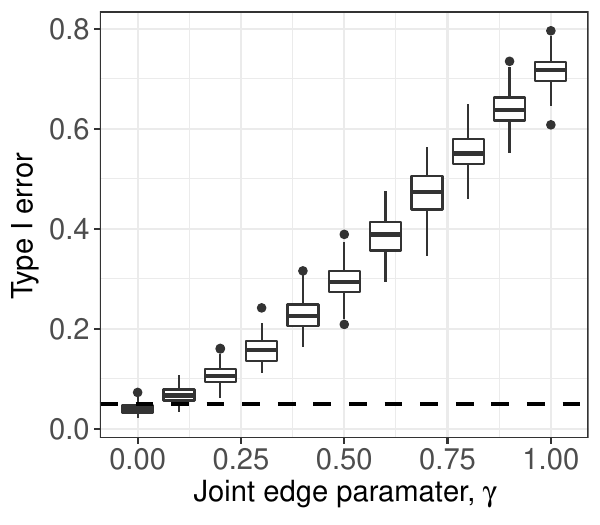}
\caption{The Type I error of the permutation test applied to correlated Erd\H{o}s-Renyi random graphs on \(n = 500\) vertices. The marginal probability of an edge \((u, v)\) is \(p = \log(n)/n\) in both the null and alternative, and the probability of the edge in both graphs is \(\gamma p\).
As \(\gamma\) increases, the edge correlation increases, causing the Type I error to increase. Here, \(\beta = 1\) and \(k = c = 50\).}
\label{figCERTypeI}
\end{figure}

\begin{figure}[!h]
\centering 
\includegraphics[width=4.00in]{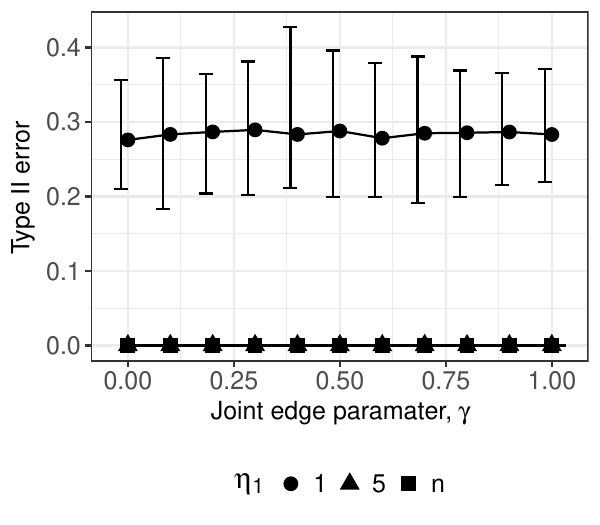}
\caption{The Type II error
of the permutation test applied to correlated Erd\H{o}s-Renyi random graphs on \(n = 500\) vertices. The marginal probability of an edge \((u, v)\) is \(p = \log(n)/n\) in both the null and alternative, and the probability of the edge in both graphs is \(\gamma p\).
The marker indicates the average Type II error, and the bars indicate the range.
Perhaps surprisingly, Figure~\ref{figCEREmptyTypeII} resembles Figures~\ref{figSBMEmptyTypeII} and \ref{figSBMNullTypeII}, despite the difference in graph structures.
}
\label{figCEREmptyTypeII}
\end{figure}

The behavior of the Type II error is more nuanced in this case.
In particular, for high values of \(\gamma\), we
would still expect the Type II error to be roughly on the same order as in the case of the stochastic block model.
However, if we were to set the threshold using the actual null distribution, meaning the correlated random graph \(\graphset_{0}\),
perhaps the threshold would be sufficiently elevated by the correlated edges, which would in turn increase the Type II error.
We do indeed see these behaviors in Figures~\ref{figCEREmptyTypeII} and \ref{figCERNullTypeII}.
In the first figure, the Type II error does not depend on \(\gamma\), since \(\graphset_{1}\) is always an Erd\H{o}s-Renyi random graph with parameter \(p\).
In the second figure, we see that the Type II error increases drastically when considering a threshold set by simulating on \(\graphset_{0}\).
However, note that for large values of \(\beta_{1}\), the graphs can still be distinguished even with moderately large correlation.
This provides hope that in correlated cases, infection patterns that depend very strongly on the graph topology could still yield correct inference for the true graph.

\begin{figure}[!h]
\centering 
\includegraphics[width=4.00in]{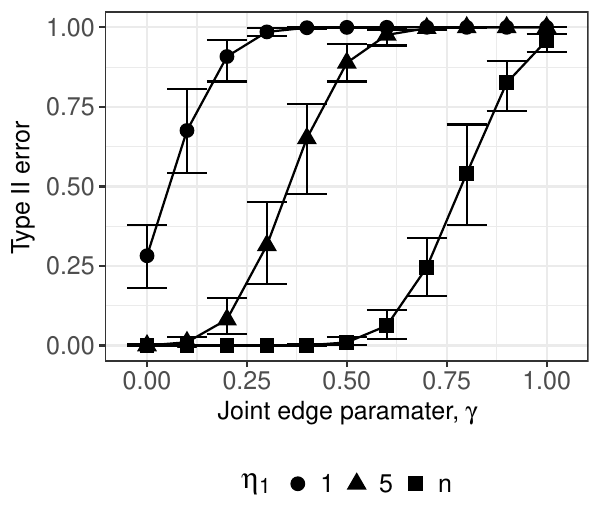}
\caption{The Type II error between a pair of correlated Erd\H{o}s-Renyi random graphs \(\graphset_{0}\) and \(\graphset_{1}\) when \(\beta_{0} = n = 500\).
The marginal probability of an edge \((u, v)\) in either graph is \(p = \log(n)/n\), and the probability of the edge in both graphs is \(\gamma p\).
Additionally, we set \(k = c = 50\).
The markers denote the average Type II errors, and the bars denote the range.
Note that this figure is very different from Figure~\ref{figCEREmptyTypeII}.
In particular, for  higher values of \(\gamma\), the Type II error is much worse, since the graphs become similar.
However, note that if \(\beta_{1}\) is sufficiently high, the graphs can still be distinguished even for moderately large \(\gamma\).}
\label{figCERNullTypeII}
\end{figure}


\section{Proofs}
\label{appProofs}

In this section, we provide proofs of our main results and the propositions.

\subsection{Proofs of main theorems}
\label{subsec: proofs of main results}

\begin{proof}[Proof of Theorem~\ref{ThmPermStat}]


Note that under $H_0$, and for any set $\mathcal{B}$, we have
\begin{equation*}
\prob_{0}\left(\infected \in \mathcal{B} \right) 
=
\prob_{0}\left(\pi\infected \in \mathcal{B} \right) 
= 
\frac{1}{|\Pi_0|} \sum_{\pi_0 \in \Pi_0} \prob_{0}\left(\pi_0 \infected \in \mathcal{B}  \right), 
\end{equation*}
where the first equality is by assumption and the second is by averaging.
The first equality is proven in detail for our models in Section~\ref{subsecSymmetries}.
In particular,
\begin{align*}
\prob_{0}\left(S(\infected) \le t   \right) 
& = \prob_{0}\left(\infected \in S^{-1}(-\infty, t]  \right)  \\ 
& = \frac{1}{|\Pi_0|} \sum_{\pi_0 \in \Pi_0} \prob_{0} \left(\pi_0 \infected \in S^{-1} (-\infty, t]  \right) \\  
& = \frac{1}{|\Pi_0|} \sum_{\pi_0 \in \Pi_0} \prob_{0}\left(S(\pi_0 \infected) \le t  \right). 
\end{align*}


Now let $\{g_1, \dots, g_m\}$ denote coset representatives for $\Pi_0$ in the permutation group $S_n$, so $\{g_1 \Pi_0, \dots, g_m \Pi_0\}$ partitions $S_n$ and has cardinality
\[
m = \frac{|S_n|}{|\Pi_0|} = \frac{n!}{|\Pi_0|}.
\] 
By assumption, we may choose $g_i \in \Pi_1$ for each $1 \le i \le m$. Since $S$ is a $\Pi_1$-invariant statistic, this means $S(g_i \pi_0 \infected) = S(\pi_0 \infected)$ for each $i$. Hence,
\begin{align*}
\prob_{0}\left(S(\infected) \le t  \right) 
 = 
\frac{1}{|\Pi_0|} \sum_{\pi_0 \in \Pi_0} \frac{1}{m} \sum_{1 \le i \le m} \prob_{0}\left(S(g_i \pi_0 \infected) \le t  \right) 
 = \frac{1}{n!} \sum_{\pi \in S_n} \prob_{0}\left(S(\pi \infected) \le t  \right), 
\end{align*}
implying that $S(\infected)$ and $S(\pi \infected)$ have the same distribution when $\pi \sim \text{Uniform}(S_n)$.
\end{proof}

\begin{proof} [Proof of Theorem~\ref{ThmPermTest}]
By Theorem~\ref{ThmPermStat}, we have
\begin{equation*}
\prob_{0}\left(S(\infected) \le t  \right) 
=  
\frac{1}{n!} \sum_{\pi \in S_n} \prob_{0}\left(S(\pi \infected) \le t   \right), 
\end{equation*}
for each $t \in \mathbb{R}$. 
Conditioning on the infection vector \(\infected\), we obtain
\begin{align*}
\prob_{0}\left(S(\infected) \le t  \right) 
& = 
\sum_{\infected \in \infspace_{k, c}} \prob_{0}\left(\infected\right)
\frac{1}{n!} \sum_{\pi \in S_n} \prob_{0}\left(S(\pi \infected) \le t \; \Big | \;  \infected\right) \\ 
&= 
\frac{1}{n!} \sum_{\infected \in \infspace_{k, c}}  \sum_{\pi \in S_n} 
\prob_{0}\left(\infected\right) \prob_{0}\left(S(\pi \infected) \le t \; \Big | \; \infected\right) .
\end{align*}
Fix some \(J \in \infspace_{k, c}\).
For an infection \(\infected\), let \(\pi_{J, \infected}\) be a permutation mapping \(\infected\) to \(J\).
Then $\infected = \pi_{J, \infected}^{-1} J$, so $\pi \infected = \pi \pi_{J, \infected}^{-1} J$, and
\begin{align*}
\prob_{0}\left(S(\infected) \le t \right) %
& = 
\frac{1}{n!} \sum_{\infected \in \infspace_{k, c}}  \sum_{\pi \in S_n} 
\prob_{0}\left(\infected\right) \prob_{0}\left(S\left(\pi  \pi_{J, \infected}^{-1} J\right) \le t \; \Big | \; \infected\right) \\
&=
\frac{1}{n!} \sum_{\infected \in \infspace_{k, c}} 
\prob_{0}\left(\infected\right) \sum_{\pi' \in S_n} \prob_{0}\left(S(\pi' J) \le t \; \Big | \; \infected\right) \\
&=
\frac{1}{n!} \sum_{\infected \in \infspace_{k, c}} 
\prob_{0}\left(\infected\right) \sum_{\pi' \in S_n} \prob_{0}\left(S(\pi' J) \le t\right),
\end{align*}
where we denote $\pi' = \pi \pi_{J, \infected}^{-1}$. In the last step, the conditioning on $\infected$ becomes irrelevant because we sum over all permutations in $S_n$. Note that
\begin{equation*}
\prob_0(S(\pi' J) \le t) = \ind\left\{S(\pi' J) \le t\right\},
\end{equation*}
since all quantities are deterministic. Hence,
\begin{align*}
\prob_{0}\left(S(\infected) \le t\right) 
&=
\frac{1}{n!} \sum_{\infected \in \infspace_{k, c}}
\prob_{0}\left(\infected\right)  \sum_{\pi' \in S_n}  \ind\left\{S(\pi' J) \le t \right\} 
= 
\frac{1}{n!} \sum_{\pi' \in S_n} 
\ind\left\{S(\pi' J) \le t \right\}.
\end{align*}
This justifies the permutation test described in Algorithm~\ref{AlgPermExact}: The threshold $t_\alpha$ to bound the Type I error at level $\alpha$ may be computed explicitly from computing the appropriate quantile of $S$ with respect to all $n!$ permutations of $J$.
\end{proof}

\begin{proof}[Proof of Theorem~\ref{ThmConverse}]
Let $\scriptO_{\Pi_1}(1)$ denote the orbit of vertex 1 in $\Pi_1 = \Aut(G_1)$, and note that by assumption, $|\scriptO_{\Pi_1}(1)| < n$. Consider the statistic
\begin{equation*}
S(J) = \sum_{v \in \scriptO_{\Pi_1}(1)} 1\{J_v = 1\},
\end{equation*}
which counts the number of infected vertices in $\scriptO_{\Pi_1}(1)$. Note that $S$ is clearly $\Pi_1$-invariant. We claim that $S(\infected)$ and $S(\pi \infected)$ do not have the same distribution under $H_0$, when $\pi \sim \text{Uniform}(S_n)$.

Let $G = \{g_1, \dots, g_a\} \subseteq \Pi_0$ be a set consisting of coset representatives such that $\{\Pi_1 g_1, \dots, \Pi_1 g_a\}$ is a partition of $\Pi_1 \Pi_0$, and let $H = \{h_1, \dots, h_b\} \subseteq S_n \backslash \Pi_0$ be representatives of the remaining cosets of $\Pi_1$ in $S_n$. For any observation vector $J$, we have
\begin{align*}
\prob_0\left(S(\pi \infected) = S(J)\right) & = \frac{1}{n!} \sum_{\pi_1 \in \Pi_1} \sum_{g \in G \cup H} \prob_0 \left(S(\pi_1 g \infected) = S(J)\right) \\
& = \frac{|\Pi_1|}{n!} \sum_{g \in G \cup H} \prob_0 \left(S(g\infected) = S(J)\right) \\
& = \frac{1}{a+b} \sum_{g \in G \cup H} \prob_0 \left(S(g\infected) = S(J)\right),
\end{align*}
where the first equality uses the fact that $\{\Pi_1 g_1, \dots, \Pi_1 g_a\} \cup \{\Pi_1 h_1, \dots, \Pi h_b\}$ is a partition of $S_n$, and the second equality uses the fact that $S$ is $\Pi_1$-invariant.

By symmetry of the spreading process on $\graphset_0$, we have
\begin{equation*}
\prob_0\left(S(g \infected) = S(J)\right) = \prob_0\left(S(\infected) = S(J)\right), \qquad \forall g \in G.
\end{equation*}
Hence,
\begin{equation*}
\prob_0\left(S(\pi \infected) = S(J)\right) = \frac{a}{a+b} \prob_0\left(S(\infected) = S(J)\right) + \frac{1}{a+b} \sum_{h \in H} \prob_0 \left(S(h\infected) = S(J)\right).
\end{equation*}
We will demonstrate a choice of $J$ for which
\begin{equation*}
\sum_{h \in H} \prob_0 \left(S(h\infected) = S(J)\right) \neq b \; \prob_0\left(S(\infected) = S(J)\right),
\end{equation*}
implying that
\begin{equation*}
\prob_0\left(S(\pi \infected) = S(J)\right) \neq \prob_0\left(S(\infected) = S(J)\right),
\end{equation*}
so $S(\pi \infected)$ and $S(\infected)$ cannot have the same distribution.

Let $J \in \infspace_{k,c}$ be a vector such that $S(J) = m$, for some \(m\) to be specified later. 
Then $\prob_0\left(S(\infected) = S(J)\right)$ is the probability that exactly $m$ vertices in $\scriptO_{\Pi_1}(1)$ are infected. On the other hand, for a fixed $g \in S_n$, the quantity $S(g\infected)$ counts the number of infected vertices in $g^{-1}\left(\scriptO_{\Pi_1}(1)\right)$. Again using symmetry of the spreading process on $\graphset_0$, we have
\begin{equation*}
\prob_0\left(S(g \infected) = m\right) = \prob_0\left(S(\infected) = m\right),
\end{equation*}
whenever $1 \in g^{-1} \left(\scriptO_{\Pi_1}(1)\right)$. However, when $1 \notin g^{-1} \left(\scriptO_{\Pi_1}(1)\right)$, we have
\begin{equation*}
\prob_0\left(S(g \infected) = m\right) < \prob_0\left(S(\infected) = m\right),
\end{equation*}
for some \(m\),
since the center of the star is more likely to be infected than any of the leaves. Finally, note that $h_j(1) \notin \scriptO_{\Pi_1}(1)$ for some $j$. Indeed, if $h_i(1) \in \scriptO_{\Pi_1}(1)$ for all $1 \le i \le b$, we would have $\pi_1 h_i(1) \in \scriptO_{\Pi_1}(1)$ for all $\pi_1 \in \Pi_1$, contradicting the fact that the cosets cover the entire space $S_n$. Thus, we have $1 \notin h_j^{-1}\left(\scriptO_{\Pi_1}(1)\right)$, implying that
\begin{equation*}
\prob_0(S(h_j \infected) = m) < \prob_0(S(\infected) = m),
\end{equation*}
and in particular,
\begin{equation*}
\sum_{i=1}^b \prob_0 \left(S(h_i \infected) = S(J)\right) < b \; \prob_0\left(S(\infected) = S(J)\right).
\end{equation*}
This completes the proof.
\end{proof}

\subsection{Proofs of invariances}
\label{subsecSymmetries}

In this appendix, we derive the following key lemma concerning permutation invariance of infection statistics under the graph automorphism group. This lemma is leveraged in the proof of Theorem~\ref{ThmPermStat}. We first outline the proof of the lemma, and then provide the proofs of several supporting lemmas.

\begin{lemma}
Let \(\Pi_0 = \Aut(\graphset_{0})\). For any \(\pi_{0}\) in \(\Pi_{0}\), we have
\[
\prob_{0}(\infected = J)
= 
\prob_{0}(\pi_{0}\infected = J)
=
\frac{1}{|\Pi_{0}|} \sum_{\pi \in \Pi_{0}} \prob_{0}(\pi \infected = J),
\]
under both the stochastic spreading and conditional Ising models.
\label{lemmaSymmetries}
\end{lemma}
\begin{proof}
Note that it suffices to prove the first equality, since the second equality may be obtained by averaging over \(\Pi_{0}\). Further note that once we have proved the equality for uncensored infection vectors, the extension to censored vectors follows immediately from the additive expression for the censoring measure over all possible uncensored configurations.


We begin by focusing on the stochastic spreading model. The probability of an (uncensored) infection vector \(J\) is 
\begin{align*}
& \begin{aligned}
\prob_{0}(\infected = J)
&=
\sum_{P \in \pathset(J)} \prob_{0}(\pathvar = P).
\end{aligned}
\end{align*}
Thus, it suffices to prove that 
\(\prob_{0}(\pathvar = P) = \prob_{0}(\pi_{0} \pathvar = P)\),
for all $P \in \pathset(J)$, where 
\[
\pi_{0} P = \pi_0(P_1, \dots, P_k)
= 
\left(\pi_{0}(P_{1}), \ldots, \pi_{0}(P_{k})\right).
\]
The probability of a path \(P\) is equal to
\begin{align*}
& \begin{aligned}
\prob_{0}(\pathvar = P)
&=
\prod_{t = 1}^{k} 
\frac{1 + \neighbors_{t, \invertex}(P)}{n + 1 - t + \beta \neighbors_{t}(P)}.
\end{aligned}
\end{align*}
Thus, it suffices to show that \(\neighbors_{t, \invertex}(P)\) and \(N_{t}(P)\) are permutation-invariant.
This is done in Lemma~\ref{lemmaNeighborsInvariant} below.

In the case of the Ising model, we have seen that the probability of an infection \(J\) only depends on its Hamiltonian. Lemma~\ref{lemmaEWInvariant} below shows that the Hamiltonian statistic is permutation-invariant.
\end{proof}


We now provide the proofs of the supporting lemmas to Lemma~\ref{lemmaSymmetries}. Note that the following lemmas are all deterministic statements in terms of a fixed graph $\graphset$, with the automorphism group $\Pi := \Aut(\graphset)$; they will be applied with $\graphset = \graphset_0$ and $\graphset_1$ in our arguments in the paper.

\begin{lemma}
The statistics \(\neighbors_{t, \invertex}\) and \(\neighbors_{t}\) are \(\Pi\)-invariant when computed on graph \(\graphset\).
\label{lemmaNeighborsInvariant}
\end{lemma}

\begin{proof}
We have
\begin{align*}
& \begin{aligned}
\neighbors_{t, \invertex}(P)
&=
\sum_{s = 1}^{t - 1} \ind\left\{(P_{s}, P_{t}) \in \edgeset\right\}, \\ 
\neighbors_{t}(P)
&=
\sum_{(u, v) \in \edgeset}
\ind\left\{\text{exactly one of \(u\) or \(v\) is in \((P_{1}, \ldots, P_{t - 1})\)}\right\}.
\end{aligned}
\end{align*}

For \(\pi \in \Pi\), we have
\begin{align*}
& \begin{aligned}
\neighbors_{t, \invertex}(P)
&=
\sum_{s = 1}^{t - 1}
\ind\left\{(P_{s}, P_{t}) \in \pi^{-1} \edgeset\right\} \\
&=
\sum_{s = 1}^{t - 1}
\ind\left\{(\pi(P_{s}), \pi(P_{t})) \in \edgeset\right\} \\
&=
N_{t, \invertex}(\pi P),
\end{aligned}
\end{align*}
where we have used the definition of an automorphism and the closure of groups under inversion for the first equality.
Similarly, we may write
\begin{align*}
& \begin{aligned}
\neighbors_{t}(P)
&=
\sum_{(u, v) \in \pi^{-1} \edgeset} 
\ind\left\{\text{exactly one of \(u\) or \(v\) is in \((P_{1}, \ldots, P_{t - 1})\)}\right\} \\ 
&=
\sum_{(u, v) \in \edgeset} 
\ind\left\{\text{exactly one of \(\pi^{-1}(u)\) or \(\pi^{-1}(v)\) is in \((P_{1}, \ldots, P_{t - 1})\)}\right\} \\ 
&=
\sum_{(u, v) \in  \edgeset} 
\ind\left\{\text{exactly one of \(u\) or \(v\) is in 
\((\pi(P_{1}), \ldots, \pi(P_{t - 1}))\)}\right\} \\ 
&=
N_{t}(\pi P).
\end{aligned}
\end{align*}
\end{proof}

\begin{lemma}
The edges-within statistic \(W(J) = \energy'(J)\) is 
\(\Pi\)-invariant when computed on graph $\graphset$.
\label{lemmaEWInvariant}
\end{lemma}

\begin{proof}
For \(\pi \in \Pi\), we have
\begin{align*}
& \begin{aligned}
W(J)
&=
\sum_{(u, v) \in \edgeset} \ind\left\{J_{u} = J_{v} = 1 \right\} \\ 
&=
\sum_{(u, v) \in \pi^{-1} \edgeset}
\ind\left\{J_{u} = J_{v} = 1 \right\} \\ 
&=
\sum_{(u, v) \in \edgeset}
\ind\left\{J_{\pi^{-1}(u)} = J_{\pi^{-1}(v)} = 1 \right\} \\ 
&=
\sum_{(u, v) \in \edgeset} 
\ind\left\{(\pi J)_{u} = (\pi J)_{v} = 1\right\} \\ 
&=
W(\pi J),
\end{aligned}
\end{align*}
using similar arguments to the proof of Lemma~\ref{lemmaNeighborsInvariant}.
\end{proof}
\subsection{Hunt-Stein}

\begin{proof}[Proof of Theorem~\ref{theoremHuntStein}]
Clearly, $\psi^*$ is $\Pi$-invariant, since for any $\pi' \in \Pi$, we have
\begin{equation*}
\psi^*(\pi' \infected) = \frac{1}{|\Pi|} \sum_{\pi \in \Pi} \varphi^*(\pi \pi' \infected) = \frac{1}{|\Pi|} \sum_{\pi \pi' \in \Pi \pi'} \varphi^*(\pi \pi' \infected) = \frac{1}{|\Pi|} \sum_{\tilde{\pi} \in \Pi} \varphi^*(\tilde{\pi} \infected) = \psi^*(\infected).
\end{equation*}

We now show that $\psi^*$ satisfies the inequality
\begin{equation}
\inf_{\pi \in \Pi} \expect_{\pi \theta} [\varphi^*(\infected)]
\leq 
\expect_{\theta} [\psi^*(\infected)]
\leq 
\sup_{\pi \in \Pi} \expect_{\pi \theta} [\varphi^*(\infected)], \qquad \forall \theta \in \Theta.
\label{eqn823}
\end{equation}
We may write
\begin{equation*}
\E_\theta[\psi^*(\infected)] = \E_\theta\left[\frac{1}{|\Pi|} \sum_{\pi \in \Pi} \varphi^*(\pi \infected)\right] = \frac{1}{|\Pi|} \sum_{\pi \in \Pi} \E_\theta[\varphi^*(\pi \infected)] = \frac{1}{|\Pi|} \sum_{\pi \in \Pi} \E_{\pi \theta} [\varphi^*(\infected)].
\end{equation*}
Furthermore, we clearly have
\[
\inf_{\pi \in \Pi} \expect_{\pi \theta} [\varphi^*(\infected)]
\leq 
\frac{1}{|\Pi|} \sum_{\pi \in \Pi} \expect_{\pi\theta}[\varphi^*(\infected)]
\leq 
\sup_{\pi \in \Pi} \expect_{\pi \theta}[\varphi^*(\infected)].
\]
Combining the two relations gives
inequality~\eqref{eqn823}.

In particular, we have
\[
\expect_{\theta} [\psi^{*}(\infected)]
\leq 
\sup_{\pi \in \Pi} \expect_{\pi \theta} [\varphi^{*}(\infected) ]
\leq 
\alpha,
\text{\hspace{20pt} for all \(\theta\) in \(\Theta_{0}\)},
\]
so $\psi^*$ is a level-$\alpha$ test. Furthermore,
\[
\expect_{\theta} [\psi^{*}(\infected)]
\geq 
\inf_{\pi \in \Pi} \expect_{\pi \theta} [\varphi^{*}(\infected) ],
\text{\hspace{20pt} for all \(\theta\) in \(\Theta_{1}\)},
\]
so \(\psi^{*}\) is maximin.
\end{proof}

\subsection{Likelihood ratio}

\begin{proof}[Proof of Theorem~\ref{theoremLikelihoodRatio}]
We begin by considering the case $c=0$.
The likelihood under \(H_{0}\) is
\begin{align*}
& \begin{aligned}
L(\graphset_{\text{empty}}, \beta; \infected)
&=
\sum_{P \in \pathset(\infected)} \frac{1}{n} \cdots \frac{1}{n - (k - 1)}
=
k! \cdot  \frac{1}{n} \cdots \frac{1}{n - (k - 1)}.
\end{aligned}
\end{align*}
For \(\graphset_{1}\), we have
\begin{align}
\label{EqnUncensored}
& \begin{aligned}
L(\graphset_{1}, \beta; \infected)
&=
\sum_{P \in \pathset(\infected)}
\prod_{t = 1}^{k}
\frac{1 + \beta \neighbors_{t, \invertex}}{(n + 1 - t) + \beta \neighbors_{t}} \\
&=
\sum_{P \in \pathset(\infected)}
\prod_{t = 1}^{k}
(1 + \beta \neighbors_{t, \invertex})
\prod_{t = 1}^{k}
\frac{1}{(n + 1 - t) + \beta \neighbors_{t}},
\end{aligned}
\end{align}
where $\neighbors_{t, \invertex}$ and $\neighbors_t$ are computed with respect to $\graphset_1$. Thus, the likelihood ratio is
\begin{align*}
& \begin{aligned}
R(\graphset_{\text{empty}}, \graphset_{1}, \beta; \infected) 
&=
\frac{1}{k!}
\sum_{P \in \pathset(\infected)}
\left(\prod_{t = 1}^{k}
(1 + \beta \neighbors_{t, \invertex})\right)
\left(
\prod_{t = 1}^{k}
\frac{n + 1 - t}{(n + 1 - t) + \beta \neighbors_{t}}\right) \\
& :=
\frac{1}{k!}
\sum_{P \in \pathset(\infected)}
A(P) B(P).
\end{aligned}
\end{align*}
We analyze each of the products \(A(P)\) and \(B(P)\) separately.
We have the upper bound
\begin{align*}
& \begin{aligned}
A(P)
&=
\exp\left(\sum_{t = 1}^{k} \log(1 + \beta \neighbors_{t, \invertex})\right) \\ 
&\leq 
\exp\left( \beta \sum_{t = 1}^{k} \neighbors_{t, \invertex}\right) \\ 
&=
\exp\left(\beta W_{1}\right) \\ 
&= 
1 + \beta W_{1} + O(\beta^{2} W_{1}^{2}),
\end{aligned}
\end{align*}
where we have used the fact that $\sum_{t=1}^k \neighbors_{t, \invertex} = W_1$ when $\neighbors_{t, \invertex}$ is computed along any possible infection path.
Furthermore, we have the lower bound
\begin{align*}
& \begin{aligned}
A(P)
&\geq 
\exp\left(\sum_{t = 1}^{k} \beta \neighbors_{t, \invertex} - \frac{1}{2} \beta^{2} \neighbors_{t, \invertex}^{2}\right) \\ 
&=
\exp\left(\beta W_{1} -\frac{1}{2} \beta^{2} \sum_{t = 1}^{k} \neighbors_{t, \invertex}^{2}\right) \\ 
&\geq 
1 + \beta W_{1} - \frac{1}{2} \beta^{2} \sum_{t = 1}^{k} \neighbors_{t, \invertex}^{2} \\
&=
1 + \beta W_{1} - O(\beta^{2} W_{1}^{2}).
\end{aligned}
\end{align*}

Turning to $B(P)$, observe that we have the trivial upper bound \(B(P) \leq 1\).
For a lower bound, we write
\begin{align*}
& \begin{aligned}
B(P)
&=
\exp\left(\sum_{t = 1}^{k} \log(n + 1 - t) - \log(n + 1 - t + \beta \neighbors_{t}) \right) \\ 
&=
\exp\left(
-\sum_{t = 1}^{k} \int_{n + 1 - t}^{n + 1 - t + \beta \neighbors_{t}} \frac{1}{x} dx 
\right) \\ 
&\geq 
\exp\left(-\sum_{t = 1}^{k} \frac{\beta \neighbors_{t}}{n + 1 - t}\right) \\ 
&\geq 
1 -\beta  \sum_{t = 1}^{k} \frac{\neighbors_{t}}{n + 1 - t} \\ 
&=
1 - \beta Q(P).
\end{aligned}
\end{align*}
Altogether, we have
\begin{align*}
& \begin{aligned}
R(\graphset_{\text{empty}}, \graphset_{1}, \beta; \infected) 
&=
\frac{1}{k!} \sum_{P \in \pathset(\infected)}
\Big(1 + \beta W_{1} + O(\beta^{2} W_{1}^{2}) \Big)
\Big(1 - O(\beta Q(P))\Big) \\
&=
\Big(1 + \beta W_{1} + O(\beta^{2} W_{1}^{2})\Big)
\left(1 - \frac{1}{k!}\sum_{P \in \pathset(\infected)} 
O(\beta Q(P))\right),
\end{aligned}
\end{align*}
using the fact that $|\pathset(\infected)| = k!$. 

We now turn to the case when \(c > 0\). Recall that
\begin{equation*}
L(\graphset_\emptytext, \beta; \infected) = \frac{1}{\binom{n}{k \; \; c}}
\end{equation*}
and
\begin{equation*}
L(\graphset_1, \beta; \infected) = \frac{1}{\mu(\infspace_{k,c}; \beta)} \sum_{\infected' \in \uncensored(\infected)} \frac{1}{\binom{n}{c}} L^u(\graphset_1, \beta; \infected'),
\end{equation*}
where $L^u$ denotes the likelihood~\eqref{EqnUncensored} computed without censoring.
Let \(\uncensored(\infected, c')\) denote the subset of \(\uncensored(\infected)\) where each infection has exactly \(k + c'\) infected vertices (i.e, $c'$ of the infected vertices were censored).
Then $\uncensored(\infected) = \bigcup_{c'=0}^c \uncensored(\infected, c')$, and we may write
\begin{align*}
& \begin{aligned}
R(\graphset_{\emptytext}, \graphset_{1}, \beta; \infected)
&=
\frac{\binom{n}{k \;\; c}}{\mu(\infspace_{k, c}; \beta) }
\cdot \frac{1}{\binom{n}{c}}
\sum_{c' = 0}^{c} \sum_{\infected' \in \uncensored(\infected, c')}
\sum_{P \in \pathset(\infected')} \\ 
&\qquad
\left(\prod_{t = 1}^{k + c'} (1 + \beta \neighbors_{t, \invertex})\right)
\left(\prod_{t = 1}^{k + c'} \frac{n + 1 - t}{n + 1 - t + \beta \neighbors_{t}}\right)
\frac{(n - k - c')!}{n!} .
\end{aligned}
\end{align*}
Again, denote the products inside the sums by \(A(P)\) and \(B(P)\), respectively.
Since $\infected'$ is an uncensored vector, the same argument as before gives
\begin{align*}
& \begin{aligned}
1 + \beta W_{1}(\infected') - O(\beta^{2} W_{1}(\infected'))
&\leq  A(P)
\leq 
1 + \beta W_{1}(\infected') + O\left(\beta^{2} W_{1}(\infected')^{2}\right), \\ 
1 - \beta Q(P)
&\leq 
B(P)
\leq 1.
\end{aligned}
\end{align*}
Hence, the likelihood ratio may be written as
\begin{align*}
& \begin{aligned}
R(\graphset_{\emptytext}, \graphset_{1}, \beta; \infected)
&=
\frac{\binom{n}{k \;\; c}}{\mu(\infspace_{k, c}; \beta) }
\cdot \frac{1}{\binom{n}{c}}
\sum_{\infected' \in \uncensored(\infected)}
\sum_{P \in \pathset(\infected')} \\ 
&\qquad 
\Big(1 + \beta W_{1} + O(\beta^{2} W_{1}(\infected')^{2})\Big)
\Big(1 - O(\beta Q(P))\Big) \frac{(n - |\infected'|)!}{n!} \\ 
&=
\frac{\binom{n}{k \;\; c}}{\mu(\infspace_{k, c}; \beta) }
\cdot \frac{1}{\binom{n}{c}}
\sum_{\infected' \in \uncensored(\infected)}
 \\
&\qquad \frac{1}{\binom{n}{|\infected'|}}
\left(1 + \beta W_{1} + O(\beta^{2} W_{1}(\infected')^{2})\right)
\left(1 - \frac{1}{|\infected'|!} \sum_{P \in \pathset(\infected')}O(\beta Q(P))\right) \\ 
&=
\sum_{\infected' \in \uncensored(\infected)}
D(\beta, k, c, |\infected'|)
\left(1 + \beta W_{1}(\infected') + 
 O(\beta^{2} W_{1}(\infected')^{2})\right) \\ 
 & \qquad \times
\left(1 - \frac{1}{|\infected'|!} \sum_{P \in \pathset(\infected')}O(\beta Q(P))\right).
\end{aligned}
\end{align*}
This completes the proof.
\end{proof}

\begin{proof}[Proof of Theorem~\ref{theoremConstantLikelihood}]
Let \(\infected\) and \(\infected'\) be two infections in \(\infspace_{k, c}\).
Then, there is some \(\pi\) in \(S_{n}\) such that \(\infected = \pi \infected'\).
Thus, we have
\begin{align*}
& \begin{aligned}
L(\theta; \infected)
&=
L(\theta; \pi \infected')
=
L(\pi \pi^{-1} \theta; \pi \infected')
=
L(\theta'; \infected'),
\end{aligned}
\end{align*}
where \(\theta' = \pi^{-1} \theta\). 
Note that \(\theta'\) is an element of \(\Theta_{0}\) by assumption.
Thus, we conclude
\[
\sup_{\theta \in \Theta_{0}} L(\theta, \infected)
=
\sup_{\theta' \in \Theta_{0}} L(\theta', \infected').
\]
The rest of the proof follows immediately from the discussion preceding the theorem.
\end{proof}


\subsection{Risk bounds when $\graphset_0$ is the star}
\label{subsec: proofs for examples}

\begin{proof}[Proof of Proposition~\ref{PropStarRisk}]
By Corollary~\ref{CorRisk}, it suffices to compute the risk bound when the null hypothesis corresponds to the empty graph.

Let $t_\alpha$ denote the rejection threshold of the permutation test. Since $t_\alpha$ is defined to be the $\alpha$-quantile of the edges-within statistic under the null hypothesis, we have
\begin{equation*}
\prob_0(W_{1} \ge t_\alpha) \le \alpha.
\end{equation*}
Thus, it remains to bound the Type II error. Our proof uses Lemma~\ref{LemBdDiff} to derive concentration of $W_{1}(\infected)$ to $\E[W_{1}(\infected)]$. Note that under both $H_0$ and $H_1$, we may apply Lemma~\ref{LemBdDiff} with $X_i$ equal to the identity of the $i^{\text{th}}$ uncensored infected node and $f(X_1, \dots, X_k) = W_{1}(\infected)$. Since each node is involved in at most $D$ edges, we may take and $c_i = D$ for all $1 \le i \le k$. This leads to the following concentration bounds, which hold for all $t > 0$:
\begin{align}
& \begin{aligned}
\label{EqnTaConc}
\prob\left(W_{1}(\infected) - \E[W_{1}(\infected)] \ge t\right) &\le \exp\left(-\frac{2t^2}{kD^2}\right), \\ 
\prob\left(W_{1}(\infected) - \E[W_{1}(\infected)] \le -t\right) &\le \exp\left(-\frac{2t^2}{kD^2}\right).
\end{aligned}
\end{align}

We begin with the following lemma:

\begin{lemma}
\label{LemTa}
The rejection threshold satisfies the bound
\begin{equation*}
t_{\alpha} 
\leq 
\frac{Dk(k - 1)}{2 (n - 1)} + \sqrt{\frac{kD^2}{2} \log \frac{1}{\alpha} }.
\end{equation*}
\end{lemma}

\begin{proof}
We first compute $\expect_0[W(\infected)]$. Let $V_i$ denote the $i^\text{th}$ uncensored vertex that is infected. We may write
\begin{align*}
  &\begin{aligned}
\expect_0[W(\infected)] 
&= 
\expect\left[\frac{1}{2} \sum_{i = 1}^{k} \sum_{j = 1}^{k}\ind\{(V_{i}, V_{j}) \in \edgeset_{1}\}\right]  \\ 
&= \frac{k(k-1)}{2} \cdot \expect\left[\ind\{(V_1, V_2) \in \edgeset_1\} \right] \\
&= \frac{k(k-1)}{2} \cdot \frac{|\edgeset_1|}{\binom{n}{2}} \\ 
&= |\edgeset_1| \frac{k(k-1)}{n(n-1)} \\ 
&=
\frac{D k(k - 1)}{2(n - 1)}
  \end{aligned}
\end{align*}

Note that the last line uses the simple equality \(|\edgeset_{1}| = Dn/2\).
Applying the bound~\eqref{EqnTaConc} with 
\[
t = \sqrt{\frac{kD^2}{2} \log\frac{1}{\alpha}},
\] 
we then have
\begin{equation*}
\prob_0\left(W(\infected) 
\ge
\frac{Dk(k-1)}{2(n-1)} + \sqrt{\frac{kD^2}{2} \log \frac{1}{\alpha}}\right) \le \alpha,
\end{equation*}
implying the desired result.
\end{proof}
We now derive a lower bound for $\E_1[W(\infected)]$. We have the following result:

\begin{lemma}
\label{LemE1W}
Let $\graphset_0$ be a vertex-transitive graph with degree $D$. Then we have the bound
\begin{equation*}
\E_1[W(\infected)] 
\ge
\frac{D}{2} C_k H(\beta).
\end{equation*}
\end{lemma}
\noindent The proof of Lemma~\ref{LemE1W} is fairly technical and is contained in Appendix~\ref{AppLemE1W}.

Combining the result of Lemma~\ref{LemE1W} with the concentration bound~\eqref{EqnTaConc}, we then have
\begin{align}
\label{EqnE1W}
\prob_1(W_{1}(\infected) < t_\alpha) 
& \le 
\prob_1\left(W(\infected) - \E[W(\infected)] < t_\alpha - \frac{D}{2} C_k H(\beta) \right) \notag \\
& \le 
\exp\left\{-\frac{2}{kD^2} \left(\frac{D}{2} C_k H(\beta) - t_\alpha\right)^2 \right\}.
\end{align}
Finally, substituting the bound on $t_\alpha$ from Lemma~\ref{LemTa} yields the required inequality.
\end{proof}


\begin{proof} [Proof of Proposition~\ref{PropStarRiskMult}]
Since the proof parallels the argument in Proposition~\ref{PropStarRisk}, we only highlight the necessary modifications. In particular, inequalities~\eqref{EqnTaConc} may be replaced by the following concentration bounds:
\begin{align}
& \begin{aligned}
\label{EqnTaConcMult}
\prob\left(\Wbar - \E[\Wbar] \ge t\right) &\le \exp\left(-\frac{2mt^2}{kD^2}\right), \\ 
\prob\left(\Wbar - \E[\Wbar] \le -t\right) &\le \exp\left(-\frac{2mt^2}{kD^2}\right).
\end{aligned}
\end{align}
This is due to the fact that we may apply Lemma~\ref{LemBdDiff} to the variables 
$\{X_{\ell,i}\}$, 
where $X_{\ell, i}$ denotes the identity of the $i^{\text{th}}$ uncensored infected node in the $\ell^{\text{th}}$ spreading process, and $M_{Dm} = \Wbar - \expect[\Wbar]$. We may take $c_{\ell,i} = D/m$ for all $(\ell,i)$.

The bound in Lemma~\ref{LemTa} may then be replaced by the following bound on the rejection threshold:
\begin{equation*}
t_\alpha 
\le
\frac{Dk(k-1)}{2(n-1)} + \sqrt{\frac{kD^2}{2m} \log \frac{1}{\alpha}}.
\end{equation*}

Similarly, although Lemma~\ref{LemE1W} remains unchanged, the bound~\eqref{EqnE1W} will be modified with an additional factor of $m$ appearing in the numerator of the exponent.
\end{proof}

\subsection{Risk bounds when $\graphset_1$ is the star}
\label{SecStarRisk}

\begin{proof}[Proof of Proposition~\ref{prop: star center}]
We first derive the maximum likelihood estimator. The likelihoods may be written as
\begin{align*}
  &\begin{aligned}
L_{i}(\beta; \infected) 
&= 
\prob_{i}\left(\infected_{1}\right) 
\prob_{i}\left(\infected | \infected_{1}\right).
\end{aligned}
\end{align*}
Note that we have the equality
\begin{align*}
  &\begin{aligned}
\prob_{0}\left(\infected | \infected_{1}\right) 
&= 
\prob_{1}\left(\infected | \infected_{1}\right),
  \end{aligned}
\end{align*} 
since under both hypotheses, given the infection status of vertex 1, all status assignments of the remaining nodes are equally likely. Hence, the MLE reduces to comparing $\prob_0(\infected_1)$ and $\prob_1(\infected_1)$.

We have
\begin{equation*}
\prob_0(\infected_1 = 1) = \frac{k}{n}, 
\quad \prob_0(\infected_1 = 0) = \frac{n-k}{n},
\end{equation*}
whereas
\begin{equation*}
\prob_1(\infected_1 = 1) > \frac{k}{n}, 
\quad \prob_1(\infected_1 = 0) < \frac{n-k}{n},
\end{equation*}
since the center of the star is more likely to be infected relative to the leaves. Hence, the test that rejects $H_0$ according to the center indicator statistic $\ind\{\infected_{1} = 1\}$ is indeed a maximum likelihood estimator. 
Note that when $\infected_1 = \star$, we may make an arbitrary decision, so we decide to default to $H_0$ in that case.
Finally, observe that the Type I error is controlled by $\alpha$ when $\alpha \ge k/n$.
\end{proof}

\begin{proof}[Proof of Proposition~\ref{prop: star risk}]

We begin with the following lemma, proved in Appendix~\ref{AppLemStar}:
\begin{lemma}
\label{LemStar}
Under the hypothesis that the graph $\graphset_1$ is a star, we have the bounds
\begin{equation*}
\prob_1(\infected_1 = 0) 
\ge 
\exp\left(- \frac{k + \beta k(k - 1) / 2}{n - k}
\right),
\end{equation*}
and
\begin{equation*}
\prob_1(\infected_1 = 0) 
\le
\begin{cases}
\exp\left(- \frac{k + \beta k (k - 1) / 2}{(n - k + 1) + (k-1) \beta}
\right), & \text{if } \beta \ge 1, \\
\exp\left(- \frac{k + \beta k (k - 1) / 2}{n}
\right), & \text{if } \beta < 1.
\end{cases}
\end{equation*}
\end{lemma}

Returning to the proof of the proposition, note that
\begin{align*}
R_{k, 0}(C, \beta) 
&= 
\prob_0(\infected_1 = 1) + \prob_1(\infected_1 = 0)
= 
\frac{k}{n} + \prob_1(\infected_1 = 0).
\end{align*}
Applying the bounds in Lemma~\ref{LemStar} then implies the desired result.
\end{proof}


\begin{proof} [Proof of Proposition~\ref{PropMultRisk}]
By the analog of Corollary~\ref{CorRisk} for multiple spreading processes, it suffices to consider the risk when $\graphset_0$ is the empty graph. 
Let \(t_{\alpha}\) be the level \(\alpha\) threshold.
We wish to bound
\begin{align*}
  &\begin{aligned}
R_{k, 0}(P_{\Cbar, \alpha}, \beta) 
&= 
\prob_{0}\left( \Cbar \geq t_{\alpha}\right)
+ \prob_{1}\left(\Cbar < t_{\alpha}\right)
\leq 
\alpha 
+ \prob_{1}\left(\Cbar < t_{\alpha}\right).
  \end{aligned}
\end{align*}
To bound the Type II error, it suffices to pick any threshold \(t'_{\alpha}\) such that 
\begin{equation}
\prob_{0}\left( \Cbar \geq t'_{\alpha}\right)
\leq 
\alpha.
\label{eqn: alt threshold}
\end{equation}
By definition, this guarantees that \(t_{\alpha} \leq t'_{\alpha}\),
and as a consequence,
\[
\prob_{1}\left(\Cbar < t_{\alpha}\right)
\leq 
\prob_{1}\left(\Cbar < t'_{\alpha}\right).
\]
Accordingly, let
\[
t'_{\alpha}
=
\frac{k}{n}
+ \sqrt{\frac{1}{2m} \log \frac{1}{\alpha}}.
\]
By applying Hoeffding's inequality, we see that \(t'_{\alpha}\) satisfies inequality~\eqref{eqn: alt threshold}.

We will apply Hoeffding's inequality again to bound $\prob_1(\Cbar < t_\alpha')$.
Note that
\begin{equation*}
\expect_{1}[\Cbar]
=
\prob_1(\infected_1 = 1) 
\ge 
p_{k,0}(\beta),
\end{equation*}
by Lemma~\ref{LemStar}.
It follows that
\begin{align*}
\prob_1(\Cbar < t'_\alpha) 
& \le \prob_1\Big(\Cbar - \expect_{1}[\Cbar] \leq t_\alpha - p_{k, 0}(\beta)\Big) \\
& \le 
\exp\left(-2m\left(\frac{k}{n} + \sqrt{\frac{1}{2m} \log \frac{1}{\alpha}} - p_{k,0}(\beta)\right)^2\right),
\end{align*}
implying the desired result.
\end{proof}

\section{Proofs of corollaries}
\label{AppCors}

In this section, we provide proofs of all corollaries stated in the main text.


\subsection{Proof of Corollary~\ref{CorRisk}}
\label{AppCorRisk}

Let $R_0$ and $R_0'$ denote the risks under null hypotheses $H_0$ and $H_0'$, respectively, and let $\scriptA$ denote the rejection region of the test statistic. We have
\begin{align*}
R_0 & = \prob_0(S(\infected) \in \scriptA) + \prob_1(S(\infected) \notin \scriptA), \quad \text{and} \\
R_0' & = \prob'_0(S(\infected) \in \scriptA) + \prob_1(S(\infected) \notin \scriptA),
\end{align*}
where $\prob'_0$ denotes the probability distribution under $H_0'$.

Note that $\Pi_1 \Pi_0 = S_n$ by assumption, and also $\Pi_1 \Pi_0' = S_n$, since $\Pi_0' = S_n$. By Theorem~\ref{ThmPermStat}, we then have
\begin{equation*}
\prob_0(S(\infected) \in \scriptA) = \frac{1}{n!} \sum_{\pi \in S_n} \prob\left(S(\pi J) \in \scriptA\right) = \prob_0'(S(\infected) \in \scriptA),
\end{equation*}
for any fixed infection vector $J \in \infspace_{k,c}$. It follows that $R_0 = R_0'$, as claimed.

\subsection{Proof of Corollary~\ref{cor: vertex transitive}}
\label{AppTransitive}
Suppose $\graphset_0$ is the star and \(\graphset_{1}\) is vertex-transitive. Then \(\Pi_{1}\) contains permutations \(g_{i}\) mapping vertex \(1\) to vertex \(i\), for \(1 \le i \le n\).
Let $G = \{g_1, \dots, g_n\}$, and note that \(G \cap \Pi_{0} = \{g_1\}\).
Furthermore, the cosets \(g_{i} \Pi_{0}\) are unique.
Finally, by equation~\eqref{eqn: pi size},
we have
\[
|G \Pi_{0}|
=
\frac{|G| |\Pi_{0}|}{|G \cap \Pi_{0}|}
=
\frac{n (n - 1)!}{1}
= n!.
\]
Thus, we conclude that \(\Pi = S_{n}\). The proof when $\graphset_1$ is the star graph is analogous.

\subsection{Proof of Corollary~\ref{CorStarRisk0}}
\label{AppCorStarRisk0}

We first show that $|\scriptC_k(u,v)| = (k-1)2^{k-1}$, for any edge $(u,v)$. Note that the number of possible choices for the $k$ infected vertices in a cascade involving $u$ and $v$ is $k-1$, corresponding to segments of $k$ neighboring nodes in the cycle graph. Furthermore, the number of orderings of infected vertices in the segment is $2^{k-1}$, corresponding to whether the infection proceeds to the right or left on each step.

Substituting $C_k = (k-1) 2^{k-1}$ and 
\[
H(\beta) = \prod_{m=1}^{k-1} \frac{\beta}{n-m+2\beta}
\] 
into the risk bound in Proposition~\ref{PropStarRisk} yields the first part of the corollary.


For the second and third part of the corollary, note that
\[
\lim _{\beta \to \infty} H(\beta)
=
\prod_{m = 1}^{k - 1} \frac{1}{2}
= 
2^{-(k - 1)}.
\]
Simple algebra then yields the desired results.


\subsection{Proof of Corollary~\ref{CorStarRisk2}}
\label{AppCorStarRisk2}

The proof of this corollary refers back to the proof of Proposition~\ref{PropStarRisk}. Let $|\edgeset_1'| = n$ denote the number of edges in the cycle graph $\graphset_1'$. If $D' = 2$ denotes the maximum degree of $\graphset_1'$, we still have the bound
\begin{equation*}
t_\alpha 
\le 
|\edgeset_1'| \frac{k(k-1)}{n(n-1)} + \sqrt{\frac{k(D')^2}{2} \log \frac{1}{\alpha}},
\end{equation*}
from Lemma~\ref{LemTa}. The analog of Lemma~\ref{LemE1W}, specialized to the case of an infection spreading over the path graph, is the following bound:

\begin{lemma}
\label{LemE1W2}
Under the alternative hypothesis that $\graphset_1$ is the path graph, we have
\begin{equation*}
\E_1[W_{\graphset_1'}(\infected)] 
\ge
\frac{(n-c)(n-c-1)}{n^2(n-1)} \cdot (k-1) 2^{k-1} \cdot \frac{n-k+1}{n}
\cdot \prod_{m=1}^{k-1} \frac{\beta}{n-m+2\beta}.
\end{equation*}
\end{lemma}
\noindent The proof of Lemma~\ref{LemE1W2} is provided in Appendix~\ref{AppLemE1W2}.

Finally, we have the concentration inequalities
\begin{align*}
  &\begin{aligned}
\prob_{0}\left(W_{\graphset'_{1}}(\infected) - \expect_{0}\left[W_{\graphset_1'}(\infected)\right] \geq t\right)
&\leq 
\exp\left(-\frac{2t^2}{k(D')^2}\right),
\\ 
\prob_{1}\left(W_{\graphset_1'}(\infected) - \E_{1}\left[W_{\graphset_1'}(\infected)\right]\leq  -t\right) 
&\le  
\exp\left(-\frac{2t^2}{k(D')^2}\right).
  \end{aligned}
\end{align*}
Combining the pieces as in the proof of Proposition~\ref{PropStarRisk} then yields the desired bound.


\section{Proofs of supporting lemmas}
\label{AppLemmas}

Finally, we provide proofs of supporting technical lemmas.

\subsection{Proof of Lemma~\ref{LemStar}}
\label{AppLemStar}

Clearly, we have
\begin{equation*}
\prob_1(\infected_1 = 0) = \prob_1(\infected_1 \neq \star) \prob_1(\infected_1 = 0 | \infected_1 \neq \star) = \left(\frac{n-c}{n}\right) \prob_1(\infected_1 = 0 | \infected_1 \neq \star).
\end{equation*}
The latter probability is easier to calculate, since we may consider a process where we first choose $c$ of the vertices $\{2, \dots, n\}$ to censor, and then compute the probability that the $k$ infected nodes lying in the remaining vertex set are all leaf nodes. 
Since vertex 1 is not infected, the spreading process is agnostic to the infection status of the $c$ censored nodes.

We first consider the lower bound. We have
\begin{align*}
\prob_1\left\{\infected_1 = 0 | \infected_1 \neq \star\right\} &= \prod_{j = 0}^{k - 1} \frac{n - c - 1 - j}{(n - c - j) + j \beta} \\ 
& = \exp\left(\sum_{j = 0}^{k - 1} \log \frac{n - c - 1 - j}{(n - c - j) + j \beta}  \right) \\ 
& = \exp\left(-\sum_{j = 0}^{k - 1} \int_{n - c - j - 1}^{(n - c - j) + j \beta} \frac{1}{x} dx \right) \\ 
& \ge \exp\left(- \sum_{j = 0}^{k - 1} \frac{1 + j \beta}{n - c - j - 1} \right) \\ 
& \ge \exp\left(- \frac{k + \beta k(k - 1) / 2}{n - c - k} \right).
\end{align*}
The upper bound may be derived in an analogous fashion. We have
\begin{align*}
\prob_1\left\{\infected_1 = 0 | \infected_1 \neq \star \right\} & = \exp\left(-\sum_{j = 0}^{k - 1}\int_{n - c - j - 1}^{(n - c - j) + j\beta} \frac{1}{x} dx \right) \\
& \le \exp\left(-\sum_{j = 0}^{k - 1} \frac{1 + j\beta}{(n - c) + j(\beta - 1)} \right).
\end{align*}
When $\beta \ge 1$, the denominator is maximized for $j = k-1$; when $\beta < 1$, the denominator is maximized for $j = 0$. In the first case, we have
\begin{align*}
\prob_1\{\infected_1 = 0 | \infected_1 \neq \star\} & \le \exp\left(- \frac{1}{(n - c) + (k-1)(\beta - 1)} \sum_{j = 0}^{k - 1} (1 + j\beta) \right) \\
& = \exp\left(- \frac{k + \beta k (k - 1) / 2}{(n - c - k + 1) + (k-1)\beta} \right).
\end{align*}
In the second case, we have
\begin{align*}
\prob_1\{\infected_1 = 0 | \infected_1 \neq \star\} & \le \exp\left(- \frac{1}{n - c} \sum_{j = 0}^{k - 1} (1 + j\beta) \right) 
 = \exp\left(- \frac{k + \beta k (k - 1) / 2}{n-c} \right).
\end{align*}

\subsection{Proof of Lemma~\ref{LemE1W}}
\label{AppLemE1W}

We begin by writing
\begin{align*}
\expect _1[W(\infected)] 
&= 
\sum_{(u, v) \in \edgeset_{1}} \prob_1\left(\infected_u = \infected_v = 1\right) \\ 
& 
\ge \sum_{(u,v) \in \edgeset_1} |\scriptC_k(u,v)| \cdot \frac{\binom{n-2}{c}}{\binom{n}{c}} 
\cdot \frac{1}{n} \prod_{m = 1}^{k - 1} \frac{\beta}{n - m + \beta (2 + m(D - 2))}.
\end{align*}
Indeed, the inequality comes from restricting our consideration to infections where the $k - 1$ nodes after the first are infected along edges of the graph rather than by random infection, and $u$ and $v$ are in the initial infection set. The term $|\scriptC_k(u,v)|$ counts the number of such infection paths, and the term $\binom{n-2}{c} / \binom{n}{c}$ computes the probability that $u$ and $v$ will remain uncensored at the end of the process. Finally, the factor
\begin{equation*}
\prod_{m = 1}^{k - 1} \frac{\beta}{n - m + \beta (2 + m(D - 2))}
\end{equation*}
lower-bounds the probability that each of the $k-1$ vertices after the first contracts the disease from one of its infected neighbors. 
Note that at each stage, the total number of edges connecting the $n-m$ uninfected nodes to the $m$ previously infected nodes is bounded above by $2 + m(D-2)$, since the infected nodes are necessarily connected to each other.

Recalling the definition of $H(\beta)$, we then have
\begin{align*}
\E_1[W(\infected)] & \ge \frac{(n-c)(n-c-1)}{n^2(n-1)} \cdot H(\beta) \cdot \sum_{(u,v) \in \edgeset_1} |\scriptC_k(u,v)| \\
& \ge |\edgeset_1| \frac{(n-c)(n-c-1)}{n^2(n-1)} C_k H(\beta),
\end{align*}
as wanted.


\subsection{Proof of Lemma~\ref{LemE1W2}}
\label{AppLemE1W2}

As in the proof of Lemma~\ref{LemE1W}, we begin by writing
\begin{align*}
\E_1\left[W_{\graphset_1'}(\infected)\right] 
& = 
\sum_{(u,v) \in \edgeset_1'} \prob_1(\infected_u = \infected_v = 1) \\
& \ge 
\sum_{(u,v) \in \edgeset_1'} |\scriptC_k'(u,v)| \cdot \frac{\binom{n-2}{c}}{\binom{n}{c}} 
\cdot \frac{1}{n} \prod_{m=1}^{k-1} \frac{\beta}{n-m + 2\beta},
\end{align*}
where $\scriptC_k'(u,v)$ denotes the set of cascades involving $(u,v)$ in $\graphset_1'$.

We claim that
\begin{equation}
\label{EqnCascadeCt}
\sum_{(u,v) \in \edgeset_1'} |\scriptC_k'(u,v)| = (k-1)(n-k+1) \cdot 2^{k-1},
\end{equation}
from which the result follows. Indeed, for $(u,v) = (i,i+1)$, with $1 \le i \le k-1$, we have
\begin{equation*}
|\scriptC_k'(u,v)| = i \cdot 2^{k-1},
\end{equation*}
since we have $i$ choices for the collection of infected vertices in the cascade, and given a collection of vertices, the infection may spread according to $2^{k-1}$ different orderings. Similarly, we may argue that
\begin{align*}
|\scriptC_k'(u,v)| & = (n-i) \cdot 2^{k-1}, \qquad \text{for } n-k+1 \le i \le n-1, \\
|\scriptC_k'(u,v)| & = (k-1) \cdot 2^{k-1}, \qquad \text{for } k \le i \le n-k.
\end{align*}
Summing up over all choices of $(u,v)$ then yields
\begin{align*}
\sum_{(u,v) \in \graphset_1'} |\scriptC_k'(u,v)| & = k(k-1) \cdot 2^{k-1} + (k-1)(n-2k+1) \cdot 2^{k-1} \\
& = (k-1)(n-k+1) \cdot 2^{k-1},
\end{align*}
which is equation~\eqref{EqnCascadeCt}.

\section{Proofs for general graphs}
\label{app:general_graphs:proofs}

In this section, we provide the proofs for the proposed discretization strategy and the validity of the proposed hypothesis testing algorithms. We also provide a discretization mechanism for the Ising model.

\subsection{Discretization proof}
\label{SecDiscProof}

The purpose of this section is to prove Proposition~\ref{propDiscretizationMain}.
To assist, we start with the following proposition:

\begin{proposition}
Let \(\betamin = \min\{\beta, \beta'\}\) and \(\betamax = \max\{\beta, \beta'\}\).
For finite \(\beta\) and \(\beta'\), we have the bound
\begin{align*}
&\begin{aligned}
\tvdist\left(\prob_{\beta}, \prob_{\beta'}\right)
&\leq 
\min\left\{
|\beta - \beta'| 
M_{0} , \;
\frac{k + c}{2} \log\frac{\betamax}{\betamin}, \; 
\frac{N_{0}}{\betamin}
\right\}.
\end{aligned}
\end{align*}
\label{propBetaBound}
\end{proposition}


\begin{proof}
Let \(\infected\) and \(\infected'\) denote infections generated according to \(\prob_{\beta}\) and \(\prob_{\beta'}\), respectively.
First, we want a simpler representation of \(\infected\) and \(\infected'\).
Let \(C\) be \(c\) censored random variables, which may be chosen in any manner, such as uniformly random choices or a fixed set of vertices, as long as this choice is independent of the path random variables to be defined now.
Let \(\pathvar\) be a path random variable (vector) of length \(k + c\) for parameter \(\beta\).
To generate this, let \(X_{t}\) denote the indicator of a spread over the edges at time \(t\), i.e.,
\[
\prob\left(X_{t} = 1 | \pathvar_{1:t - 1} = P_{1:t - 1}\right)
=
\frac{\wneighbors_{t}(P_{1:t - 1}) \beta}{n + 1 - t + \wneighbors_{t}(P_{1:t - 1}) \beta}.
\]
Next, we denote \(Y_{t}\) to be a uniform choice a random vertex, i.e.,
\[
\prob\left(Y_{t} = v | \pathvar_{1:t - 1} = P_{1:t - 1}\right)
=
\begin{cases}
\frac{1}{n + 1 - t}, & v \not \in \pathvar_{1:t - 1}, \\
0, & v \in \pathvar_{1:t-1}.
\end{cases}
\] 
Further, if \(X_{t} = 0\), then \(\pathvar_{t} = Y_{t}\).
Finally, we denote \(Z_{t}\) to be a random vertex chosen according to
\[
\prob\left(Z_{t} = v | \pathvar_{1:t - 1} = P_{1: t - 1}\right)
=
\frac{\wneighbors_{t, v}(P_{1: t - 1})}{ \wneighbors_{t}(P_{1:t - 1})},
\]
and if \(X_{t} = 1\), then we set \(\pathvar_{t} = Z_{t}\).
Next, we can write
\begin{align*}
&\begin{aligned}
\infected
&=
f(C, \pathvar)
=
g(C, X_{1:k + c}, Y_{1:k + c}, Z_{1:k + c}),
\end{aligned}
\end{align*}
where the functions \(f\) and \(g\) may be defined as follows:
For \(f\), the vertices of \(C\) are censored.
Then, the first \(k\) uninfected variables of \(\pathvar\) are chosen to be infected, and the remaining choices are ignored.
The function \(g\) is chosen similarly, since we are simply breaking the path into its various choices.
Note that this gives the desired measure on \(\infected\).
Similarly, we can write
\begin{align*}
&\begin{aligned}
\infected'
&=
f(C, \pathvar')
=
g(C, X'_{1:k + c}, Y'_{1:k + c}, Z'_{1:k + c}),
\end{aligned}
\end{align*}
for analogous choices \(X'_{t}\), \(Y'_{t}\), and \(Z'_{t}\) for the parameter \(\beta'\) and path \(\pathvar'\).
Further, we couple \(\pathvar\) and \(\pathvar'\) such that if \(\pathvar_{1:t - 1} = \pathvar'_{1:t - 1}\), then \(Y_{t} = Y'_{t}\) and \(Z_{t} = Z'_{t}\). Additionally, assume that \(X_{t}\) and \(X'_{t}\) are maximally coupled.

Let \(\prob_{X}\) denote the measure with respect to a random variable \(X\).
The first step is to use the data-processing inequality of Lemma~\ref{lemmafDivDPI} and Lemma~\ref{lemmaTVCoupling} to obtain
\begin{align*}
&\begin{aligned}
\tvdist\left(\prob_{\beta}, \prob_{\beta'}\right)
&\leq 
\tvdist\left(\prob_{C, X_{1:k + c}, Y_{k + c}, Z_{k + c}},
\prob_{C, X'_{1:k + c}, Y'_{k + c}, Z'_{k + c}}
\right) \\
&\leq 
\prob\left((X'_{1:k + c}, Y'_{k + c}, Z'_{k + c}) \neq (X'_{1:k + c}, Y'_{k + c}, Z'_{k + c}) 
\right) \\
&\leq 
\prob\left(
X_{1: k + c} \neq X'_{1:k + c}
\right),
\end{aligned}
\end{align*}
where the final inequality follows due to the coupling.

Next, we define the set \(\pathset_{t}\) to be
\[
\pathset_{t}
:=
\left\{
(X_{1:t}, Y_{1:t}, Z_{1:t}, X'_{1:t}, Y'_{1:t}, Z'_{1:t}):
X_{1:t} = X'_{1:t}
\right\}.
\]
Then, we have
\begin{align}
&\begin{aligned}
\tvdist\left(\prob_{\beta}, \prob_{\beta'}\right)
&\leq
\sum_{t = 1}^{k + c} 
\prob\left(X_{1:t - 1} = X'_{1:t - 1}\right)
\prob\left(X_{t} \neq X'_{t} | X_{1:t - 1} = X'_{1:t - 1}\right) \\
&\leq 
\sum_{t = 1}^{k + c}
\sum_{E \in \pathset_{t - 1}} \prob(E) \prob\left(X_{t} \neq X_{t}' | E\right) \\
&\leq 
\sum_{t = 1}^{k + c} \max_{E \in \pathset_{t - 1}}
\prob\left(X_{t} \neq X_{t}' | E\right). 
\label{eqnSplitFiniteInfinite}
\end{aligned}
\end{align}

Inequality~\eqref{eqnSplitFiniteInfinite} yields
\begin{align*}
&\begin{aligned}
\tvdist\left(\prob_{\beta}, \prob_{\beta'}\right)
&\leq 
\sum_{t = 1}^{k + c} \max_{P_{1:t - 1}}
\biggr|
\frac{\wneighbors_{t}(P_{1:t - 1}) \beta}{n + 1 - t + \wneighbors_{t}(P_{1:t - 1}) \beta} 
-
\frac{\wneighbors_{t}(P_{1:t - 1}) \beta'}{n + 1 - t + \wneighbors_{t}(P_{1:t - 1}) \beta'}
\biggr|,
\end{aligned}
\end{align*}
where this follows from the maximal coupling of \(X_{t}\) and \(X'_{t}\), i.e., a Bernoulli \(B_{p}\) of parameter \(p\) and a Bernoulli \(B_{q}\) of parameter \(q\) are maximally coupled if \(\tvdist(\prob_{B_{p}}, \prob_{B_{q}}) = |p - q|\).

The next step is to examine the function
\[
h(x)
=
\frac{ax}{b + ax},
\]
where \(a, b \geq 0\).
The derivative of \(h\) is
\begin{align*}
&\begin{aligned}
h'(x)
&=
\frac{ab}{(b+ax)^2}
\leq 
\min\left\{\frac{a}{b}, \;
\frac{1}{2x}, \;
\frac{b}{ax^{2}}
\right\}.
\end{aligned}
\end{align*}

We apply this to our present problem to obtain
\begin{align*}
&\begin{aligned}
\tvdist\left(\prob_{\beta}, \prob_{\beta'}\right)
&\leq
\sum_{t = 1}^{k + c}
\max_{P_{1:t - 1}}
\int_{\betamin}^{\betamax} 
\min\left\{
\frac{\wneighbors_{t}(P_{1:t - 1})}{n + 1 - t}, \;
\frac{1}{2x}, \;
\frac{n + 1 - t}{\wneighbors_{t}(P_{1:t - 1})x^{2}}
\ind\left\{
\wneighbors_{t}(P_{1:t - 1}) > 0
\right\}
\right\}
dx \\
&\leq 
\sum_{t = 1}^{k + c}
\max_{P_{1:t - 1}}
\min\left\{
\int_{\betamin}^{\betamax} 
\frac{\wneighbors_{t}(P_{1:t - 1})}{n + 1 - t} dx, \;
\int_{\betamin}^{\betamax} 
\frac{1}{2x}dx, \; 
\right. \\ &\qquad \left.
\int_{\betamin}^{\betamax} 
\frac{n + 1 - t}{\wneighbors_{t}(P_{1:t - 1})x^{2}}
\ind\left\{
\wneighbors_{t}(P_{1:t - 1}) > 0
\right\}
dx
\right\} \\
&=
\sum_{t = 1}^{k + c}
\max_{P_{1:t - 1}}
\min\left\{
|\beta - \beta'| 
\frac{\wneighbors_{t}(P_{1:t - 1})}{n + 1 - t}, \;
\frac{1}{2} \log\frac{\betamax}{\betamin}, \; 
 \right. \\ &\qquad \left.
\frac{n + 1 - t}{\wneighbors_{t}(P_{1:t - 1})}
\left(\frac{1}{\betamin} - \frac{1}{\betamax}
\right)
\ind\left\{
\wneighbors_{t}(P_{1:t - 1}) > 0
\right\}
\right\}
\\
&\leq 
\min\left\{
|\beta - \beta'| 
\sum_{t = 1}^{k + c}
\max_{P_{1:t - 1}}
\frac{\wneighbors_{t}(P_{1:t - 1})}{n + 1 - t}, \;
\frac{k + c}{2} \log\frac{\betamax}{\betamin}, \; 
 \right. \\ &\qquad \left.
\frac{1}{\betamin}
\sum_{t = 1}^{k + c}
\max_{P_{1:t - 1}}
\frac{n + 1 - t}{\wneighbors_{t}(P_{1:t - 1})}
\ind\left\{
\wneighbors_{t}(P_{1:t - 1}) > 0
\right\}
\right\} \\
&\leq 
\min\left\{
|\beta - \beta'| 
M_{0} , \;
\frac{k + c}{2} \log\frac{\betamax}{\betamin}, \; 
\frac{k + c}{\betamin}
\left(n - \frac{k + c - 1}{2} \right)
\right\}  \\
&=
\min\left\{
|\beta - \beta'| 
M_{0} , \;
\frac{k + c}{2} \log\frac{\betamax}{\betamin}, \; 
\frac{N_{0}}{\betamin}
\right\},
\end{aligned}
\end{align*}
where in the final inequality we use the assumption that \(W_{t}(P_{1:t - 1}) \geq 1\) or that \(
\wneighbors_{t}(P_{1:t - 1}) = 0\),
and this completes the proof.
\end{proof}
We can now prove our main proposition:

\begin{proof}[Proof of Proposition~\ref{propDiscretizationMain}]
Let \(\beta\) be an element of \(\completespace\), and let \(\infected\) be an infection generated on \(\graphset\) with parameter \(\beta\).
We need to show that \(\completespace_{D}\) has an element \(\beta'\) such that an infection \(\infected'\) on \(\graphset\) with parameter \(\beta'\) has a distribution close to that of \(\infected\). 
We consider the cases where \(\beta\) is in \([0, M_{2} / M_{1}]\),  \([M_{2} / M_{1}, N_{0} / \delta]\), and \([N_{0} / \delta, \infty]\).

First, suppose \(\beta\) is in \([0, M_{2} / M_{1}]\). 
Then \(\beta' = F_{D}(\beta)\) satisfies \(|\beta - \beta'| \leq \frac{1}{2M_1}\).
By Proposition~\ref{propBetaBound}, we have
\begin{align*}
& \begin{aligned}
\tvdist\left(\prob_{\beta}, \prob_{\beta'}\right)
&\leq 
|\beta - \beta'|
M_{0}
\leq 
\frac{M_{0}}{2M_{1}}
\leq 
\delta,
\end{aligned}
\end{align*}
which is what we wanted to show.

Next, we consider the case of \(\beta\) in \([M_{2} / M_{1}, N_{0} / \delta]\).
Let \(\beta' = F_{D}(\beta)\).
Then we have
\[
\betamax 
\leq 
\exp\left(\frac{2 \delta}{k + c}\right) 
\betamin.
\]
Using this and Proposition~\ref{propBetaBound}, we have
\begin{align*}
& \begin{aligned}
\tvdist\left(\prob_{\beta}, \prob_{\beta'}\right)
&\leq 
\frac{k + c}{2} \log \frac{\betamax}{\betamin}
\leq 
\delta.
\end{aligned}
\end{align*}
Thus, we once again see that \(\beta\) and \(\beta'\) satisfy equation~\eqref{eqnDiscretizationFunction}.

Finally, we consider the case of \(\beta \geq N_{0} / \delta\).
Then we have \(\beta' = F_{D}(\beta) \geq N_{0} / \delta\), and by Proposition~\ref{propBetaBound}, we have
\begin{align*}
& \begin{aligned}
\tvdist\left(\prob_{\beta}, \prob_{\infty}\right)
&\leq 
\frac{N_{0}}{\betamin}
\leq
\delta.
\end{aligned}
\end{align*}
This completes the proof of the proposition.
\end{proof}

\subsection{Algorithm proofs}

Our main goal is to prove Theorem~\ref{theoremPCDiscretizationAlgorithm}.
This this end, we separate the Type I error into two components: the ``actual'' Type I error for an unusual observation and a Type I error resulting from simulation inaccuracies.
We start with a lemma to this effect.

\begin{lemma}
Let \(t\) be a real number, let \(\hat{t}\) be a random threshold independent of \(\infected\), and let \(S: \infspace_{k, c} \to \reals\) be a statistic.
Then we have
\[
\prob_{0, \beta}\left(S(\infected) > \hat{t}\right)
\leq 
\prob_{0, \beta}\left(S(\infected) > t\right)
+
\prob_{0, \beta}\left(\hat{t} < t\right).
\]
\label{lemmaTypeIerrors}
\end{lemma} 
The proof is a series of straightforward manipulations that we give in Section~\ref{sec:disc_technical:additionalProofs}.
Next, we want to use basic concentration results to show that quantiles of statistics concentrate sufficiently nicely.

\begin{lemma}
Let \(S: \infspace_{k, c} \to \reals\) be a statistic.
If 
\[
N_{\sims} 
\geq 
\left(\frac{1}{2\epsilon^{2}} + \frac{2}{3 \epsilon}\right) \log\frac{1}{\xi},
\]
then we have
\[
\prob_{0, \beta}\left(
\hat{t}_{\alpha - \epsilon, \beta} < t_{\alpha}
\right)
\leq 
\xi.
\]
\label{lemmaQuantileConcentration}
\end{lemma}

The proof of this lemma is also given in Section~\ref{sec:disc_technical:additionalProofs}.
Finally, we can prove Theorem~\ref{theoremPCDiscretizationAlgorithm}.

\begin{proof}[Proof of Theorem~\ref{theoremPCDiscretizationAlgorithm}]
Let $\gamma = \delta + \xi$. First, we use Proposition~\ref{propDiscretizationMain}, Lemma~\ref{lemmaTypeIerrors}, and inequality~\eqref{eqnGeqThreshold} to obtain
\begin{align*}
& \begin{aligned}
\sup_{\beta \in \paramset_{0}} \prob_{0, \beta}\left(S(\infected) > \hat{t}_{\alpha - \gamma - \epsilon, \beta}\right)
&\leq 
\sup_{\beta \in \paramset_{0, D}} \prob_{0, \beta}\left(S(\infected) > \hat{t}_{\alpha - \gamma - \epsilon, \beta}\right)
+
\delta \\
&\leq 
\sup_{\beta \in \paramset_{0, D}} \prob_{0, \beta}\left(
S(\infected) > t_{\alpha - \gamma, \beta}\right)
+
\prob_{0, \beta}\left(\hat{t}_{\alpha - \gamma - \epsilon, \beta} < t_{\alpha - \gamma, \beta}\right) 
+ \delta \\
&\leq 
\alpha - \gamma + \delta
+
\sup_{\beta \in \paramset_{0, D}}
\prob_{0, \beta}\left(\hat{t}_{\alpha - \gamma - \epsilon, \beta} < t_{\alpha - \gamma, \beta}\right) \\
&\leq 
\alpha - \gamma + \delta + \xi 
\\ &= 
\alpha,
\end{aligned}
\end{align*}
where the last inequality follows by Lemma~\ref{lemmaQuantileConcentration}.
This completes the proof.
\end{proof}

\subsection{Two-statistic test}
\label{SecTwoStatistic}

The purpose of this section is to prove Corollary~\ref{propTwoStatistic}.
The proof is similar to that of Theorem~\ref{theoremPCDiscretizationAlgorithm}, although we need to be slightly more careful with the discretization since we do not require one of \(-S_{0}(\infected)\) or \(S_{1}(\infected)\) to be larger than all of their \(1 - (\alpha_{i} - \gamma - \epsilon)\) quantiles.

\begin{proof}[Proof of Proposition~\ref{propTwoStatistic}]
Again define $\gamma = \delta + \epsilon$. Let \(F_{D}: \paramset_{0} \to \paramset_{0, D}\) be the discretization function.
Define the function
\[
\hat{s}_{1, \alpha}(\beta)
:=
\hat{s}_{1, \alpha, F_{D}(\beta)}.
\]

Now, we split the Type I error into two terms, as follows:
\begin{align*}
& \begin{aligned}
\typeone
&=
\sup_{\beta_{0} \in \paramset_{0}}
\prob _{0, \beta_{0}}\left(
\left\{S_{1}(\infected) > \hat{s}_{1, \alpha_{1} - \gamma - \epsilon}(\beta_{0})\right\}
\cup 
\left\{
S_{0}(\infected) < \hat{s}_{0, 1  - (\alpha_{0} - \gamma - \epsilon)}(\beta_{0})
\right\}
\right)
\\ 
&\leq 
\sup_{\beta_{0} \in \paramset_{0}}
\prob _{0, \beta_{0}}\left(
S_{1}(\infected) > \hat{s}_{1, \alpha_{1} - \gamma - \epsilon} (\beta_{0})
\right)
+ 
\sup_{\beta_{0} \in \paramset_{0}}
\prob _{0, \beta_{0}}\left(
S_{0}(\infected) < \hat{s}_{0, 1  - (\alpha_{0} - \gamma - \epsilon)} (\beta_{0})
\right).
   \end{aligned}
\end{align*}
Since the analyses of the \(S_{0}\) and \(S_{1}\) terms is identical up to a sign, we focus on the term involving \(S_{1}\),
which we denote by \(\typeone_{1}\).

At this point, we observe that 
\[
\hat{s}_{1, \alpha_{1} - \gamma - \epsilon}(\beta)
=
\hat{s}_{1, \alpha_{1} - \gamma - \epsilon}(\beta')
\]
whenever \(F_{D}(\beta) = F_{D}(\beta')\).
Thus, we have
\begin{align*}
& \begin{aligned}
\typeone_{1} 
&=
\sup_{\beta_{0} \in \paramset_{0}}
\prob _{0, \beta_{0}}\left(
S_{1}(\infected) > \hat{s}_{1, \alpha_{1} - \gamma - \epsilon}(\beta_{0})
\right) \\
&=
\sup_{\beta_{0} \in \paramset_{0, D}}
\sup_{\beta' \in F_{D}^{-1}(\beta_{0})}
\prob _{0, \beta'}\left(
S_{1}(\infected) > \hat{s}_{1, \alpha_{1} - \gamma - \epsilon}(\beta_{0})
\right) \\
&=
\sup_{\beta_{0} \in \paramset_{0, D}}
\sup_{\beta' \in F_{D}^{-1}(\beta_{0})}
\prob _{0, \beta'}\left(
S_{1}(\infected) > \hat{s}_{1, \alpha_{1} - \gamma - \epsilon, \beta_{0}}
\right) \\
&\leq 
\sup_{\beta_{0} \in \paramset_{0, D}}
\prob _{0, \beta_{0}}\left(
S_{1}(\infected) > \hat{s}_{1, \alpha_{1} - \gamma - \epsilon, \beta_{0}}
\right) + \delta,
   \end{aligned}
\end{align*}
where the inequality is due to Proposition~\ref{propDiscretizationMain}.

Now, the proof follows that of Theorem~\ref{theoremPCDiscretizationAlgorithm}.
Using Lemma~\ref{lemmaTypeIerrors}, equation~\eqref{eqnGeqThreshold}, and Lemma~\ref{lemmaQuantileConcentration}, we obtain
\begin{align*}
& \begin{aligned}
\typeone_{1}
&\le
\sup_{\beta_{0} \in \paramset_{0, D}}
\prob _{0, \beta_{0}}\left(
S_{1}(\infected) > s_{1, \alpha_{1} - \gamma, \beta_{0}}
\right)
+
\prob _{0, \beta_{0}}\left(
\hat{s}_{1, \alpha_{1} - \gamma, \beta_{0}}
< 
s_{1, \alpha_{1} - \gamma, \beta_{0}}
\right) 
+ \delta \\
&\leq 
\alpha_{1} - \gamma + \delta +
\sup_{\beta_{0} \in \paramset_{0, D}}
\prob _{0, \beta_{0}}\left(
\hat{s}_{1, \alpha_{1} - \gamma, \beta_{0}}
< 
s_{1, \alpha_{1} - \gamma, \beta_{0}}
\right) \\
&\leq 
\alpha_{1} - \gamma + \delta +
\xi \\ 
&\leq 
\alpha_{1}.
   \end{aligned}
\end{align*}
Since the term for \(S_{0}\) may be computed similarly, we see that
\(
\typeone
\leq 
\alpha_{1} + \alpha_{0}
= 
\alpha,
\)
which completes the proof.
\end{proof}

\subsubsection{Confidence sets}

We also provide an additional confidence set algorithm.
This algorithm mirrors the two-statistic test of Algorithm~\ref{algorithmTwoStatistic}.
\begin{algorithm}[!h]
\SetKwInOut{Input}{Input}
\Input{Type I error tolerances \(\alpha_{0}\) and \(\alpha_{1}\) where $\alpha_{0} + \alpha_{1} = \alpha$, approximation parameters \(\epsilon\), \(\delta\), and \(\xi\), observed infection vector $\infected$, set of graphs \(\graphspace\), statistics \(S_{0}\) and \(S_{1}\), discretization \(\paramset_{\graphset, D}\) of size \(N_{\graphset}\) for each \(\graphset\) in \(\graphspace\), discretization functions \(F_{\graphset, D}\)}

Define 
\(N_{\sims} \geq \left(\frac{1}{2\epsilon^{2}} + \frac{8}{3 \epsilon}\right) \log \frac{1}{\xi}\)

For each \(\beta\) in \(\paramset_{D}\), simulate an infection \(N_{\sims}\) times on \(\graphset\) to obtain the approximate the \((\alpha_{0} - \gamma - \epsilon)\) quantile \(\hat{s}_{0, 1 - \alpha_{0} - \delta - \xi - \epsilon, \beta}\) and the \(1 - (\alpha_{1} - \delta - \xi - \epsilon)\) quantile \(\hat{s}_{1, (\alpha_{1} - \delta - \xi - \epsilon), \beta}\)

Define the discrete confidence set
\[
C_{\graphset, D}
:=
\left\{
\beta \in \paramset_{D}:
S_{0}(\infected) \geq \hat{s}_{0, 1 - \alpha_{0} - \delta - \xi - \epsilon, \beta}
\text{ and }
S_{1}(\infected) \leq \hat{s}_{1, (\alpha_{1} - \delta - \xi - \epsilon), \beta}
\right\}
\]

Return the confidence set \(C = \bigcup_{\graphset \in \graphspace} 
\left(\graphset, F_{\graphset, D}^{-1}(C_{\graphset, D})\right)\)
\caption{Two-Statistic Confidence Set}
\label{algorithmTwoStatisticConfidenceSet}
\end{algorithm}

\begin{corollary}
If for each \(\graphset\) in \(\graphspace\) the discretization \(\paramset_{\graphset, D}\) and discretization function \(F_{\graphset, D}\) satisfy equation~\eqref{eqnDiscretizationFunction}, then the set \(C\) given by Algorithm~\ref{algorithmTwoStatisticConfidenceSet} is an \((1 - \alpha)\)-confidence set.
\label{corTwoStatisticConfidenceSet}
\end{corollary}

\subsection{Threshold Function Test}
\label{SecThreshold}

In this section, we present the details of the threshold function test.
We start by introducing the algorithm, and then we prove the validity of the test.

\subsubsection{Algorithm}

The main difficulty is computing an appropriate threshold function for all values of \(\beta_{0}\) and \(\beta_{1}\) for a statisic \(S\) of the form given in equation~\eqref{eqnParameterDependentStatistic}.
A straightforward solution to this problem is to impose a Lipschitz assumption on \(S_{1}(I; \cdot)\) and \(S_{0}(I; \cdot)\), which leads to the Lipschitz continuity of the threshold.
Let \(\paramset_{1, D}\) be any finite subset of \(\paramset_{1}\),
and let \(\paramset_{D} = \paramset_{0, D} \times \paramset_{1, D}\).
Define the \((L_{0}, L_{1})\)-Lipschitz-penalized empirical quantile
\[
\tilde{t}_{\alpha, L_{0}, L_{1}}(\beta_{0}, \beta_{1})
=
\hat{t}_{\alpha, \beta_{0}'}(\beta_{0}', \beta_{1}')
+
\left(
L_{0}|\beta_{0} - \beta_{0}'|
+
L_{1}|\beta_{1} - \beta_{1}'|
\right),
\]
where \(\vect{\beta}' = F_{D}(\vect{\beta}) = (F_{0, D}(\beta_{0}), F_{1, D}(\beta_{1}))\), for a discretization function \(F: \paramset_{0} \times \paramset_{1} \to \paramset_{D}\).

\begin{algorithm}[!h]
\SetKwInOut{Input}{Input}
\Input{Type I error tolerance $\alpha > 0$, approximation parameters \(\epsilon\), \(\delta\), and \(\xi\), observed infection vector $\infected$, null graph \(\graphset_{0}\), statistic \(S\), discretization \(\paramset_{D}\), discretization function \(F_{D}\)}

Define 
\(
N_{\sims} 
\geq \left(\frac{1}{2\epsilon^{2}} + \frac{8}{3\epsilon} \right) \log
\frac{1}{\xi}
\)

For each \((\beta_{0}, \beta_{1})\) in \(\paramset_{D}\), simulate an infection \(N_{\sims}\) times on \(\graphset_{0}\) to obtain the approximate \(1 - (\alpha - \delta - \xi - \epsilon)\) quantile \(\hat{t}_{\alpha - \delta - \xi - \epsilon, \beta_{0}}\)

Compute the function \(\tilde{t}_{\alpha, L_{0}, L_{1}}(\beta_{0}, \beta_{1})\) with discretization function \(f\)

Reject the null hypothesis if 
\(
S(\infected; \beta_{0}, \beta_{1}) 
> 
\tilde{t}_{\alpha - \delta - \xi - \epsilon}(\beta_{0}, \beta_{1})
\) for all \((\beta_{0}, \beta_{1}) \in \paramset_{0} \times \paramset_{1}\)
\caption{Parameter-Dependent Lipschitz Test}
\label{algorithmPDSLipschitzAS}
\end{algorithm}

\begin{proposition}
Let \(S\) have the form given by equation~\eqref{eqnParameterDependentStatistic}.
Additionally, suppose \(S_{1}(I; \cdot)\) is \(L_{1}\)-Lipschitz and \(S_{0}(I; \cdot)\) is \(L_{0}\)-Lipschitz with probability \(1\).
If the discretization \(\paramset_{0, D}\) with discretization function \(F_{0, D}\) satisfies equation~\eqref{eqnDiscretizationFunction}, 
then Algorithm~\ref{algorithmPDSLipschitzAS} using \(\tilde{t}_{\alpha, L_{0}, L_{1}}\) controls the Type I error at level \(\alpha\).
\label{propPDSLipschitzAS}
\end{proposition}

Note that in order to obtain a discretization and discretization function satisfying equation~\eqref{eqnDiscretizationFunction}, it is sufficient to use the discretization that we introduced in Definition~\ref{defPinskerCouplingDiscretization}.

While this is a good first step toward using likelihood approximations, this result may be too conservative to be useful for certain graphs.
The main problem is the use of almost-sure Lipschitz assumptions, where the almost-sure Lipschitz constants may be far larger than what is seen in typical infections.
A natural idea is to relax the almost-sure Lipschitz constants into high-probability Lipschitz constants that can be approximated via simulation.
Unfortunately, doing this with the approximation techniques we have developed thus far does not seem to be possible.

\subsubsection{Proofs}

The goal of this section is to prove Proposition~\ref{propPDSLipschitzAS}.
We start by proving some useful lemmas.

\begin{lemma}
Let \(S(I; \beta_{0}, \beta_{1}) = S_{1}(I, \beta_{1}) - S_{0}(I, \beta_{0})\) where \(S_{0}(I, \cdot)\) and \(S_{1}(I, \cdot)\) are almost surely \(L_{0}\)-Lipschitz and \(L_{1}\)-Lipschitz, respectively.
Then \(t_{\alpha, \beta_{0}}(\beta_{0}', \beta_{1}')\) is \(L_{1}\)-Lipschitz in \(\beta_{1}'\) and \(L_{0}\)-Lipschitz in \(\beta_{0}'\).
\label{lemmaLipschitzThreshold}
\end{lemma}

\begin{proof}
Define the set of supporting infections \(X(\beta_{0}, \beta_{1})\) as
\[
X(\beta_{0}, \beta_{1})
:=
\left\{
I \in \infspace_{k, c}:
S(I; \beta_{0}, \beta_{1}) > t_{\alpha, \beta_{0}}(\beta_{0}, \beta_{1})
\right\}.
\]
Note that we have 
\[
t_{\alpha, \beta_{0}}(\beta_{0}, \beta_{1}) 
= 
\max\{S(I; \beta_{0}, \beta_{1}): I \not\in X(\beta_{0}, \beta_{1})\}.
\]
Thus, we have \(\prob_{0, \beta_{0}}(\infected \in X(\beta_{0}, \beta_{1})) \leq \alpha\), and for any \(t \in \support(S(\cdot; \beta_{0}, \beta_{1}))\) such that \(t < t_{\alpha, \beta_{0}}(\beta_{0}, \beta_{1})\), we have
\[
\prob_{0, \beta_{0}}(S(\infected; \beta_{0}, \beta_{1}) \geq t)
\geq 
\alpha.
\]

For simplicity, let \(\Delta = L_{0}|\beta_{0} - \beta_{0}'| + L_{1}|\beta_{1} - \beta_{1}'|\).
Next, we prove the upper and lower bounds on \(t_{\alpha, \beta_{0}}(\beta_{0}, \beta_{1})\) separately.
Since \(S_{i}\) is \(L_{i}\)-Lipschitz,
we see that
\[
S(I; \beta_{0}', \beta_{1}')
\leq 
S(I; \beta_{0}, \beta_{1}) + \Delta
\]
for each \(I\) in \(X(\beta_{0}, \beta_{1})\).
Thus, the probability that 
\[
S(\infected; \beta_{0}', \beta_{1}') \geq t_{\alpha, \beta_{0}}(\beta_{0}, \beta_{1}) + \Delta
\]
is at most \(\alpha\).
Therefore, we have 
\[
t_{\alpha, \beta_{0}}(\beta_{0}', \beta_{1}') 
\leq 
t_{\alpha, \beta_{0}}(\beta_{0}, \beta_{1}) + \Delta.
\]

Next, we prove the lower bound.
Analogously, we see that
\[
S(I; \beta_{0}', \beta_{1}')
\geq 
S(I; \beta_{0}, \beta_{1}) - \Delta
\]
for each \(I\) in \(X(\beta_{0}, \beta_{1})\).
Thus, the probability that 
\[
S(\infected; \beta_{0}', \beta_{1}') \geq t_{\alpha, \beta_{0}}(\beta_{0}, \beta_{1}) - \Delta
\]
is at least \(\alpha\).
Therefore, we have 
\[
t_{\alpha, \beta_{0}}(\beta_{0}', \beta_{1}') 
\geq 
t_{\alpha, \beta_{0}}(\beta_{0}, \beta_{1}) - \Delta.
\]
Putting everything together, we have
\[
\left|t_{\alpha, \beta_{0}}(\beta_{0}', \beta_{1}') 
-
t_{\alpha, \beta_{0}}(\beta_{0}, \beta_{1})
\right|
\leq 
\Delta,
\]
which proves the desired Lipschitz result.
\end{proof}

\begin{lemma}
Define the event
\[
E 
:=
\left\{
\tilde{t}_{\alpha - \gamma - \epsilon}(\beta_{0}', \beta_{1}')
<
t_{\alpha - \gamma, \beta_{0}}(\beta_{0}', \beta_{1}')
\right\},
\]
and the event
\[
E'
:=
\left\{
\hat{t}_{\alpha - \gamma - \epsilon}(\beta_{0}, \beta_{1})
<
t_{\alpha - \gamma, \beta_{0}}(\beta_{0}, \beta_{1})
\right\},
\]
for \(\beta_{0} \in \paramset_{0, D}\), \(\beta_{1} \in \paramset_{1, D}\), \(\beta_{0} \in \paramset_{0}\), and \(\beta_{1} \in \paramset_{1}\).
Then we have the inclusion \(E \subseteq E'\).
\label{lemmaThresholdSetInclusion}
\end{lemma}

\begin{proof}
Suppose that \(E\), holds.
As before, let \(\Delta = L_{0}|\beta_{0} - \beta_{0}'| + L_{1}|\beta_{1} - \beta_{1}'|\).
Then by the definition of \(\tilde{t}\), the assumption that \(E\) occurs, and Lemma~\ref{lemmaLipschitzThreshold}, we have
\begin{align*}
& \begin{aligned}
\hat{t}_{\alpha - \gamma - \epsilon}(\beta_{0}, \beta_{1}) + \Delta
&=
\tilde{t}_{\alpha - \gamma - \epsilon}(\beta_{0}', \beta_{1}') 
<
t_{\alpha - \gamma, \beta_{0}}(\beta_{0}', \beta_{1}') 
\leq 
t_{\alpha - \gamma, \beta_{0}}(\beta_{0}, \beta_{1})
+
\Delta.
\end{aligned}
\end{align*}
Subtracting \(\Delta\), we have
\[
\hat{t}_{\alpha - \gamma - \epsilon}(\beta_{0}, \beta_{1})
<
t_{\alpha - \gamma, \beta_{0}}(\beta_{0}, \beta_{1}),
\]
which completes the proof.
\end{proof}

\begin{proof}[Proof of Proposition~\ref{propPDSLipschitzAS}]

The proof is similar to the proof of Proposition~\ref{propTwoStatistic}. We write the Type I error as
\begin{align*}
& \begin{aligned}
\typeone
&:=
\sup_{\beta_{0} \in \paramset_{0}, \beta_{1} \in \paramset_{1}} 
\prob_{0, \beta_{0}}\left(
S(\infected; \beta_{0}, \beta_{1})
\geq 
\tilde{t}_{\alpha - \gamma - \epsilon}(\beta_{0}, \beta_{1})
\right)
\\
&\leq 
\sup_{\beta_{0} \in \paramset_{0}, \beta_{1} \in \paramset_{1}}
\sup_{\beta_{0}' \in F_{0, D}^{-1}(F_{0, D}(\beta_{0}))}
\prob_{0, \beta_{0}}\left(
S(\infected; \beta_{0}', \beta_{1}) \geq \tilde{t}_{\alpha - \gamma - \epsilon}(\beta_{0}', \beta_{1})
\right),
\end{aligned}
\end{align*}
where the discretization function is
\[
F_{D}(\beta_{0}, \beta_{1}) = (F_{0, D}(\beta_{0}), F_{1, D}(\beta_{1})).
\]
We then use Proposition~\ref{propDiscretizationMain} and Lemma~\ref{lemmaTypeIerrors} to obtain
\begin{align*}
& \begin{aligned}
\typeone 
&\leq 
\sup_{\beta_{0} \in \paramset_{0, D}, \beta_{1} \in \paramset_{1}}
\sup_{\beta_{0}' \in F_{0, D}^{-1}(\beta_{0})}
\prob_{0, \beta_{0}}\left(
S(\infected; \beta_{0}', \beta_{1}) \geq 
\tilde{t}_{\alpha - \gamma - \epsilon}(\beta_{0}', \beta_{1})
\right)
+ \delta \\
&\leq 
\sup_{\beta_{0} \in \paramset_{0, D}, \beta_{1} \in \paramset_{1}}
\sup_{\beta_{0}' \in F_{0, D}^{-1}(\beta_{0})}
\prob_{0, \beta_{0}}\left(
S(\infected; \beta_{0}', \beta_{1}) \geq t_{\alpha - \gamma}(\beta_{0}', \beta_{1}) 
\right)
\\&\qquad 
+
\prob_{0, \beta_{0}}\left(
\tilde{t}_{\alpha - \gamma - \epsilon}(\beta_{0}', \beta_{1})
<
t_{\alpha - \gamma}(\beta_{0}', \beta_{1})
\right)
+ \delta \\ 
&\leq 
\alpha - \gamma + \delta
+ 
\sup_{\beta_{0} \in \paramset_{0, D}, \beta_{1} \in \paramset_{1}}
\sup_{\beta_{0}' \in F_{0, D}^{-1}(\beta_{0})}
\prob_{0, \beta_{0}}\left(
\tilde{t}_{\alpha - \gamma - \epsilon}(\beta_{0}', \beta_{1})
<
t_{\alpha - \gamma}(\beta_{0}', \beta_{1})
\right),
\end{aligned}
\end{align*}
where we simplified the inner supremum using 
\(F_{0, D}(\beta_{0}) = \beta_{0}\) for \(\beta_{0}\) in \(\paramset_{0, D}\).
The second additional step that we need to take is to get this final probability back to a finite set of \(\beta_{0}\) and \(\beta_{1}\).
By Lemma~\ref{lemmaThresholdSetInclusion} and by Lemma~\ref{lemmaQuantileConcentration}, we obtain
\begin{align*}
& \begin{aligned}
\typeone
&\leq 
\alpha - \gamma + \delta 
+
\sup_{\beta_{0} \in \paramset_{0, D}, \beta_{1} \in \paramset_{1, D}}
\prob_{0, \beta_{0}}\left(
\hat{t}_{\alpha - \gamma - \epsilon}(\beta_{0}, \beta_{1})
<
t_{\alpha - \gamma}(\beta_{0}, \beta_{1})
\right)
\leq
\alpha - \gamma + \delta + \xi 
=
\alpha,
\end{aligned}
\end{align*}
and this completes the proof.
\end{proof}


\subsection{Ising model}
\label{sec:disc_technical:isingModel}
In this section, we state and prove a valid discretization for the Ising model.
Our discretization only applies for bounded parameter sets \(\paramset \subseteq \completespace_{\ising, R} =  [0, R]\).
\begin{definition}
Let 
\(H = \max_{J \in \uncensored(\infspace_{k, c})} |\energy(J)|\).
Define 
\[
N = \left\lceil \frac{RH}{\delta^{2}} \right\rceil + 1.
\]
We define the discretization 
\[
\completespace_{\ising, R, D}
:=
\left\{\frac{(i - 1) \delta^{2}}{H} : i = 1, \ldots, N 
\right\},
\]
and associated discretization function
\[
F_{\ising, R, D}(\beta) 
=
\argmin_{\beta' \in \completespace_{\ising, R, D}} |\beta - \beta'|.
\]
\label{defIsingDiscretization}
\end{definition}

\begin{proposition}
Let \(\paramset_{\ising}\) be a parameter set bounded by \(R\).
The discretization \(\completespace_{\ising, R, D}\) with discretization function \(F_{\ising, D}\) satisfies equation~\eqref{eqnDiscretizationFunction}.
\label{propIsingDiscretization}
\end{proposition}

We now embark on proving this proposition.
Note that because we lack a simple stochastic process representation here, we need to use different techniques than for the stochastic spreading model. 
First, we consider a helpful lemma.


\begin{lemma}
Let \(S: \infspace_{k, c} \to \reals\) be a statistic.
Consider the generalized Ising model with energy function \(\energy\).
Let \(H = \max_{J \in \uncensored(\infspace_{k, c})} |\energy(J)|\).
Then
\[
\tvdist\left(\prob_{\ising, \beta}, \prob_{\ising, \beta'}\right)
\leq 
\sqrt{2H |\beta - \beta'| }.
\]
\label{lemmaIsingSmallBeta}
\end{lemma}


\begin{proof}
We start by using Pinsker's inequality to obtain
\begin{align*}
&\begin{aligned}
\tvdist\left(\prob_{\ising, \beta}, \prob_{\ising, \beta'}\right)
&\leq
\sqrt{\frac{1}{2}\kldiv\left(\prob_{\ising, \beta}, \prob_{\ising, \beta'}
\right)
}.
\end{aligned}
\end{align*}
Thus, it only remains to compute the Kullback-Leibler divergence.
Applying the log-sum inequality of Lemma~\ref{lemmaLogSumInequality}, we have
\begin{align*}
&\begin{aligned}
\kldiv\left(
\prob_{\ising, \beta},
\prob_{\ising, \beta'}
\right)
&=
\sum_{J \in \infspace_{k, c}}
\prob_{\ising, \beta}(\infected = J)
\log
\frac{\prob_{\ising, \beta}(\infected = J)}{\prob_{\ising, \beta'}(\infected = J)} \\
&=
\sum_{J \in \infspace_{k, c}}
\frac{1}{Z(\beta)}
\sum_{J' \in \uncensored(J)}
\frac{1}{Z_{\ising}(\beta)} \exp(-\beta \energy(J'))
\\&\qquad
\times
\left(
\log \frac{Z(\beta')}{Z(\beta)} 
+
\log 
\frac{\sum_{J \in \uncensored(J)} \frac{1}{Z_{\ising}(\beta)} \exp(-\beta \energy(J'))}{\sum_{J' \in \uncensored(J)} \frac{1}{Z_{\ising}(\beta')} \exp(-\beta' \energy(J'))}
\right) \\
&\leq 
\log \frac{Z(\beta')}{Z(\beta)}
+
\sum_{J \in \infspace_{k, c}} \sum_{J' \in \uncensored(J)}
\frac{1}{Z(\beta)} \cdot \frac{1}{Z_{\ising}(\beta)}
\exp\left(-\beta \energy(J')\right)
\\ &\qquad
\times
\left(
\log \frac{Z_{\ising}(\beta')}{Z_{\ising}(\beta)}
+
\log \frac{\exp(-\beta \energy(J'))}{\exp(-\beta' \energy(J'))}
\right) \\
&\leq 
\log \frac{Z(\beta')}{Z(\beta)}
\\&\qquad +
\sum_{J \in \infspace_{k, c}} \sum_{J' \in \uncensored(J)}
\frac{1}{Z(\beta)} \cdot \frac{\exp(-\beta \energy(J'))}{Z_{\ising, k + c'}(\beta)}
\log \frac{Z_{\ising, k + c'}(\beta')}{Z_{\ising, k + c'}(\beta)}
\\&\qquad +
\sum_{J \in \infspace_{k, c}} \sum_{J' \in \uncensored(J)}
\frac{1}{Z(\beta)} \cdot \frac{\exp(-\beta \energy(J'))}{Z_{\ising}(\beta)}
\energy(J') (\beta' - \beta))
\\
&=:
S_{1} + S_{2} + (\beta' - \beta) \expect _{\ising, \beta} \energy(\infected).
\end{aligned}
\end{align*}
Note that the final term can be bounded by \(|\beta' - \beta| H\).
Now, we need to analyze the first two terms on the right hand side above, which we denote \(S_{1}\) and \(S_{2}\).
Starting with the latter, we have
\begin{align*}
&\begin{aligned}
S_{2}
&=
\sum_{J \in \infspace_{k, c}} \sum_{J' \in \uncensored(J)}
\frac{1}{Z(\beta)} \cdot \frac{1}{Z_{\ising, k + c'}(\beta)}
\exp(-\beta \energy(J'))
\log 
\frac{\sum_{J \in \infspace_{k + c', 0}} \exp(-\beta' \energy(J))}{\sum_{J \in \infspace_{k + c', 0}} \exp(-\beta \energy(J))} \\
&\leq 
\sum_{J \in \infspace_{k, c}} \sum_{J' \in \uncensored(J)}
\frac{1}{Z(\beta)} \cdot \frac{1}{Z_{\ising, k + c'}(\beta)}
\exp(-\beta \energy(J'))
\\&\qquad \times
\frac{1}{Z_{\ising}(\beta')}
\sum_{J \in \infspace_{k + c', 0}}
\exp(-\beta' \energy(J))
\log 
\frac{\exp(-\beta' \energy(J))}{\exp(-\beta \energy(J))} \\
&=
\sum_{J \in \infspace_{k, c}} \sum_{J' \in \uncensored(J)}
\frac{1}{Z(\beta)} \cdot \frac{1}{Z_{\ising, k + c'}(\beta)}
\exp(-\beta \energy(J')) 
\\  &\qquad\times 
\frac{1}{Z_{\ising}}
\sum_{J \in \infspace_{k + c', 0}} \exp(-\beta' \energy(J))
\energy(J)(\beta - \beta') \\
&\leq 
\sum_{J \in \infspace_{k, c}} \sum_{J' \in \uncensored(J)}
\frac{1}{Z(\beta)} \cdot \frac{1}{Z_{\ising, k + c'}(\beta)}
\exp(-\beta \energy(J'))
|\beta - \beta'| H \\
&=
|\beta - \beta'| H.
\end{aligned}
\end{align*}

Now, we can compute \(S_{1}\).
We have
\begin{align*}
&\begin{aligned}
S_{1}
&=
\log \frac{Z(\beta')}{Z(\beta)} \\
&\leq 
\frac{1}{Z(\beta')} \sum_{J \in \infspace_{k, c}}
\sum_{J' \in \uncensored(J)} \frac{1}{Z_{\ising, k + c'}(\beta')}
\exp\left(-\beta' \energy(J')\right) 
\\ & \qquad \times
\left(
\log \frac{Z_{\ising, k + c'}(\beta)}{Z_{\ising, k + c'}(\beta')}
+
\log \frac{\exp(-\beta' \energy(J'))}{\exp(-\beta \energy(J'))}
\right) \\
&\leq 
|\beta - \beta'| H
+
(\beta - \beta')
\expect_{\ising, \beta'} \energy(\infected) \\
&\leq 
2 |\beta - \beta'| H.
\end{aligned}
\end{align*}
Putting everything together completes the proof.
\end{proof}

\begin{proof}[Proof of Proposition~\ref{propIsingDiscretization}]
Let \(\beta\) in \(\completespace_{\ising}\) be given, and define \(\beta' = F_{\ising, D}(\beta)\).
Thus, by our discretization, we have 
\[
|\beta - \beta'|
\leq 
\frac{\delta^{2}}{2H}.
\]
Plugging this into the result of Lemma~\ref{lemmaIsingSmallBeta} completes the proof.
\end{proof}

\subsection{Additional general graphs proofs}
\label{sec:disc_technical:additionalProofs}
In this section, we have proofs for the supporting lemmas used in Section~\ref{app:general_graphs:proofs}.


\begin{proof}[Proof of Lemma~\ref{lemmaTypeIerrors}]
We use straightforward union bounds to obtain
\begin{align*}
& \begin{aligned}
\prob_{0, \beta}\left(S(\infected) > \hat{t}\right)
&\leq 
\prob_{0, \beta}\left(\left(\left\{S(\infected) > \hat{t}\right\} \cap \left\{\hat{t} \geq t\right\} \right)
\cup 
\left(
\left\{S(\infected) > \hat{t}\right\} \cap \left\{\hat{t} < t\right\}
\right)
\right) \\
&\leq 
\prob_{0, \beta}\left(\left\{S(\infected) > \hat{t}\right\} \cap \left\{\hat{t} \geq t\right\} \right)
+
\prob_{0, \beta}
\left(
\left\{S(\infected) > \hat{t}\right\} \cap \left\{\hat{t} < t\right\}
\right) \\ 
&\leq 
\prob_{0, \beta}\left(S(\infected) >  t \right)
+
\prob_{0, \beta}\left(
\hat{t} < t
\right).
\end{aligned}
\end{align*}
This completes the proof.
\end{proof}

\begin{proof}[Proof of Lemma~\ref{lemmaQuantileConcentration}]
The goal of the proof is to apply Bernstein's inequality, given as Lemma~\ref{lemmaBernsteinsInequality}.
First, by inequality~\eqref{eqnGeqThreshold}, we have
\[
p 
:= 
\prob_{0, \beta}(S(\infected_{i}) \geq t_{\alpha}) > \alpha.
\]
Thus, we have
\begin{align*}
& \begin{aligned}
\prob_{0, \beta}\left(\hat{t}_{\alpha - \epsilon, \beta} < t_{\alpha}\right)
&=
\prob_{0, \beta}
\left(
\frac{1}{N_{\sims}}
\sum_{i = 1}^{N_{\sims}} \ind\{S(\infected_{i}) \geq t_{\alpha}\} \leq \alpha - \epsilon
\right) \\
&=
\prob_{0, \beta}
\left(
\frac{1}{N_{\sims}}
\sum_{i = 1}^{N_{\sims}} \ind\{S(\infected_{i}) \geq t_{\alpha}\} - p \leq -\epsilon -(p - \alpha)
\right).
\end{aligned}
\end{align*}
Now, by considering \(X_{i} = p - \ind\{S(\infected_{i}) \geq t_{\alpha}\}\), 
we have \(\expect[X_{i}] = 0\) and \(|X_{i}| \leq 1\), and
we also see that \(\expect[X_{i}^{2}] = p(1 - p) \leq 1/4\).
Hence, we can apply Bernstein's inequality to obtain
\begin{align*}
& \begin{aligned}
\prob_{0, \beta}\left(\hat{t}_{\alpha - \epsilon, \beta} < t_{\alpha}\right)
&\leq 
\exp\left(-\frac{N_{\sims} (\epsilon + (p - \alpha))^{2}}{1/2 + 2 (\epsilon + (p - \alpha)) / 3}
\right).
\end{aligned}
\end{align*}
Thus, in order for this to be less than or equal to \(\delta\), we require
\begin{align*}
& \begin{aligned}
N_{\sims}
&\geq 
\left(\frac{1}{2(\epsilon + p -\alpha)^{2}} + \frac{2}{3(\epsilon + p -\alpha) }\right) 
\log \frac{1}{\xi}
\geq 
\left(\frac{1}{2\epsilon^{2}} +\frac{2}{3\epsilon}\right)
\log \frac{1}{\xi},
\end{aligned}
\end{align*}
which completes the proof.
\end{proof}

\begin{remark}
From the proof of Lemma~\ref{lemmaQuantileConcentration}, we see that if \(p = \prob_{0, \beta}(S(\infected_{i}) \geq t_{\alpha})\) is strictly greater than \(\alpha\), then using \(N_{\sims}\) simulations actually leads to a tighter bound on the error probability than \(\xi\).
Alternatively, we could achieve the desired error probability with fewer simulations, but since \(p\) is unknown, we settle for a coarser upper bound on the number of simulations required.
\end{remark}


\section{Auxiliary lemmas}
\label{AppAux}

For our risk bounds, we need a standard concentration result: 

\begin{lemma}
\label{LemBdDiff}
Suppose \(\{M_{t}\}_{t = 1}^{T}\) is a martingale with respect to some filtration and the differences
\(M_{t} - M_{t - 1}\) have expectation \(0\) and are bounded by \(c_{t}\).
Then
\begin{equation*}
\prob\left(M_{T} \ge t\right) \le \exp\left(-\frac{2t^2}{\sum_{t = 1}^{T} c_{t}^2}\right),
\end{equation*}
and
\begin{equation*}
\prob\left(M_{T} \le -t\right) \le \exp\left(-\frac{2t^2}{\sum_{t=1}^n c_t^2}\right).
\end{equation*}
\end{lemma}

We will apply Lemma~\ref{LemBdDiff} to the statistic \(W_{1}\) in the following manner:
Let \(\pathvar_{1:t}\) denote the the first \(t\) infected vertices, and define \(W_{1}(\pathvar_{1:t})\) to be the edges-within statistic for the infection \(J\) in \(\infspace_{t, 0}\) corresponding to the partial path.
Let \(\field_{t} = \sigma(\pathvar_{1:t})\) be the sigma-field of infections up to time \(t\).
Finally, define the martingale
\[
M_{t}
=
W_{1}(\pathvar_{1:t}) - \expect[W_{1}(\pathvar_{1:t})]
=
\sum_{s = 1}^{t} 
\left(W_{1}(\pathvar_{1:s}) - W_{1}(\pathvar_{1:s - 1})\right)
-
\expect\left[W_{1}(\pathvar_{1:s}) - W_{1}(\pathvar_{1:s - 1})\right],
\]
where \(W_{1}(\pathvar_{1:0}) = 0\). Since $W_{1}(\pathvar_{1:s}) - W_{1}(\pathvar_{1:s - 1}) \in [0, D]$ (where $D$ is the maximum degree of the graph), the martingale differences satisfy 
\[
|\Delta_{t}|
= 
|M_{t} - M_{t - 1}|
=
\Big|\left(W_{1}(\pathvar_{1:t}) - W_{1}(\pathvar_{1:t - 1})\right)
-
\expect\left[W_{1}(\pathvar_{1:t}) - W_{1}(\pathvar_{1:t - 1})\right]\Big| \le D.
\]
Thus, Lemma~\ref{LemBdDiff} applies with $c_t = D$.

Next, we provide a data-processing inequality of \cite{wu2017}.
Let \(D_{f}(P, Q)\) denote the \(f\)-divergence of \(P\) and \(Q\), which is defined as
\[
D_{f}(P, Q)
:=
\expect _{D} f\left(\frac{P}{Q}\right),
\]
for an \(f\) satisfying specific properties.
For our purposes, we are interested in \(f(t) = t \log t\) and \(f(t) = |t - 1| / 2\), which correspond to the Kullback-Leibler divergence and total variation distance, respectively.

\begin{lemma}
Let \(X\) be a random variable with distributions \(P_{X}\) and \(Q_{X}\).
Let \(Y = g(X)\) for some function \(g\), and define the induced distributions of \(Y\) to be \(P_{X}\) and \(Q_{Y}\).
Then we have
\begin{align*}
&\begin{aligned}
D_{f}(P_{Y}, Q_{Y})
&\leq 
D_{f}(P_{X}, Q_{X}).
\end{aligned}
\end{align*}
\label{lemmafDivDPI}
\end{lemma}


\begin{proof}
For the proof, we use \(p_{x}\), \(p_{y}\), and \(p_{x|y}\) to refer to probability masses for \(X = x\), \(Y = y\), and \(X = x\) given \(Y = y\).
Additionally, we define the analogous quantities for \(q\).

Since \(x\) fully determines \(y\), we have
\begin{align*}
&\begin{aligned}
D_{f}(P_{X}, Q_{X})
&=
\sum_{x} q_{x} f\left(\frac{p_{x}}{q_{x}}\right) \\
&=
\sum_{x, y} q_{x, y} f\left(\frac{p_{x, y}}{q_{x, y}}\right) \\
&=
\sum_{y} q_{y}
\sum_{x} q_{x | y} f\left(\frac{p_{xy}}{q_{xy}}\right) \\
&=
\expect_{Q_{Y}} \expect_{Q_{X | Y}} f\left(\frac{P_{X Y}}{Q_{X Y}}\right).
\end{aligned}
\end{align*}
Hence, we can use Jensen's inequality, which yields
\begin{align*}
&\begin{aligned}
D_{f}(P_{X}, Q_{X})
&\geq 
\expect_{Q_{Y}}  f\left(\expect_{Q_{X | Y}} \frac{P_{X Y}}{Q_{X Y}}\right) \\
&=
\sum_{y} 
q_{y} f\left(\sum_{x} q_{x | y} \frac{p_{xy}}{q_{xy}}\right)
\\
&=
\sum_{y} q_{y}
f\left(\sum_{x} \frac{p_{y} p_{x | y}}{p_{y}} \right)
\\
&=
\sum_{y} q_{y}
f\left(\frac{p_{y}}{p_{y}} \sum_{x} p_{x | y} \right) \\
&=
\sum_{y} q_{y}
f\left(\frac{p_{y}}{p_{y}}  \right)  \\
&=
D_{f}(P_{Y}, Q_{Y}).
\end{aligned}
\end{align*}
This completes the proof.
\end{proof}

\begin{lemma}
Let \(X\) and \(X'\) be jointly-defined random variables on the same finite space \(\xspace\) with respect to the measure \(\prob\). 
Let the respective marginal measures be \(\prob_{\beta}\) and \(\prob_{\beta'}\), i.e., \(\prob_{\beta}(X = J)\) and \(\prob_{\beta'}(X' = J)\) are the marginal distributions.
Then we have
\[
\tvdist\left(\prob_{\beta}, \prob_{\beta'}\right)
\leq 
\prob\left(X \neq X'\right).
\]
\label{lemmaTVCoupling}
\end{lemma}

\begin{proof}[Proof of Lemma~\ref{lemmaTVCoupling}]
Let \(E\) be any event.
Then \(E\) has the form 
\[
E
=
\{\pathvar \in E'\},
\] 
for some subset \(E'\) of the space of possible outcomes \(\xspace\).
By definition, we have
\begin{align*}
&\begin{aligned}
\tvdist\left(\prob_{\beta}, \prob_{\beta'}\right)
&=
\sup_{E} |\prob_{\beta}\left(E\right) - \prob_{\beta'}\left(E\right)| 
=
\sup_{E' \in \xspace}
|\prob_{\beta}\left(X \subseteq E'\right) - \prob_{\beta'}\left(X' \in E'\right)|.
\end{aligned}
\end{align*}
Note that since \(\xspace\) is finite, this supremum is achieved for some \(E' = E^{*}\).

With this setup, we have
\begin{align*}
&\begin{aligned}
\tvdist\left(\prob_{\beta}, \prob_{\beta'}\right)
&=
| \prob_{0, \beta}(X \in E^{*}) 
- \prob_{0, \infty}(X \in E^{*})|  \\
&\leq 
|\expect \ind\{X \in E^{*}\} - \expect \ind\{X' \in E^{*}\}| \\
&\leq 
\expect\left| \ind\{X \in E^{*}\} - \ind\{X \in E^{*}\}\right| \\
&\leq 
\prob(X \neq X').
\end{aligned}
\end{align*}
This completes the proof.
\end{proof}

Next, we have Pinsker's inequality, which may be found in \cite{tsybakov2009}.

\begin{lemma}[Pinsker's inequality]
Let \(P\) and \(Q\) be two probability measures on the same discrete space.
Then we have
\[
\tvdist(P, Q)
\leq 
\sqrt{\frac{1}{2} \kldiv(P, Q)}.
\]
\label{lemmaPinskersInequality}
\end{lemma}

The next lemma is the log sum inequality, which is given as a simple consequence of Jensen's inequality on \(f(x) = x \log x\) in \cite{cover2012}.

\begin{lemma}
Let \(a_{i}\) and \(b_{i}\) be nonnegative numbers for \(i = 1, \ldots, n\).
Then we have the inequality
\[
\left(\sum_{i = 1}^{n} a_{i}\right) \log\frac{\sum_{i = 1}^{n} a_{i}}{\sum_{i = 1}^{n} b_{i}}
\leq
\sum_{i = 1}^{n} a_{i} \log \frac{a_{i}}{b_{i}}.
\]
\label{lemmaLogSumInequality}
\end{lemma}

Next, we have a standard concentration result, which may be found in \cite{boucheron2013}.

\begin{lemma}[Bernstein's Inequality]
Let \(X_{i}\) be centered random variables, i.e., \(\expect X_{i} = 0\), such that \(|X_{i}| \leq M\) almost surely.
Then for any \(t > 0\), we have
\[
\prob\left(\sum_{i = 1}^{n} X_{i} > t\right)
\leq 
\exp
\left(
-\frac{t^{2}}{2 \sum_{i = 1}^{n} \expect[X_{i}^{2}] + 2Mt/3}
\right).
\]
\label{lemmaBernsteinsInequality}
\end{lemma}

\end{document}